\RequirePackage{fix-cm}
\documentclass[11pt]{amsart}
\usepackage{graphicx, amssymb, amsmath,amsthm}

\usepackage{enumerate}
\usepackage{enumitem}
\usepackage{graphicx}
\usepackage{geometry}
\usepackage{mathtools}
\usepackage{color}
\usepackage[mathscr]{euscript}
\usepackage{amssymb}
\usepackage{bbm}
\usepackage[colorlinks,
            linkcolor=blue,
            anchorcolor=blue,
            hypertexnames=false, 
            citecolor=blue]{hyperref}

\newtheorem{thm}{Theorem}

\newcommand{\ds}{\displaystyle}

\def \a{\alpha} \def \b{\beta} \def \g{\gamma} \def \d{\delta}
\def \t{\theta}   
\def \s{\sigma} \def \l{\lambda} \def \z{\zeta} \def \o{\omega}
\def \O{\Omega} \def\E{\mathbf{E}} \def \F{\mathcal{F}} \def\P{\mathcal{P}}
 \def \r{\rho}
\def \k{\kappa}  \def \G{\Gamma}

\def \L{\Lambda}
\def\R{\mathbb{R}}
\def\To{\mathbb{T}}
\def \ra{\rightarrow}
\def \id{\bold 1}

\def\mkm{\mathfrak{m}}
\def\ob{\mathring{B}}
\def\mct{\mathcal{T}}
\def\mcr{\mathcal{R}}
\def\mcd{\mathcal{D}}
\def\mbl{\mathbb{L}}

\def\la{\left\langle}\def\ra{\right\rangle}
\def\lp{\left\{}\def\rp{\right\}}

\def\dh{\|\cdot\|}
\def\H{\mathcal{H}}

\newtheorem{theorem}{Theorem}[section]

\newtheorem{remark}{Remark}[section]

\newtheorem{lemma}[theorem]{Lemma}
\newtheorem{proposition}[theorem]{Proposition}

\newtheorem{corollary}[theorem]{Corollary}

\numberwithin{equation}{section}

\begin{document}

\subjclass{Primary: 60H15, 60F10, 76D05, 37A05} \keywords{Navier-Stokes Equations, periodic Ruelle-Perron-Frobenius theorem, Donsker-Varadhan large deviation principle, degenerate noise, time inhomogeneous Markov processes.}

\author{Rongchang Liu} \address[Rongchang Liu] {  Department of Mathematics\\
University of Arizona\\Tucson, AZ 85721, USA}
 \email[R.~Liu]{lrc666@math.arizona.edu}

\author{Kening Lu} \address[Kening Lu] {  School of Mathematics\\
Sichuan University\\
Chengdu, Sichuan 610064, PR China}
\email[k.~Lu]{klu@math.byu.edu}

\title[Large deviations for 2D stochastic Navier-Stokes Equations driven by a periodic force and a degenerate noise]{Large deviations for 2D stochastic Navier-Stokes Equations driven by a periodic force and a degenerate noise}

\pagestyle{plain}

\begin{abstract} {
We consider  the incompressible 2D Navier-Stokes equations on the torus, driven by a deterministic time periodic force and a noise that is white in time and degenerate in Fourier space. The main result is twofold. 

Firstly, we establish  a Ruelle-Perron-Frobenius type theorem for the time inhomogeneous Feynman-Kac evolution operators with regular  potentials associated with the stochastic Navier-Stokes system. The theorem characterizes asymptotic behaviors of the Feynman-Kac operators in terms of the periodic family of  principal eigenvalues and corresponding unique eigenvectors. The proof involves a time inhomogeneous version of Ruelle's lower bound technique. 

Secondly, utilizing this Ruelle-Perron-Frobenius type theorem and a Kifer's criterion, we establish a Donsker-Varadhan type large deviation principle with a nontrivial good rate function for the occupation measures of the time inhomogeneous solution processes. 
} \end{abstract}

\maketitle

\baselineskip 14pt

\tableofcontents

\section{Introduction}

The 2D Navier-Stokes equations subject to a time periodic force  has been extensively studied over the years, dating back to Serrin \cite{Ser59}, Yudovich  \cite{Yud60} and  many others, see Galdi and  Kyed \cite{GK18} for recent progress. It was proved by Da Prato and Debussche \cite{DD08} that the periodically forced 2D Navier-Stokes system with a non-degenerate noise has an exponentially mixing periodic invariant measure. Furthermore, it was proved recently \cite{LL22} that when the noise is degenerate, the strong law of large numbers and central limit theorem hold.  In particular, the central limit theorem describes the small deviations from the periodic invariant measure. 

\vskip0.05in

The aim of the current paper is to investigate the probabilities of large deviations from the periodic invariant measure, for the periodically forced 2D Navier-Stokes equations with a degenerate white in time noise. To the best of our knowledge, this work seems to be the first one that establishes Donsker-Varadhan type large deviation results for the occupation measures of the time inhomogeneous solution processes solving a given SPDE. 
In what follows, we will briefly describe our results (Theorems \ref{theorem-A} and \ref{theorem-B} below), while leave the technical preliminaries to the next section. 

\vskip0.05in

Consider the incompressible 2D Navier-Stokes equations on the torus $\To^2:=\R^2/2\pi\mathbb{Z}^2$ in the following vorticity form
\begin{align}\label{NS-intro}
dw(t, x) + B(\mathcal{K} w, w) (t, x)d t = \nu\mathrm{\mathrm{\Delta}} w(t, x)dt  + f(t, x)dt + \sum_{i=1}^d g_idW_i(t), \quad t>s, \quad w(s) = w_0,
\end{align}
for $s\in \R, x\in\To^2$ and some integer $d\geq 1$. Here $w$ is the vorticity field, $\mathcal{K}$ is the Biot-Savart integral operator and $B(\mathcal{K} w, w) =(\mathcal{K} w)\cdot\nabla w $ is the nonlinearity. The phase space is chosen as $H:= \left\{w \in \mathrm{L}^2\left(\mathbb{T}^{2}, \mathbb{R}\right): \int_{\mathbb{T}^2} w d x=0\right\}$ whose norm is denoted by $\|\cdot\|$ and the initial $w_0\in H$. We also denote $H_{s}=H^s\cap H$ and the corresponding norm as $\left\|\cdot\right\|_s$ where $H^s$ is the usual Sobolev space of order $s$. The deterministic force  $f$ is $2\pi$-periodic in $t$. The noise $W = (W_1, W_2, \cdots, W_d)$ is a two-sided $\R^d$-valued standard Wiener process over the sample space $(\O, \mathcal{F}, \mathbf{P})$ where $\mathbf{P}$ is the Wiener measure,  and the noise coefficients $\{g_i\}$ are smooth elements of $H$.

\vskip0.05in

We denote the solution to \eqref{NS-intro} by $\Phi_{s, t}(w_0)$. To introduce the Feynman-Kac operators in our time inhomogeneous setting, we utilize the time symbols formulation used in deterministic systems, see Chepyzhov and Vishik \cite{CV02a} for example. To be precise, we consider a family of Navier-Stokes systems indexed by time symbols $h\in\To: = \R/2\pi\mathbb{Z}$, 
\begin{align}\label{NS-Family}
dw(t, x) + B(\mathcal{K} w, w) (t, x)d t = \nu\mathrm{\mathrm{\Delta}} w(t, x)dt  + f(\b_th,  x)dt + \sum_{i=1}^d g_idW_i(t), \, t>s, \, w(s) = w_0,
\end{align}
whose solution is denoted by $\Phi_{s, t, h}(w_0)$. Here $\b_th = h+t$ is the circle rotation on $\To$.
Note that \eqref{NS-Family} when $h=0$ is the original equation \eqref{NS-intro}. 

\vskip0.05in

Now for every potential $V\in C_b(H\times \To)$, the associated Feynman-Kac operator $P_{s, t, h}^V: C_b(H)\to C_b(H)$ corresponding to \eqref{NS-Family} with time symbol $h\in\To$ is given by
\begin{align}\label{intro-Feynman-Kac}
	P_{s, t, h}^V\phi(w) = \E\left[ \exp \left(\int_s^tV(\Phi_{s, r, h}(w), \b_rh)dr\right)\phi(\Phi_{s, t, h}(w))\right], \, \phi\in C_b(H), w\in H. 
\end{align}
It naturally induces by duality an action $P_{s, t, h}^{V*}$ on the space $\mathcal{M}_{+}(H)$ of non-negative finite Borel measures on $H$, which can be regarded as an averaged transfer operator.

\vskip0.05in

Since the paper is working with a degenerate noise, we recall the Lie bracket condition from Hairer and Mattingly \cite{HM11}. Define the set $A_{\infty}$ by setting 
\begin{align}\label{A-infty}
	A_1=\{g_l : 1\leq l \leq d\}, A_{k+1} = A_{k}\cup \{\widetilde{B}(u,v): u, v\in A_{k}\}, A_{\infty} = \overline{\mathrm{span}(\cup_{k\geq 1} A_{k})},
\end{align}
where $\widetilde{B}(u,w)=-B(\mathcal{K}u,w)-B(\mathcal{K}w,u)$ is  the symmetrized nonlinear term. The assumptions for our main results will be the H\"ormander condition $A_{\infty}=H$ and regularity conditions on the force $f$.  

\vskip0.05in

\vskip0.1in
\noindent
\textbf{Ruelle-Perron-Frobenius type theorem for the Feynman-Kac operators}
\vskip0.1in

In the study of an infinite one dimensional lattice gas model, Ruelle \cite{Rue68} proved his Perron-Frobenius type theorem for the corresponding transfer operator with H\"older regular potentials. This theorem is now known as Ruelle-Perron-Frobenius (RPF) theorem and has been widely used and extended in the study of dynamical systems and thermodynamic formalism, see Baladi \cite{Bal00}, Bowen \cite{Bow08},  Climenhaga and Pesin \cite{CP17} and references therein.  Ruelle's ideas 
in \cite{Rue68} were then used by Kuksin and Shirikyan \cite{KS00} to show unique ergodicity and asymptotic stability of Markov operators for a class of dissipative PDEs.
Later the same ideas were applied by Jakšić et al. \cite{JNPS15} to the discrete time Feynman-Kac semigroups associated with the models considered in \cite{KS00} for which an RPF type theorem was proved in a  compact setting. An extension to the non-compact setting was obtained by the same authors in \cite{JNPS18}, which was further extended to continuous time by Martirosyan and Nersesyan \cite{MN18}. All the existing results work with time homogeneous Feynman-Kac semigroups.  

\vskip0.05in

Our Theorem \ref{theorem-A} below gives a version of the RPF type theorem for the time inhomogeneous Feynman-Kac evolution operators of \eqref{NS-intro}, which is effective in establishing LDP for the time inhomogeneous solution processes.  Here the effectiveness means that the eigen-elements are continuous on the time symbols which enables us to establish 
the LDP for the original system \eqref{NS-intro} from Theorem \ref{theorem-A}. Below $\P(H)$ is the space of probability measures on $H$. 

\begin{thm}[RPF Type Theorem]\label{theorem-A}
 Assume $A_{\infty}=H$ and $f\in C(\R, H_2)$ is $2\pi$-periodic and Lipschitz. For each potential $V\in C_b^1(H\times\To)$ there exists a unique family of eigen-triple of the Feynman-Kac operators \eqref{intro-Feynman-Kac} associated with the Navier-Stokes family \eqref{NS-Family}, 
\begin{align*}
 	(\G^V, \l^V, F^V)\in C(\To, \P(H))\times C(\To,\R)\times C(H\times\To),
\end{align*}
with $F^V$ strictly positive and normalized by $\G^V$ such that the following properties hold:
\begin{enumerate}
 	\item Eigenmeasure/eigenfunction property with positive eigenvalue $\l_{t, h, V}=\exp\left(\int_0^t\l^V(\b_rh)dr\right)$, 
\begin{align}\label{Thm-A-eigen-property}
\begin{split}
	&P_{0, t, h}^{V*}\G^V(h) = \l_{t, h, V}\G^V(\b_th), \qquad \text{(eigenmeasure property)} \\
	&P_{0, t, h}^{V}F^V(\cdot, \b_th)(w) = \l_{t, h, V}F^V(\cdot, h), \quad \text{(eigenfunction property)} 
\end{split}
\end{align}
for all $h\in\To$ and $t\geq 0$. 
\item Forward and pullback stability: for any $h\in\To$, $\phi\in C_b(H)$ and $w\in H$, 
\begin{align}\label{Thm-A-RPF-Stablity}
\begin{split}
	&\lim_{t\to\infty}\left|\l^{-1}_{t, h, V}P_{0, t, h}^V\phi(w) - \la \G^V(\b_th), \phi\ra F^V(w, h)\right| = 0, \\
	&\lim_{t\to\infty}\left|\l^{-1}_{t, \b_{-t}h, V}P_{0, t, \b_{-t}h}^V\phi(w) - \la \G^V(h), \phi\ra F^V(w, \b_{-t}h)\right| = 0. 
\end{split}
\end{align}
\end{enumerate}
\end{thm}
Theorem \ref{theorem-A} is a simplified version of Theorem \ref{asymptotic-behavior-feynman-kac}, which will be  proved in section \ref{Asymptotic-behavior-discrete-time}. Note that when $h=0$ the theorem gives the RPF type theorem for the Feynman-Kac operators associated with \eqref{NS-intro}. Theorem \ref{theorem-A} is proved by first showing a weak version (without continuous dependence on time symbols)  for the discrete time restrictions and then passing to the whole continuous time.

\vskip0.05in

A major step in proving the weak version is to show the forward $L^1$ stability after one obtains the existence of the eigen-elements. In \cite{KS00,JNPS18}, this was proved by combining the uniform Feller property and irreducibility with Ruelle's lower bound technique, which relies on an $L^1$ non-expansive property. Due to the time inhomogeneity, such non-expansiveness is not available in the current setting. The approach we take in the proof provides a way to tailor the core of Ruelle's lower bound technique in this situation, to obtain the desired forward  $L^1$ stability.  The impact of time inhomogeneity is also reflected throughout the process. 

\vskip0.05in

The passage from the weak version to the whole continuous time  is accomplished through the eigen-properties and uniqueness. Then we establish the continuity of eigen-triple and the representation of the eigenvalues through a common $\l^V$. 
The continuity of the eigenmeasure (and eigenvalue) is obtained by controlling the lower bound of the evaluation of eigenmeasure along the eigenfunction.
The continuity of the eigenfunction is proved by combining the eigenfunction property \eqref{Thm-A-eigen-property} and the time continuity of the Feynman-Kac family, as well as the dynamical structure of the time inhomogeneity. 

\vskip0.05in

\vskip0.1in
\noindent
\textbf{Donsker-Varadhan type large deviation principle}
\vskip0.1in

For time homogeneous Markov processes,  asymptotic statistical behaviors can be described by
the strong law of large numbers and central limit theorem when there is a mixing invariant measure, see Kuksin and Shirikyan \cite{KS12} for example. 
In particular, the central limit theorem characterizes fluctuations of the occupation measures around the invariant measure. A more detailed analysis enables one to derive a LDP that quantifies the exponentially decaying rates of the probabilities of large deviations from the invariant measure,
see for example the pioneering works of Donsker and Varadhan \cite{DV01,DV02}.

\vskip0.05in

For time inhomogeneous Markov processes, the asymptotic statistical profile depends on the time inhomogeneity. 
It was proved recently \cite{LL22} that the statistical dynamics of the time inhomogeneous solution processes for the Navier-Stokes equations \eqref{NS-intro} is characterized by an exponential mixing periodic invariant measure $\G\in C(\To, \P(H))$. In addition, one has the following strong law of large numbers and central limit theorem:
\begin{align*}
	&\lim_{t\to\infty}\frac1t\int_0^t\big(\phi(\Phi_{0, r}) - \la\G(\b_r0), \phi \ra \big)dr\to 0, \, \text{ almost surely, }\\
	&\lim_{t\to\infty}\frac{1}{\sqrt{t}}\int_0^t\big(\phi(\Phi_{0, r}) -  \la\G(\b_r0), \phi \ra \big)dr = N(0, \s^2), \, \text{ in distribution, }
\end{align*}
where recall that $\b_r$ is the circle rotation \eqref{NS-Family}  and $\Phi_{0, r}$ is the solution to \eqref{NS-intro}. Here $N(0, \s^2)$ is a centered normal random variable with variance $\s^2$. The central limit theorem describes the distribution of the fluctuations around the periodic invariant measure. 

\vskip0.05in 

Our Theorem \ref{theorem-B} below gives quantitative information on the probabilities of large deviations from the periodic invariant measure. Denote the occupation measure by
\begin{align*}
	\z_{t} &= \frac1t\int_{0}^t\d_{(\Phi_{0, r}, \b_r0)}dr, 
\end{align*}
which is a family of  random probability measures on $H\times\To$. In the following theorem $\P(H\times\To)$ is the space of probability measures on $H\times\To$ with the topology of weak convergence.

\begin{thm}[Large Deviation Principle]\label{theorem-B}
 Assume $A_{\infty}=H$ and $f\in C(\R, H_2)$ is $2\pi$-periodic and Lipschitz. Then there is a non-trivial good rate function $I: \P(H\times\To)\to[0, \infty]$ such that for every initial vorticity $w\in H$, we have 
\begin{align*}
	-\inf_{\mu\in \mathring{B}}I(\mu)\leq \liminf_{t\to\infty}\frac1t\log\mathbf{P}_{w}\{\z_{t}\in B\}\leq \limsup_{t\to\infty}\frac1t\log \mathbf{P}_{w}\{\z_{t}\in B\}\leq -\inf_{\mu\in \overline{B}}I(\mu),
\end{align*}
for any Borel subset $B$ in $\P(H\times\To)$, with $\mathring{B}$ and $ \overline{B}$ its interior and closure. 
\end{thm}
Theorem \ref{theorem-B} is a simplified version of Theorem \ref{LDP}, which will be  proved in section \ref{Proof-main-theorems}. It immediately implies probability bounds for large deviations from the unique periodic invariant measure $\G$. Indeed, for any Borel set $E\subset \R$ and observable function $\phi\in C_b(H)$, define the Borel subset $B$ of $\P(H\times\To)$, 
\[B:= \lp\mu\in \P(H\times\To): \la \mu, \phi_{\G}\ra \in E\rp,\]
where $\phi_{\G}(w, h) = \phi(w) - \la\G(h),\phi\ra$ and $\la \mu, \phi\ra$ denotes the integral of a function $\phi$ with respect to a measure $\mu$. Note that 
\begin{align*}
	\z_{t}\in B\text{ is equivalent to }  \la \z_{t}, \phi_{\G}\ra=\frac{1}{t}\int_{0}^{t}\big(\phi(\Phi_{0, r}) - \la \G(\b_r0),\phi\ra\big)dr \in E.
\end{align*}
Thus by Theorem \ref{theorem-B} one obtains when $t\to\infty$,
\begin{align*}
	e^{-tL^-}\lesssim \mathbf{P}_{w}\lp\frac{1}{t}\int_{0}^{t}\big(\phi(\Phi_{0, r}) - \la \G(\b_r0),\phi\ra\big)dr\in E\rp \lesssim e^{-tL^+}, 
\end{align*}
where $L^{\pm}=L^{\pm}(\phi_{\G}, E)\geq 0$ are constants determined by the rate function. 

\vskip0.05in

The theorem is proved through a combination of the RPF type Theorem \ref{theorem-A} and a Kifer's criterion established by Jakšić et al. \cite{JNPS18}. It is known that existence of the pressure function 
and exponential tightness imply the large deviation upper bound with its Legendre transform as the rate function. Such existence in the time homogeneous case is a direct consequence of the stability of Feynman-Kac semigroups but more effort is needed here. The existence here follows by a combination of the forward stability \eqref{Thm-A-RPF-Stablity}, integral representation for $\l_{t, h, V}$ in \eqref{Thm-A-eigen-property}, and the following Birkhoff ergodic theorem for circle rotation with unique invariant Lebesgue measure $m$, 
\begin{align*}
	\lim_{t\to\infty}\frac1t\int_0^t\l^V(\b_r0)dr = \int_{\To}\l^V(h)m(dh), 
\end{align*}
which is ensured by the continuity of $\l^V$ established in Theorem \ref{theorem-A}. We emphasize that if there is no continuity then Birkhoff's ergodic theorem is not necessarily true for the time symbol $0$ and therefore the existence of the  pressure function for the given system \eqref{NS-intro} would become an issue.

\vskip0.05in

The large deviation lower bound requires additionally the uniqueness of equilibrium states for a rich class of regular potentials. This is usually done by connecting the equilibrium state with the unique invariant measure of an auxiliary Markov semigroup obtained from the Feynman-Kac semigroup by Doob's $h$-transform. 
We follow this approach but additional issues arise due to time inhomogeneity. Firstly, the Doob's $h$-transform of the inhomogeneous Feynman-Kac family is still inhomogeneous. To establish the uniqueness of invariant measure for its homogenization $\mct_t^V$, we make a correspondence between their invariant measures, so that we can apply the uniqueness of eigen-triple in Theorem \ref{theorem-A}.
Secondly, to utilize the Donsker-Varadhan functional of $\mct_t^V$ to establish an effective connection between equilibrium states and invariant measures, we show that the domain of the infinitesimal generator of $\mct_t^V$ is rich enough to be determining. This involves a verification of the Feller property of $\mct_t^V$, where the continuity on time symbols of the eigen-triple plays a crucial role. 

\vskip0.05in

We end the introduction with a brief review on the Donsker-Varadhan type large deviations from a stationary distribution for 2D Navier-Stokes equations. It is worth mentioning that the first such LDP for SPDEs was given by Gourcy \cite{Gou07a,Gou07b}. By assuming that the white in time noise is sufficiently irregular in space, he established the LDP for occupation measures of the stochastic Burgers and Navier-Stokes equations, using a general result of Wu \cite{Wu01} that requires the strong Feller property. Utilizing the approach that  combines the asymptotic behavior of Feynman-Kac semigroups with a Kifer type criterion, Jakšić et al. \cite{JNPS15} established the first result on LDP for dissipative SPDEs when the strong Feller property does not hold, including stochastic Navier-Stokes equations perturbed by a smooth bounded random kick force. 
The authors in  \cite{JNPS15} extended their work to the case of smooth unbounded random kick forces in \cite{JNPS18}, 
which was extended by Martirosyan and Nersesyan \cite{MN18,Ner19} to the continuous time case, giving a proof of the LDP for the 2D Navier-Stokes equation with a non-degenerate white in time noise that is  smooth in space. Combining this approach with Malliavin calculus, Nersesyan, Peng and Xu \cite{NPX23} proved the LDP in the case when there is no deterministic force and the white noise is  degenerate in Fourier space. The LDP for 2D Navier-Stokes equations with a space-time localised noise was also considered in Peng and Xu \cite{PX22} by using the same approach. The solution processes considered by these works are time homogeneous.

\vskip0.2in

\centerline{\textbf{Notations}}

\vskip0.05in

Let $X$ be a Banach space and $Y$ a polish space. $\To=\R/(2\pi)\mathbb{Z}$ is the $1$-dimensional torus with Lebesgue measure $m$ normalized as a probability measure. 
\begin{itemize}
\item $\id_{A}$ is the indicator function of a set $A$.  We use $m_V, M_V$ to denote the infimum and supremum of a bounded function $V$ on its whole domain. 
\item We denote $B_r^X(a)$ the closed ball in $X$ centered at $a$ with radius $r$. When $a=0$ is the origin, we simply denote the ball by $B_r^X$. The open ball is denoted by $\ob_r^X(a)$ (or $\ob_r^X$ if $a=0$). 
\item $C_b(Y)$ (respectively $B_b(Y)$) is the space of bounded continuous (measurable) functions on $Y$ with the usual supremum norm $\|\cdot\|_{\infty}$.  $C_b^1(X)$ is the subset of $C_b(X)$ that is Fréchet differentiable with a bounded continuous derivative. 

\item For any measurable function $\mathfrak{m}\geq 1$ on $X$, $C_{\mathfrak{m}}(X\times\To)$ (respectively $L_{\mathfrak{m}}^{\infty}(X\times\To)$) is the space of continuous (measruable) functions $f$ on $X\times\To$ such that 
\[\|f\|_{L_{\mathfrak{m}}^{\infty}} = \sup_{(x, h)\in X\times\To}\frac{|f(x, h)|}{\mathfrak{m}(x)}<\infty.\]
  
\item The space of regular potential functions $C_b^{1, 0}(X\times\To)$ is the set of functions $f\in C_b(X\times\To)$ that are continuously Fr\'echet differentiable in $x\in X$ and continuous in $h$ such that 
\[\sup_{h\in \To}\|\nabla f(\cdot, h)\|_{\infty}+ \sup_{x\in X, h_1\neq h_2}\frac{|f(x, h_1)-f(x, h_2)|}{\k(|h_1-h_2|)}<\infty,\]
for some $\k: [0, \infty)\to[0, \infty)$ that is continuous, increasing with $\k(0)=0$. 
\item $\mathcal{M}_+(Y)$ is the space of non-negative finite Borel measures on $Y$ with the topology of weak convergence. Given any bounded measurable function $f$ on $Y$ and $\mu\in \mathcal{M}_+(Y)$, we write 
\[\langle\mu, f\rangle = \langle f, \mu\rangle = \int_{Y}f(y)\mu(dy), \quad \text{ and } |f|_{\mu}=\int_{Y}|f(y)|\mu(dy).\]
\item $\P(Y)\subset \mathcal{M}_+(Y)$ is the space of probability measures on $Y$. $\P_{\mkm}(Y)$ is the set of probability measures $\mu$ such that  $\langle \mu, \mkm\rangle <\infty$. For $\eta, M>0$, we denote 
\[\mathcal{P}_{\eta, M}(H) = \left\{\mu\in \mathcal{P}(H): \int_{H}e^{\eta\|w\|^2}\mu(dw)\leq M\right\}. \]

\item $H:= \left\{w \in \mathrm{L}^2\left(\mathbb{T}^{2}, \mathbb{R}\right): \int_{\mathbb{T}^2} w d x=0\right\}$, where the norm is denoted by $\left\|\cdot\right\|$ and the inner product is $\langle\cdot,\cdot\rangle$. $H_{s}=H^s\cap H$ and the corresponding norms $\left\|\cdot\right\|_s$ by $\|w\|_{s}=\left\|\left(-\mathrm{\Delta}\right)^{s/2} w\right\|$. 

\item For any subset $A\subset H$, we denote $(A, \dh)$ the topological space with subspace topology induced from $H$. And for any $a\in H$, we denote $B_{R}^{H, A}(a): =B_R^H(a)\cap A$ (or $B_{R}^{H, A}$ if $a = 0$) as the closed ball in $(A, \dh)$. 
\item Lyapunov functions on $H$: $\mathfrak{m}_{\eta}(w) = e^{\eta\|w\|^2}$ and $\mkm_q(w) = 1+\|w\|^{2q}$, for $w\in H$ and $\eta, q>0$. 
\item $\lceil a\rceil$ is the smallest integer greater than or equal to $a$. $\phi^+$ (respectively $\phi^-$) is the positive (negative) part of a function. 
\item $\Phi_{s, t, h}$ (respectively $\Phi_{s, t} = \Phi_{s, t, 0}$) is the solution to equations \eqref{Settings-NS-family} (equations \eqref{NS}). $P_{s, t, h}^V$ is the Feynman-Kac operator \eqref{setting-Feynman-Kac} with potential $V$. $\H_t$ is the homogenized process \eqref{setting-homogenized-process} and $\P_t^V$ is its Feynman-Kac semigroup \eqref{setting-homogenized-feynman-kac}. $P_{s, t, h}=P_{s, t, h}^0$ and $\P_t=\P_t^0$ are the corresponding Markov operators. 
\end{itemize}

\section{Settings and main results}
Recall from the introduction that the equation under investigation is
\begin{align}\label{NS}
dw(t, x) + B(\mathcal{K} w, w) (t, x)d t = \nu\mathrm{\mathrm{\Delta}} w(t, x)dt  + f(t, x)dt + \sum_{i=1}^d g_idW_i(t), \quad t>s, \quad w(s) = w_0,
\end{align}
for $s\in \R, x\in\To^2$ and some integer $d\geq 1$.

\vskip0.05in

The standard $d$-dimensional two-sided  Brownian motion $W = (W_1, W_2, \cdots, W_d)$ is obtained as follows.  Let $W^{\pm}(t)$ be two independent standard $d$-dimensional Brownian motion, then  define
\begin{align*}
W(t) :=\left\{\begin{array}{rr}
W^{+}(t) \text{ , } t\geq 0,\\
W^{-}(-t) \text{ , } t < 0 .\\
\end{array}
\right.
\end{align*}
The sample space is denoted by  $(\mathrm{\mathrm{\Omega}} ,\mathcal{F},\mathbf{P})$, where $\mathrm{\Omega} = \left\{\omega\in C(\mathbb{R}, \mathbb{R}^d): \omega (0)=0\right\}$ endowed with the compact open topology, $\mathcal{F}$ is the Borel $\sigma$-algebra and $\mathbf{P}$  is the Wiener measure associated with the Brownian motion $W$. Denote by $\mathcal{F}_t: = \s(W(u)-W(v): -\infty<u, v\leq t)$ the filtration of $\sigma$-algebras  generated by  $W$.  The coefficients of the noise $\lp g_i\rp\subset \bigcap_{s>0} H_{s}$. 

\vskip0.05in

For the deterministic force, we assume that $f\in C_b(\R, H_2)$ is $2\pi$-periodic and Lipschitz. The well-posedness of \eqref{NS} under the above conditions is classical \cite{KS12}. For every $s\in\R$ and $w_0\in H$, the solution belongs to $C([s, \infty); H)\cap L_{\mathrm{loc}}^2((s, \infty); H_3)$ almost surely and it generates a random dynamical system. We denote the solution at time $t\geq s$ by $\Phi_{s, t}$. 

\subsection{Time symbols and Feynman-Kac operators}
Denote by $\b_t$ the circle rotation on the torus $\To = \R/2\pi\mathbb{Z}$, i.e., 
\begin{align*}
	\b_th=t+h \mod 2\pi, h\in\To, t\in\R. 
\end{align*}
We consider a family of Navier-Stokes equations indexed by time symbols $h\in\To$ which is obtained from \eqref{NS} by replacing $f(t, x)$ with $f(\b_th, x)$,
\begin{align}\label{Settings-NS-family}
dw(t, x) + B(\mathcal{K} w, w) (t, x)d t = \nu\mathrm{\mathrm{\Delta}} w(t, x)dt  + f(\b_th,  x)dt + \sum_{i=1}^d g_idW_i(t), \, t>s, \, w(s) = w_0,
\end{align}
with corresponding solution $\Phi_{s, t, h}(w_0)$. Note that the symbol $h=0$ corresponds to the original equation \eqref{NS}. For any $s\in\mathbb{R}, t\geq 0, (w, h)\in H\times\To$, It follows from the uniqueness of solution that 
\begin{align}\label{translation-identity-solution}
\Phi_{s, s+t, h}(\o, w) = \Phi_{0, t, \b_sh}(\t_s\o, w), \, \mathbf{P}\text{-}\mathrm{a.s.}, 
\end{align}
where $\theta$ is the Wiener shift. The coupled process 
\begin{align}\label{setting-homogenized-process}
	\H_t(w, h) = (\Phi_{0, t, h}(w), \b_th)
\end{align}
is a homogeneous Markov process on $H\times\To$.  

\vskip0.05in

For every symbol $h\in\To$ and potential $V\in C_b(H\times \To)$, the associated Feynman-Kac operator $P_{s, t, h}^V: C_b(H)\to C_b(H)$ corresponding to \eqref{Settings-NS-family} is defined by 
\begin{align}\label{setting-Feynman-Kac}
	P_{s, t, h}^V\phi(w) = \E\left[ \exp \left(\int_s^tV(\Phi_{s, r, h}(w), \b_rh)dr\right)\phi(\Phi_{s, t, h}(w))\right], \, \phi\in C_b(H), w\in H. 
\end{align}
It naturally induces by duality an action $P_{s, t, h}^{V*}$ on the space $\mathcal{M}_{+}(H)$ of non-negative finite Borel measures on $H$, 
\begin{align*}
	\la P_{s, t, h}^{V*}\mu, \phi\ra = \la \mu, P_{s, t, h}^V\phi\ra, \, \mu\in \mathcal{M}_{+}(H), \phi\in C_b(H). 
\end{align*}
It also satisfies the following evolution relation by the Markov property 
\begin{align*}
	P_{u, t, h}^V = P_{u, s, h}^VP_{s, t, h}^V, \, \text{ for all } u\leq s\leq t, \, h\in\To.
\end{align*}
The  translation relation \eqref{translation-identity-solution} of solutions yields the following translation identity of $P_{s, t, h}^V$,
\begin{align}\label{translation-identity}
	P_{s, s+t, h}^V = P_{0, t, \b_sh}^V, \, s\in\mathbb{R}, t\geq 0, h\in\To.
\end{align}
Combining this with the evolution relation we obtain the following cocycle property 
\begin{align}\label{Feynman-Kac-cocycle}
	P_{0, s+t, h}^V = P_{0, s, h}^VP_{0, t, \b_sh}^V, \, s, t\geq 0, h\in\To.
\end{align}
It is also convenient to use the homogenized Feynman-Kac semigroup $\P_t^V$ acting on $C_b(H\times\To)$ as 
\begin{align}\label{setting-homogenized-feynman-kac}
	\P_t^V\phi(w, h) = P_{0, t, h}^V\phi(\cdot, \b_th)(w), \, t\geq 0, (w, h)\in H\times\To, \phi\in C_b(H\times\To). 
\end{align}

\vskip0.05in

The Feynman-Kac operators with vanishing potential are the usual Markov operators, which we denote by $P_{s, t, h}:=P^0_{s, t, h}$, as well as $\P_{t}:=\P_{t}^0$.

\subsection{Main theorems}

\vskip0.05in

The following two theorems  are our main results under the standing assumption: 
\begin{align}\label{assumption}
	f\in C(\R, H_{2}) \text{ is }2\pi\text{-periodic and Lipschitz};  g_i\in H_{\infty}, \forall 1\leq i\leq d;  \text{ and } A_{\infty} = H, 
\end{align}
where the H\"ormander bracket condition $A_{\infty}=H$ is defined as in \eqref{A-infty}. 

\vskip0.05in

The first result is the following time inhomogeneous Ruelle-Perron-Frobenius type theorem. 
\begin{theorem}\label{asymptotic-behavior-feynman-kac}
Let $V\in C_b^{1,0}(H\times\To)$. Then there is a unique family of eigen-triple 
$(\G^V, \l^V, F^V)$ for the Feynman-Kac evolution operators $P_{0, t, h}^V$, such that
\[(\G^V, \l^V, F^V)\in C(\To, \P(H))\times C(\To,\R)\times C(H\times\To),\]
 with $F^V$ strictly positive everywhere. Moreover, the following properties hold: 
\begin{enumerate}
\item Normalization property: $\la\G^V(h), F^V(\cdot, h)\ra=1$ for all $h\in\To$. \\
Eigenmeasure/eigenfunction property, 
\begin{align*}
	&P_{0, t, h}^{V*}\G^V(h) = \l_{t, h, V}\G^V(\b_th), \qquad \text{(eigenmeasure property)} \\
	&P_{0, t, h}^{V}F^V(\cdot, \b_th)(w) = \l_{t, h, V}F^V(\cdot, h), \quad \text{(eigenfunction property)} 
\end{align*}
with positive eigenvalue 
\begin{align}\label{eigen-value-integral-representation}
\l_{t, h, V}=\exp\left(\int_0^t\l^V(\b_rh)dr\right),\, \text{ and } \l^V(h) = \la V(\cdot, h), \G^V(h)\ra, 
\end{align}
for all $h\in\To$ and $t\geq 0$. 
\item Stability property: for any $\eta\in (0, \eta_0)$ and $\phi\in C_{\mkm_{\eta}}(H\times\To)$, and $M, R>0$, denoting 
\[\widetilde{\phi}(w, h) = \phi(w, h) - \la\phi(\cdot, h), \G^V(h)\ra F^V(w, h), \quad (w, h)\in H\times\To, \]
then we have for every $h\in\To$ and $R>0$, 
\begin{enumerate}
\item Pullback stability: 
\begin{align*}
	&\lim_{t\to\infty}\left| \l_{t, \b_{-t}h, V}^{-1}P_{0, t, \b_{-t}h}^{V}\widetilde{\phi}(\cdot, h)\right|_{C_b(B_{R}^H)} = 0,\\
	&\lim_{t\to\infty}\left|  \l_{t, \b_{-t}h, V}^{-1}P_{0, t, \b_{-t}h}^{V}\widetilde{\phi}(\cdot, h)\right|_{\G^V(\b_{-t}h)} = 0,\\
	&\lim_{t\to\infty}\sup_{\mu\in\P_{\eta^{\prime}, M}(H)}\left|\la  \l_{t, \b_{-t}h, V}^{-1}P_{0, t, \b_{-t}h}^{V}\widetilde{\phi}(\cdot, h), \mu\ra\right| = 0, \, \forall \eta^{\prime}\in (\eta, \eta_0);
\end{align*}
\item Forward stability: 
\begin{align}\notag
	&\lim_{t\to\infty}\left| \l_{t, h, V}^{-1}P_{0, t, h}^{V}\widetilde{\phi}(\cdot, \b_th)\right|_{C_b(B_{R}^H)} = 0,\\\notag
	&\lim_{t\to\infty}\left| \l_{t, h, V}^{-1}P_{0, t, h}^{V}\widetilde{\phi}(\cdot, \b_th)\right|_{\G^V(h)} = 0,\\\label{main-thereom-forward-mxing}
	&\lim_{t\to\infty}\sup_{\mu\in\P_{\eta^{\prime}, M}(H)}\left|\la  \l_{t, h, V}^{-1}P_{0, t, h}^{V}\widetilde{\phi}(\cdot, \b_th), \mu\ra\right| = 0, \, \forall \eta^{\prime}\in (\eta, \eta_0).
\end{align}
\end{enumerate}
\end{enumerate}
\end{theorem}
\begin{remark}Here are some remarks about the theorem.
\begin{enumerate}
\item In fact, we have 
\[(\G^V, F^V)\in C(\To, \P_{\mkm_{\eta}}(H))\times C_{\mkm_q}(H\times\To),\]
for $\eta\in(0, \eta_0)$ and some $q=q(V)\geq 1$, where $\eta_0$ is from Proposition \ref{bounds}. See Proposition \ref{Existence-of-Eigenmeasure} and Proposition \ref{Existence-of-Eigenfunction}. 
\item If the potential $V=0$ then $(\G^V, \l^V, F^V) = (\G, 0, 1)$, where $\G$ is the unique periodic invariant measure of  the Markov operators $P_{0, t, h}$ of system \eqref{NS}. In this vein, Theorem \ref{asymptotic-behavior-feynman-kac} extends the unique ergodicity obtained in \cite{LL22} to general Feynman-Kac evolution operators. 
\end{enumerate}
\end{remark}

The second theorem is a Donsker-Varadhan type large deviation principle for occupation measures of the time inhomogeneous solution processes. Denote the occupation measure by
\begin{align}\label{occupation-measure}
	\z_{t, h} &= \frac1t\int_{0}^t\d_{(\Phi_{0, r, h}, \b_rh)}dr, 
\end{align}
which is a family of  random probability measures on $H\times\To$.

\begin{theorem}\label{LDP}
For any $\eta\in (0, \eta_0)$ and $M>0$, the family of random probability measures $\z_{t, h}$ satisfies the LDP for any $h\in\To$, uniformly with respect to initial measures $\s\in \P_{\eta, M}(H)$, i.e., there is a non-trivial good rate function $I: \P(H\times\To)\to[0, \infty]$ that does not depend on $\eta, M$ such that for any $h\in\To$, we have 
\begin{align*}
	&\limsup_{t\to\infty}\frac1t\log \sup_{\s\in \P_{\eta, M}(H)}\mathbf{P}_{\s}\{\z_{t, h}\in F\}\leq -\inf_{\mu\in F}I(\mu),\\
	&\liminf_{t\to\infty}\frac1t\log \sup_{\s\in \P_{\eta, M}(H) }\mathbf{P}_{\s}\{\z_{t, h}\in G\}\geq -\inf_{\mu\in G}I(\mu),
\end{align*}
for any closed set $F$ and open set $G$ in $\P(H\times\To)$. 
\end{theorem}
\begin{remark}
Here are some remarks about the theorem. 
\begin{enumerate}
\item  The rate function $I$ is the Legendre transformation of the pressure function $Q$, 
\begin{align}\label{rate-function}
	I(\mu)=
	\left\{\begin{array}{cc}
		\ds\sup_{V\in C_b(H\times\To)}\left(\la V, \mu\ra - Q(V)\right), &\mu\in\P(H\times\To),\\
		\qquad+\infty, &\text{otherwise.}
	\end{array}\right.
\end{align}
Here the pressure function $Q: C_b(H\times\To)\to \mathbb{R}$ is given by the limit  
\begin{align}\label{main-results-Q-ae-limit}
Q(V) = \lim_{t\to\infty}\frac{1}{t}\sup_{\s\in \P_{\eta, M}(H)}\log\E_{\s}\exp\la tV, \zeta_{t, h}\ra, \text{ for } V\in C_b(H\times\To),
\end{align}
where the limit $Q(V) $ does not depend on initial data $\s\in \P_{\eta, M}(H)$ and $h\in\To$. The function $Q$ is $1$-Lipschitz and convex. In particular,  
\[Q(V) = \int_{\To}\l^V(h)m(dh), \text{ for } V\in C_b^{1,0}(H\times\To),  \]
where  $\l^V\in C(\To, \R)$ is the eigenvalue obtained in Theorem \ref{asymptotic-behavior-feynman-kac}.  See Theorem \ref{Kifer-condition}.  
\item For any Borel set $E\subset \R$ and $\phi\in C_b(H)$, define the set 
\[K(\phi_{\G}, E):= \lp\mu\in \P(H\times\To): \la\phi_{\G}, \mu\ra \in E\rp,\]
where $\G\in C(\To, \P_{\mkm_{\eta}}(H))$ is the unique periodic invariant measure of the Navier-Stokes system \eqref{NS} obtained in \cite{LL22} and $\phi_{\G}(w, h) = \phi(w) - \la\phi,\G(h)\ra$. Note that 
\begin{align*}
\z_{t, 0}\in K(\phi_{\G}, E) \text{ is equivalent to }  \la\phi_{\G}, \z_{t, 0}\ra=\frac{1}{t}\int_{0}^{t}\left(\phi(\Phi_{0, r}) - \int_{H}\phi(w)\G(\b_r0)\right)dr \in E.
\end{align*}
Thus by Theorem \ref{LDP} one obtains a version of the LDP involving the unique periodic invariant measure for solution processes of \eqref{NS}:
\begin{align*}
	&\limsup_{t\to\infty}\frac1t\log\mathbf{P}_{\s}\lp\frac{1}{t}\int_{0}^{t}\left(\phi(\Phi_{0, r}) - \int_{H}\phi(w)\G(\b_r0)\right)dr\in E\rp\leq -L^+(\phi_{\G}, E),\\
	&\liminf_{t\to\infty}\frac1t\log\mathbf{P}_{\s}\lp\frac{1}{t}\int_{0}^{t}\left(\phi(\Phi_{0, r}) - \int_{H}\phi(w)\G(\b_r0)\right)dr\in E\rp\geq -L^-(\phi_{\G}, E),
\end{align*}
for any  $\s\in\P_{\eta, M}(H)$, where 
\begin{align*}
	L^-(\phi_{\G}, E) = \inf_{\mu\in K(\phi_{\G}, \mathring{E})}I(\mu), \,\text{ and } L^+(\phi_{\G}, E) = \inf_{\mu\in K(\phi_{\G}, \overline{E})}I(\mu).
\end{align*}
The central limit theorem proved in \cite{LL22} implies that $\sqrt{t}\la\phi_{\G}, \z_{t, 0}\ra$ is asymptotically Gaussian, which describes small deviations. Here the LDP gives quantitative information on the probabilities of large deviations. 
\end{enumerate}
\end{remark}

\section{A Ruelle-Perron-Frobenius type theorem}\label{Asymptotic-behavior-discrete-time}
The aim of this section is to prove the Ruelle-Perron-Frobenius type Theorem \ref{asymptotic-behavior-feynman-kac} on the asymptotic behavior of the Feynman-Kac evolution operators $P_{0, t, h}^V$ in terms of the corresponding eigen-triple. We first show the existence of eigen-triple of the discrete time operators $P_{0, k, h}^V, k\in\mathbb{N}$ in the first subsection \ref{subsection-Existence of eigen-triple}. Then in subsection \ref{subsection-Mixing, uniqueness and regularity} we prove the forward and pullback stability of the eigen-triple for $P_{0, k, h}^V$. Subsection \ref{subsection-uniqueness-regularity-discrete-time} is devoted to the proof of uniquenesses and certain regularity properties of the eigen-triple. The proof of Theorem \ref{asymptotic-behavior-feynman-kac} will be given in subsection \ref{subsection-proof-main-theorem-Feynman-Kac}.

\vskip0.05in
\subsection{Existence of eigen-triple}\label{subsection-Existence of eigen-triple}In this subsection we establish the existence of eigen-triple for the discrete time operators $P_{0, k, h}^V, k\in\mathbb{N}$. In Proposition \ref{Existence-of-Eigenmeasure} we prove the existence of eigenmeasure and eigenvalue by combining the idea of pullback and a Schauder-Tychonoff's fixed point theorem. Then in Proposition \ref{Existence-of-Eigenfunction} we construct the corresponding eigenfunction by utilizing the eigenmeasure and eigenvalue, as well as a Krylov–Bogolyubov averaging procedure. 

\vskip0.05in

For $\eta, A>0$, let $\overline{C_p(\To, \mathcal{P}_{\eta, A}(H))}$ be the closure of $C(\To, \mathcal{P}_{\eta, A}(H))$ with respect to the topology of point-wise convergence. Using the idea of pull-back we  define the following operator family 
\begin{align}\label{the-semigroup}
S_t^V(\g)(h): = \frac{P_{0, t, \b_{-t}h}^{V*}\g(\b_{-t}h)}{\big(P_{0, t, \b_{-t}h}^{V*}\g(\b_{-t}h)\big)(H)}, \quad (t, h)\in\mathbb{R}_{+}\times\To, \g\in \overline{C_p(\To, \mathcal{P}_{\eta, A}(H))}.
\end{align}
It is straightforward to verify that this family of operators has the semigroup property, whose fixed point gives the desired eigenmeasure.
 
\begin{proposition}[Existence of Eigenmeasure and eigenvalue]\label{Existence-of-Eigenmeasure}
Assume $V\in C_b^{1,0}(H\times\To)$ and $\eta\in(0,\eta_0)$.  Then there is a constant $A>0$ such that the operator $S_1^V$ maps $\overline{C_p(\To, \mathcal{P}_{\eta, A}(H))}$  into itself, in which it has a fixed point $\G^V$. Denote $\l_{k, h, V}=P_{0, k, h}^{V*}\G^V(h)(H)$. Then $\G^V$ has the following properties: 
\begin{enumerate}
\item Eigenmeasure property: 
\[P_{0, k, h}^{V*}\G^V(h) =\l_{k, h, V}\G^V(\b_kh), \, (k, h)\in\mathbb{N}\times\To.\]
Cocycle property of eigenvalues $\l_{k, h, V}$: 
\begin{align}\label{eigenvalue-cocycle}
\l_{m+\ell, h, V} = \l_{m, \b_{\ell}h, V}\l_{\ell, h, V}, \text{ and } \l_{m, h, V} = \prod_{k=0}^{m-1}\l_{1, \b_kh, V}.
\end{align}

\item Support property: for each $h\in\To$, the support of $\G^V(h) $ is $H$. In addition, for any $\r, r>0$, there is a number $p = p(\r, r, V)>0$ such that 
\begin{align}\label{support-property}
\inf_{(w, h)\in B_{\r}^{H_2}\times\To}\G^V(h)(B_{r}^{H, H_2}(w))\geq p. 
\end{align}
Moreover, there is a large $L_0$ such that for any $L\geq L_0$, there holds 
 \begin{align}\label{support-property-1}
\inf_{(w, h)\in B_{\r}^{H_2}\times\To}\G^V(h)(B_{r}^{H, B_L^{H_2}}(w))\geq p/2. 
\end{align}
\item Concentration property: for each $h\in\To$, the measure $\G^V(h)$ is concentrated on $H_2$, i.e. $\G^V(h) (H_2) = 1$.  There is a constant $C$ such that for any $R>0$,
\begin{align}\label{concentration-property}
\sup_{h\in\To}\G^V(h)(H\setminus B_{R}^{H_2})\leq \frac{C}{R^2}.
\end{align}

\end{enumerate}

\end{proposition}
\begin{proof}
The proof is divided into the following four steps. 

\vskip0.05in

{\it Step 1: Show that  $S_1^V$ maps  $C(\To, \mathcal{P}_{\eta, A}(H))$ into itself for appropriately chosen $A$}. Let $\g\in C(\To, \mathcal{P}_{\eta, A}(H))$ and recall from the notation that $m_V, M_V$ the infimum and supremum of $V$.  We first observe that 
\begin{align}\label{two-sided-bounds}
\begin{split}
0<e^{ m_V}\leq &\big(P_{0, 1, \b_{-1}h}^{V*}\g(\b_{-1}h)\big)(H) \leq e^{M_V}, \forall h\in \To. 
\end{split}
\end{align}
By estimate \eqref{eq: est1} and Jensen's inequality, 
\begin{align*}
\int_{H} e^{\eta\|w\|^2}\left(S_1^V(\g)(h)\right)(dw) 
	& =\frac{ \int_{H} \E\left[\exp\left(\int_0^1V\big(\H_r(w, \b_{-1}h)\big)dr+ \eta\|\Phi_{0,1,\b_{-1}h}(w)\|^2\right)\right]\g(\b_{-1}h)(dw)}{\big(P_{0, 1, \b_{-1}h}^{V*}\g(\b_{-1}h)\big)(H)}\\
	&\leq e^{M_V-m_V}\int_{H} \E e^{\eta\|\Phi_{0,1,\b_{-1}h}(w)\|^2}\g(\b_{-1}h)(dw)\\
	&\leq C e^{M_V-m_V}\int_{H} e^{\eta e^{-\nu }\|w\|^2}\g(\b_{-1}h)(dw)\\
	&\leq C e^{M_V-m_V}\left(\int_{H} e^{\eta\|w\|^2}\g(\b_{-1}h)(dw)\right)^{e^{-\nu }}\leq C A^{e^{-\nu }} e^{(M_V-m_V)},
\end{align*}
therefore if we choose $A\geq \exp\left(\frac{\ln C + M_V-m_V}{1-e^{-\nu }}\right)$ then $S_1^V(\g)(h)\in  \mathcal{P}_{\eta, A}(H)$ for every $h\in\To$. 

\vskip0.05in
 
It remains to show the continuity, that is $S_1^V(\g)(h_n)\to S_1^V(\g)(h)$ weakly if $h_n\to h$. Let $\phi$ be any bounded Lipschitz function on $H$. Since $V\in C_b^{1,0}(H\times\To)$, and $|\b_rh_n - \b_rh|=|h_n-h|$, one has 
\begin{align}\label{Lipschitz-modulo-Lyapunov}
\begin{split}
&\left|P_{0, 1, h_n}^V\phi(w) -P_{0, 1, h}^V\phi(w) \right|\\
	&\begin{aligned}\leq 
		&\E\left|e^{\int_0^1 V(\H_{r}(w, h_n))dr}\left(\phi(\Phi_{0, 1, h_n}(w))- \phi(\Phi_{0, 1, h}(w))\right)\right|\\
      		&\qquad+ \E\left|\left(e^{\int_0^1V(\H_{r}(w, h_n))dr}-e^{\int_0^1V(\H_{r}(w, h))dr}\right)\phi(\Phi_{0,1, h}(w))\right| 		     
	\end{aligned}\\
	&\leq C(V, \phi)\E\left(\|\Phi_{0, 1, h_n}(w)-\Phi_{0, 1, h}(w)\|+ \int_0^1\left|V(\H_{r}(w, h_n))-V(\H_{r}(w, h))\right|dr\right)\\
	&\leq C(V, \phi)\left(\E\|\Phi_{0, r, h_n}(w)-\Phi_{0, r, h}(w)\|+\k(|h_n-h|)\right)\leq C(V, \phi)e^{\eta\|w\|^2}(|h_n-h|+\k(|h_n-h|)),
\end{split}
\end{align}
for some modulus of continuity $\k$, where we used estimate \eqref{continuousonhull} and Lipschitz continuity of $f$ in the last step. Hence 
\begin{align}\label{eigen-measure-continuity-1}
\begin{split}
	\left|\Big\langle \g(h_n), P_{0, 1, h_n}^V\phi -P_{0, 1, h}^V\phi\Big\rangle\right| 
	&\leq C\int_H e^{\eta\|w\|^2}\g(h_n)(dw)(|h_n-h|+\k(|h_n-h|))\\
	&\leq C(|h_n-h|+\k(|h_n-h|)). 
\end{split}
\end{align}
Since $ P_{0, 1, h}^V$ is Feller, $ P_{0, 1, h}^V\phi \in C_b(H)$. By continuity of $\g$, we have  
\begin{align}\label{eigen-measure-continuity-2}
\left|\Big\langle \g(h_n) - \g(h), P_{0, 1, h}^V\phi\Big\rangle\right| \to 0, \quad n\to\infty. 
\end{align}
Therefore by \eqref{eigen-measure-continuity-1} and \eqref{eigen-measure-continuity-2}, we have 
\[\left|\Big\langle \g(h_n), P_{0, 1, h_n}^V\phi\Big\rangle -\Big\langle \g(h),  P_{0, 1, h}^V\phi\Big\rangle\right| \to 0, \quad n\to\infty, \]
which means $P_{0, 1, h_n}^{V*} \g(h_n)$ converges weakly to $ P_{0, 1, h}^V\g(h)$ since $\phi$ was chosen arbitrarily.  Hence 
\[h\to  P_{0, 1, h}^V\g(h)\]
is continuous. As a result, $S_1^V(\g)(h)$ is continuous in $h\in\To$. So $S_1^V$  maps  $C(\To, \mathcal{P}_{\eta, A}(H))$ into itself as desired. 

\vskip0.05in

{\it Step 2: Prove that $S_1^V$  maps $\overline{C_p(\To, \mathcal{P}_{\eta, A}(H))}$  into itself}. We claim that $S_1^V$  is a continuous map from $\mathcal{P}_{\eta, A}(H)^{\To}$ to itself with respect to the topology of point-wise convergence. Then it will follow that $S_1^V$  maps $\overline{C_p(\To, \mathcal{P}_{\eta, A}(H))}$  into itself.  Foy any $h\in\To$, let $\Pi_{h}$ be the canonical projection of $\mathcal{P}_{\eta, A}(H)^{\To}$  to the $h$ component, which is $\mathcal{P}_{\eta, A}(H)$. 

Note that $P_{0, t, \b_{-1}h}^{V*}$ is a continuous map from $\mathcal{P}_{\eta, A}(H)$ to itself, and 
\[F_1: \quad \g\to P_{0, 1, \b_{-1}h}^{V*}\g(\b_{-1}h) = P_{0, 1, \b_{-1}h_0}^{V*}\Pi_{\b_{-1}h}\g\]
is the composition of $P_{0, 1, \b_{-1}h}^{V*}$ with the projection $\Pi_{\b_{-1}h}$. Hence it is a continuous map from 
$\mathcal{P}_{\eta, A}(H)^{\To}$  to $\mathcal{P}_{\eta, A}(H)$.  Similarly, we see that 
\[F_2: \quad\g\to \frac{1}{P_{0, 1, \b_{-1}h}^{V*}\g(\b_{-1}h) (H)}\]
is a continuous map from $\mathcal{P}_{\eta, A}(H)^{\To}$  to $\R$ since $P_{0, 1, \b_{-1}h}^{V*}\g(\b_{-1}h) (H) $ is bounded away from zero by \eqref{two-sided-bounds}. Observe that the product map 
\[(a, \mu) \to a\mu\]
is continuous, hence $\Pi_{h}S_1^V$ as the composition of the product map with  $(F_2, F_1)$ is continuous. Since $h$ is arbitrary, we have that $S_1^V$ from $\mathcal{P}_{\eta, A}(H)^{\To}$ to itself is continuous.

\vskip0.1in

{\it Step 3: Existence of a fixed point}. We show that the image of $S_1^V$  is precompact and conclude the existence of a fixed point by Schauder-Tychonoff's fixed point theorem. In view of the well-known Tychonoff's compactness theorem, to show precompactness, we only need to show that $\left\{(S_1^V\g)(h): \g\in \overline{C_p(\To, \mathcal{P}_{\eta, A}(H))}\right\}$ is tight for every $h\in\To$. From estimate \eqref{contraction-polynomial}, one has 
\begin{align*}
&\int_{H}\|w\|_2^2\left(S_1^V\g(h)\right)(dw)\\
	& =\frac{1}{P_{0, 1, \b_{-1}h}^{V*}\g(\b_{-1}h)(H)} \int_H\E\left[\exp\left(\int_0^1V(\H_r(w, \b_{-1}h))dr\right)\|\Phi_{0, 1, 		\b_{-1}h}(w)\|_2^2\right]\g(\b_{-1}h)(dw)\\
	&\leq e^{M_V-m_V}\int_{H}Ce^{\eta\|w\|^2}\g(\b_{-1}h)(dw)\leq C. 
\end{align*}
 Hence 
\begin{align}\label{fixed-point-tail-estimate}
	 S_1^V\g(h)\{\|w\|_2>R\}\leq R^{-2}\int_{H}\|w\|_2^2\left(S_1^V\g(h)\right)(dw)\leq R^{-2}C. 
\end{align}
Since $\{\|w\|_2\leq R\}$ is compactly embedded in $H$,   we know that $\left\{(S_1^V\g)(h): \g\in \overline{C_p(\To, \mathcal{P}_{\eta, A}(H))}\right\}$  is tight. 

\vskip0.05in

Now observe that $\overline{C_p(\To, \mathcal{P}_{\eta, A}(H))}$ is a nonempty closed convex subset of the  locally convex Hausdorff topological vector space of the maps from $\To$ to the space of finite signed measures on $H$, and $ S_1^V$ is a continuous mapping of $\overline{C_p(\To, \mathcal{P}_{\eta, A}(H))}$ into a compact subset. Hence the Schauder-Tychonoff's theorem (See Appendix of \cite{BV62} or Theorem 10.1 in \cite{Pat19}) implies the existence of a fixed point $\G^V$ of $ S_1^V$ in $\overline{C_p(\To, \mathcal{P}_{\eta, A}(H))}$. 

\vskip0.05in

{\it Step 4: Prove that the fixed point has the desired properties in items (1)-(3)}. The eigenmeasure property and cocycle property directly follow from the definition of $S_k^V$.  Indeed, Since  $S_k^V$ has the semigroup property, $\G^V$ is also a fixed point of $S_m^V$, $m\in  \mathbb{N}$, i.e., 
\begin{align*}
 P_{0, m, \b_{-m}h}^{V*}\G^V(\b_{-m}h) = P_{0, m, \b_{-m}h}^{V*}\G^V(\b_{-m}h)(H)\G^V(h), \quad\forall h\in\To.
\end{align*}
Replacing $h$ by $\b_mh$, one obtains 
\begin{align*}
 P_{0, m, h}^{V*}\G^V(h) = P_{0, m, h}^{V*}\G^V(h)(H)\G^V(\b_mh), \quad\forall h\in\To.
\end{align*}
Combining this with the cocycle property of the inhomogeneous Feynman-Kac evolution operators, we have 
\begin{align*}
 P_{0, m, h}^{V*}\G^V(h) &=  P_{0, m-1, \b_1h}^{V*} P_{0, 1, h}^{V*}\G^V(h) = P_{0, m-2, \b_2h}^{V*} P_{0, 1, \b_1h}^{V*}P_{0,1, h}^{V*}\G^V(h)(H) \G^V(\b_1h)\\
 & = P_{0, m-3, \b_3h}^{V*} P_{0, 1, \b_2h}^{V*}\G^V(\b_2h) P_{0,1, \b_1h}^{V*}\G^V(\b_1h)(H)  P_{0,1, h}^{V*}\G^V(h)(H)\\
 &\cdots\\
 & = \prod_{k=0}^{m-1}P_{0,1, \b_kh}^{V*}\G^V(\b_kh)(H) \G^V(\b_mh).
\end{align*}
The desired cocycle property then follows since $ \G^V(\b_mh)(H) = 1$. 

\vskip0.05in 

The concentration property follows from \eqref{fixed-point-tail-estimate}. To show \eqref{support-property}, we use the uniform irreducibility from Proposition \ref{Uniform-Irreducibility}. From the concentration inequality \eqref{concentration-property}, we can choose an $R$ large such that 
\[\inf_{h\in \To}\G^V(h)(B_R^H)\geq 1/2. \]
 Choose $m$ large such that $m>2\ln R$. Then it follows from \eqref{two-sided-bounds} and Proposition \ref{Uniform-Irreducibility} that 
\begin{align*}
\G^V(h)(B_{r}^{H, H_2}(w)) &= \int_{H}\id_{B_{r}^{H, H_2}(w)}(u)(S_{m}^V\G^V)(h) (dw)\\
& =  \int_{H}\id_{B_{r}^{H, H_2}(w)}(u) \frac{P_{0, m, \b_{-m}h}^{V*}\G^V(\b_{-m}h)(du)}{\big(P_{0, m, \b_{-m}h}^{V*}\G^V(\b_{-m}h)\big)(H)}\\
&\geq  e^{-mM_V} \int_{H} P_{0, m, \b_{-m}h}^{V}(u, B_{r}^{H, H_2}(w))\G^V(\b_{-m}h)(du)\\
&\geq  e^{-mM_V} \int_{B_{R}^H} P_{0, m, \b_{-m}h}^{V}(u, B_{r}^{H, H_2}(w))\G^V(\b_{-m}h)(du)\geq \frac{p}{2}e^{-mM_V}, 
\end{align*}
where $p = p(r,\r)$ is from Proposition \ref{Uniform-Irreducibility}.  Hence \eqref{support-property} is proved. The proof of \eqref{support-property-1} is similar by combining \eqref{irreducible-large-H2ball}.  In particular we have shown that $\G^V(h)$ has a  positive measure for any nondegenerate ball in $H$, so its support is $H$.  The proof is complete. 
\end{proof}

\vskip0.05in 

We now prove the existence of an eigenfunction corresponding to the eigenvalue and eigenmeasure from the previous proposition. The scheme of the proof by using a Krylov–Bogolyubov averaging procedure is a time inhomogeneous version of the arguments developed in \cite{JNPS18}. 
\begin{proposition}[Existence of Eigenfunction]\label{Existence-of-Eigenfunction}
Assume $A_{\infty} = H $ and $V\in C_b^{1,0}(H\times\To)$. Then the family of functions 
\begin{align*}
	F_{k}(w, h): = \frac{1}{k}\sum_{\ell=0}^{k-1}\frac{P_{0,\ell, h}^V\id_H(w)}{\l_{\ell, h, V}}, \quad (k, h)\in \mathbb{N}\times\To
\end{align*}
is uniformly equicontinuous on $(B_R^{H_2}, \dh)$ for any $R>0$. So there is a subsequence converges to a function $F^V\in L_{\mathfrak{m}_{q}}^{\infty}(H\times\To)$ for $q=q(V)$ as in Proposition \ref{Growth-Condition}. Moreover,  $F^V$ satisfies the following 
\begin{enumerate}
\item Strict positivity: for every $(w, h)\in H\times\To$, 
	\begin{align}\label{strict-positive}
		F^V(w, h)>0.
	\end{align}
\item Eigenfunction property:
	\begin{align}\label{eigenfunction_property}
		P_{0, m, h}^VF^V(\cdot, \b_mh)(w) =\l_{m, h, V}F^V(w, h), \quad (w, h)\in H\times\To, m\in \mathbb{N}.
	\end{align}
\item Normalization property: 
	\begin{align}\label{Normalization-property}
		\langle F^V(\cdot, h), \G^V(h)\rangle = 1, \quad h\in\To. 
	\end{align}
\item Equicontinuity: for any $R>0$ the restriction of the family $\lp F^V(\cdot, h)\rp_{h\in\To}$ to $(B_R^{H_2},\dh)$ is equicontinuous. 
\end{enumerate}
\end{proposition}

\begin{proof}
The proof is divided into the following three steps. 

\vskip0.05in

{\it Step 1: Existence of $F^V\in L_{\mathfrak{m}_{q}}^{\infty}(H_2\times\To)$ satisfying items (3) and (4)}. By uniform Feller property from Proposition \ref{Uniform-Feller}, we have that 
\begin{align*}
	\frac{P_{0,\ell, h}^V\id_H(w)}{\|P_{0,\ell, h}^V\id_H\|_{R_0}} 
	= \frac{\quad \frac{P_{0,\ell, h}^V\id_H(w)}{\l_{\ell, h, V}}\quad }{\quad  \frac{\|P_{0,\ell, h}^V\id_H\|_{R_0}}{\l_{\ell, h, V}}\quad }
\end{align*}
is uniformly equicontinuous on $(B_R^{H_2},\dh)$ for any $R>0$. If one can show the existence of a constant $C=C(R_0)$ such that 
\begin{align}\label{Bound-summand-Fk}
	\sup_{\ell, h}\frac{\|P_{0,\ell, h}^V\id_H\|_{R_0}}{\l_{\ell, h, V}}\leq C,
\end{align}
then $\{F_k\}$ will be uniformly equicontinuous.  

\vskip0.05in

Let $g(w) = \ds\inf_{\ell, h} \frac{P_{0,\ell, h}^V\id_H(w)}{\|P_{0,\ell, h}^V\id_H\|_{R_0}}, w\in H_2$. Since the infimum of a uniform equicontinuous and bounded family is still continuous, $g\in C(B_{R}^{H_2},\dh)$ for any $R>0$. Note that $g$ is non-negative and each $\frac{P_{0,\ell, h}^V\id_H}{\|P_{0,\ell, h}^V\id_H\|_{R_0}}$ is on the unit sphere in $C(B_{R_0}^{H_2},\dh)$, hence there must be a point $w_{0}\in B_{R_0}^{H_2}$ and a constant $a>0$ such that $g(w_0)> 2a$. Then for each $R\geq R_0$,
by the fact that $g\in C(B_{R}^{H_2},\dh)$, there is $r>0$ such that $g(w)>a$ for any $w\in B_{r}^{H, B_{R}^{H_2}}(w_0)$.  By choosing $R$ large, the support property \eqref{support-property-1} implies the existence of a constant $c>0$ such that 
\[\inf_{h}\G^V(h)(B_{r}^{H, B_{R}^{H_2}}(w_0))\geq c. \]
Since $\G^V(h)$ is concentrated on $H_2$ for every $h\in\To$, one has 
\begin{align*}
\|P_{0,\ell, h}^V\id_H\|_{R_0}^{-1}\l_{\ell, h, V}
	&=\|P_{0,\ell, h}^V\id_H\|_{R_0}^{-1}\int_{H}P_{0,\ell, h}^V\id_H(w)\G^V(h)(dw)\\
	&\geq \int_{H_2}g(w)\G^V(h)(dw)\geq \int_{B_{r}^{H, B_{R}^{H_2}}(w_0)}g(w)\G^V(h)(dw)\\
	&\geq ca.
\end{align*}
This verifies the uniform bound \eqref{Bound-summand-Fk}. Hence $\lp F_k(\cdot, h)\rp$ is uniformly equicontinuous on $(B_R^{H_2},\dh)$ for any $R>0$. It is also uniformly bounded on any ball $(B_R^{H_2},\dh)$ by the growth condition from Proposition \ref{Growth-Condition}. 

\vskip0.05in

By the Arzelà–Ascoli theorem, there is a subsequence $\{F_{k_j}(\cdot, h)\}$ and a function $F^V(\cdot, h):H_2\to\mathbb{R}_+$ that is continuous when restricted on $(B_R^{H_2},\dh)$, such that 
\begin{align}\label{eigenfun-convergence}
	\lim_{j\to\infty}\sup_{w\in B_R^{H_2}}|F_{k_j}(w, h) - F^V(w, h) | = 0, \quad \forall R>0. 
\end{align}
The growth condition from Proposition \ref{Growth-Condition} implies that 
\begin{align*}
	\frac{\qquad \frac{\|P_{0,\ell, h}^V\id_H\|_{L_{\mathfrak{m}_{q}}^{\infty}}}{\l_{\ell, h, V}}\qquad }{\frac{\|P_{0,\ell, h}^V\id_H\|_{R_0}}{\l_{\ell, h, V}}}\leq \frac{\|P_{0,\ell, h}^V\mathfrak{m}_{q}\|_{L_{\mathfrak{m}_{q}}^{\infty}}}{\|P_{0,\ell, h}^V\id_H\|_{R_0}}\leq M_q. 
\end{align*}
Together with boundedness \eqref{Bound-summand-Fk}, this shows that 
\begin{align}\label{bound-summand-Fk1}
	\frac{P_{0,\ell, h}^V\id_H(w)}{\l_{\ell, h, V}} \leq C_{R_0}M_q\mkm_q (w), \quad w\in H, 
\end{align}
where the right-hand side is $\G^{V}(h)$ integrable. Hence Lebesgue's dominated convergence theorem and eigenmeasure property from Proposition \ref{Existence-of-Eigenmeasure} yield
\begin{align*}
	\langle F^V(\cdot, h), \G^V(h)\rangle = \lim_{j\to\infty}\langle F_{k_j}(\cdot, h), \G^V(h)\rangle = 1, \quad \forall h\in\To,
\end{align*}
which verifies the normalization property \eqref{Normalization-property}. It also follows from \eqref{eigenfun-convergence} and \eqref{bound-summand-Fk1} that $F^V\in L_{\mathfrak{m}_{q}}^{\infty}(H_2\times\To)$. Since for each $R>0$, the family $\{F_k(\cdot, h)\}_{h\in\To}$ is uniformly equicontinuous on the ball $(B_R^{H_2}, \dh)$, one deduces item (4) about equicontinuity of $\{F^V(\cdot, h)\}_{h\in\To}$. 

\vskip0.05in

{\it Step 2: Proof of $F^V\in L_{\mathfrak{m}_{q}}^{\infty}(H\times\To)$ and item (2) on the eigenfunction property}. We first show that the eigenfunction property is true for $w\in H_2$. By cocycle property \eqref{eigenvalue-cocycle}, 
\begin{align*}
	(\P_1F_{k_j})(w, h) 
	&= \frac{1}{k_j}\sum_{\ell=0}^{k_j-1}P_{0,1,h}^V\left(\frac{P_{0, \ell, \b_1h}^V\id_{H}(\cdot)}{\l_{\ell, \b_1h, V}}\right)(w)\\
	& = \frac{\l_{1, h, V}}{k_j}\sum_{\ell=0}^{k_j-1}\frac{P_{0, \ell+1, h}^V\id_{H}(w)}{\l_{1, h, V}\l_{\ell, \b_1h, V}}\\
	& =\frac{\l_{1, \b_1h, V}}{k_j}\sum_{\ell=0}^{k_j-1}\frac{P_{0, \ell, h}^V\id_{H}(w)}{\l_{\ell, h, V}} -\frac{\l_{1, \b_1h, V}}{k_j}\left(1+\frac{P_{0, k_j, h}^V\id_{H}(w)}{\l_{k_j, h, V}}\right)
\end{align*}
It follows from the construction of $F^V(w, h)$  and the boundedness from \eqref{bound-summand-Fk1} that the right-hand side converges to $\l_{1, \b_1h, V}F^V(w, h)$ as $j\to\infty$. The left-hand side converges to 
\[\P_1F^V(w, h) = P_{0,1,h}^VF^V(\cdot, \b_1h)(w)\]
by dominated convergence theorem and the fact that $w_{0,1,h}\in H_2$. This proves the eigenfunction property for $m=1$ and $(w, h)\in H_2\times\To$. The general case for $m\geq 2$ and $(w, h)\in H_2\times\To$ follows from cocycle property \eqref{eigenvalue-cocycle}.  By the regularizing property of the Navier-Stokes solution process, for any $(w,h)\in H\times\To$, the function $\frac{P_{0,1, h}^VF^V(\cdot, \b_1h)(w)}{\l_{1, h, V}}$ is well defined and satisfies (similar to the proof of \eqref{bound-summand-Fk1})
\begin{align*}
	\frac{P_{0,1, h}^VF^V(\cdot, \b_1h)(w)}{\l_{1, h, V}}\leq C(F^V, R_0)M_q\mkm_q(w), \quad w\in H. 
\end{align*}
 Define $F^V(w, h) = \frac{P_{0,1, h}^VF^V(\cdot, \b_1h)(w)}{\l_{1, h, V}}$ for those $(w, h)\in H\times\To\backslash H_2\times\To$. Then $F^V\in L_{\mathfrak{m}_{q}}^{\infty}(H\times\To)$ is a nonnegative function satisfying the eigenfunction property for any $(w, h)\in H\times\To$.  Since $\G^V(h)$ is concentrated on $H_2$, the definition on $H\setminus H_2$ does not affect the normalization identity \eqref{Normalization-property}. 
 
 \vskip0.05in
 
{\it Step 3: Proof of item (1) on the strict positivity}. Let $R\geq 1$, and $(w_0, h_0)\in B_R^H\times\To$. Since $\la F^V(\cdot, h), \G^V(h)\ra =1$, and $\G^V(h)$ is concentrated on $H_2$, for every $h\in\To$ there is a point $w_h\in H_2$ and a number $c = c(w_h)>0$ such that $F^V(w_h, h)\geq c$. To apply the irreducibility from Proposition \ref{Uniform-Irreducibility}, we choose an $R>0$ such that $w_0\in B_R^H$, and integer $\ell>2\ln R$. Then we choose $L>L_0(\ell, R)$ as in \eqref{irreducible-large-H2ball} large such that $w_{\b_{\ell}h_0}\in B_{L}^{H_2}$. Since  $F^V(\cdot, \b_{\ell}h_0)$ is continuous on $(B_{L}^{H_2}, \dh)$,  there is a radius $r>0$ such that 
\[F^V(w, \b_{\ell}h_0)\geq c/2 \text{ for any } w\in B_r^{H, B_L^{H_2}}(w_{\b_{\ell}h_0}).\]
By the eigenfunciton property \eqref{eigenfunction_property} and boundedness of $V$, we have, 
\begin{align*}
	F^V(w_0, h_0)
	 = \frac{P^V_{0, \ell, h_0}F^V(w_0, \b_{\ell}h_0)}{\l_{\ell, h_0, V}}
	&\geq e^{-\ell M_V}\int_{B_r^{H, B_L^{H_2}}(w_{\b_{\ell}h_0})}F^V(w, \b_{\ell}h_0)P^V_{0, \ell, h_0}(w_0, du)\\
	&\geq ce^{-\ell M_V}P^V_{0, \ell, h_0}(w_0, B_r^{H, B_L^{H_2}}(w_{\b_{\ell}h_0}))>0
\end{align*}
by \eqref{irreducible-large-H2ball}. 

 \vskip0.05in
 
The proof of Proposition \ref{Existence-of-Eigenfunction} is complete. 
\end{proof}

\subsection{Forward and pullback stability for discrete time}\label{subsection-Mixing, uniqueness and regularity}

This subsection is devoted to the proof of the forward stability in the following Theorem \ref{fd-convergence} for the eigen-triple of the discrete time Feynman-Kac evolution operators $P_{0, k, h}^V$. The pullback stability is stated in Theorem \ref{prop-convergence-discrete-time}  at the end of the section where the proof is omitted since it is similar to the proof of forward stability by combining the eigenfunction/eigenmeasure property in the pullback sense. 

\begin{theorem}[Forward stability the eigen-triple]\label{fd-convergence}
Let $\eta\in (0, \eta_0)$ and $V\in C_b^{1,0}(H\times\To)$. For any $\phi\in C_{\mkm_{\eta}}(H\times\To)$, $M, R>0, h\in\To$, and $\eta^{\prime}\in (\eta, \eta_0)$ we have
\begin{align}\label{fd-convergence-BH}
	&\lim_{k\to\infty}\left| \l_{k, h, V}^{-1}P_{0, k, h}^{V}\phi(\cdot, \b_kh) - \left\langle \phi(\cdot, \b_kh), \G^V(\b_kh)\right\rangle F^V(\cdot, h)\right|_{C_b(B_{R}^H)} = 0,\\\label{fd-convergence-Gh}
	&\lim_{k\to\infty}\left| \l_{k, h, V}^{-1}P_{0, k, h}^{V}\phi(\cdot, \b_kh) - \left\langle \phi(\cdot, \b_kh), \G^V(\b_kh)\right\rangle F^V(\cdot, h)\right|_{\G^V(h)} = 0,\,  \\\label{fd-covergence-eigenmeasure}
	&\lim_{k\to\infty}\sup_{\mu\in\P_{\eta^{\prime}, M}(H)}\left|\la  \l_{k, h, V}^{-1}P_{0, k, h}^{V*}\mu , \varphi(\cdot, \b_kh)\ra - \langle F^V(\cdot, h), \mu\rangle \la\G^V(\b_kh), \varphi(\cdot, \b_kh)\ra\right| = 0, 
\end{align}
\end{theorem}
This theorem is proved in subsection \ref{subsection-proof-forward-stability}, after we establish the $L^1$-forward stability  for differentiable observable functions in subsection \ref{subsection-proof-forward-stable-L1} through a version of Ruelle's lower bound technique. Before proceeding, let's introduce some notations and a lemma. 

\vskip0.05in

Fix $h\in\To$ for which the dependence on $h$ will be dropped unless stated throughout the proof. Set 
\begin{align}\label{def-g}
	g(w, h) = \phi(w, h) - \left\langle \phi(\cdot, h), \G^V(h)\right\rangle F^V(w, h)
\end{align}
and 
\begin{align}\label{def-gk}
	g_k(w, h) = \frac{(\P_k^Vg)(w, h)}{\l_{k, h, V}}.
\end{align}
It follows from the eigenfunction property \eqref{eigenfunction_property} that for every $h\in\To$, 
\begin{align}\label{def-gk-1}
	g_k(w, h) = \frac{P_{0, k, h}^{V}\phi(\cdot, \b_kh)(w)}{\l_{k, h, V}} - \left\langle \phi(\cdot, \b_kh), \G^V(\b_kh)\right\rangle F^V(w, h).
\end{align}
and 
\begin{align}\label{vanishing-meaning}
\la g_k(\cdot, h), \G^V(h)\ra=0, \text{ and } \la g_k^+(\cdot, h), \G^V(h)\ra = \la g_k^{-}(\cdot, h), \G^V(h)\ra = \frac12|g_k(\cdot, h)|_{\G^V(h)}. 
\end{align}
by combining the eigenmeasure property from Proposition \ref{Existence-of-Eigenmeasure}. The following observation is useful. 
\begin{lemma}\label{lemma-equicontinuity-gk}
Let $V \in C_b^{1,0}(H \times \mathbb{T})$, $\phi\in C_{b}^1(H\times\To)$ and $q$ be the number from Proposition \ref{Existence-of-Eigenfunction}. Then the family $g_k(\cdot, h)$ satisfies
\begin{align*}
	\{g_k(\cdot, h)\}_{(k, h)\in \mathbb{N}\times\To}\subset L_{\mathfrak{m}_{q}}^{\infty}(H\times\To),  
\end{align*}
and is uniformly equicontinuous on $C_b(B_R^{H_2}, \dh)$  for any $R>0$.
\end{lemma}
\begin{proof}
In view of \eqref{bound-summand-Fk1} and the fact that $\phi\in C_b^1(H\times\To)$, we have 
\begin{align*}
\frac{P_{0, k, h}^{V}\phi(\cdot, \b_kh)(w)}{\l_{k, h, V}}\leq \|\phi\|_{\infty}\frac{P_{0, k, h}^{V}\id_{H}(w)}{\l_{k, h, V}}\leq C\mkm_q(w),
\end{align*}
where the constant $C$ does not depend on $h, k$. In addition, $\left\langle \phi(\cdot, \b_kh), \G^V(\b_kh)\right\rangle$ is bounded uniformly for  $(k, h)\in \mathbb{N}\times\To$. It then follows from Proposition \ref{Existence-of-Eigenfunction} that the family 
\[\left\{\left\langle \phi(\cdot, \b_kh), \G^V(\b_kh)\right\rangle F^V(\cdot, h)\right\}_{(k, h)\in \mathbb{N}\times\To}\]
 is uniformly equicontinuous on $B_R^{H_2}$ and 
 \[\left|\left\langle \phi(\cdot, \b_kh), \G^V(\b_kh)\right\rangle F^V(w, h)\right|\leq C \mkm_q(w).\]
 As a consequence (note that the following \eqref{uniform-bound-gk} is still true if $\phi\in C_b(H\times\To)$), 
 \begin{align}\label{uniform-bound-gk}
  \sup_{(k, h)\in \mathbb{N}\times\To}|g_{k}(w, h)|\leq C \mkm_q(w). 
 \end{align}
 Furthermore, the uniform bound \eqref{Bound-summand-Fk} and the uniform Feller property in Proposition \ref{Uniform-Feller} imply that the family 
 \[\left\{\frac{P_{0, k, h}^{V}\phi(\cdot, \b_kh)}{\l_{k, h, V}}\right\}_{(k, h)\in \mathbb{N}\times\To}\]
  is also uniformly equicontinuous on $(B_R^{H_2}, \dh)$. Therefore so is $\{g_k(\cdot, h)\}_{(k, h)\in \mathbb{N}\times\To}$.  

\end{proof}

\subsubsection{Forward stability in $L^1$ norm} \label{subsection-proof-forward-stable-L1}

This subsection is devoted to prove the following proposition on forward stability in $L^1$ norm for differentiable observable functions.  

\begin{proposition}\label{proof-stability-L1}
For $V \in C_b^{1,0}(H \times \mathbb{T})$, $\phi\in C_{b}^1(H\times\To)$, and $h\in\To$, we have 
\begin{align}\label{equation-forward-stability-L1}
	\lim _{k \rightarrow \infty}\left|\lambda_{k, h, V}^{-1} P_{0, k, h}^V \phi\left(\cdot, \beta_k h\right)-\left\langle\phi\left(\cdot, \beta_k h\right), \Gamma^V\left(\beta_k h\right)\right\rangle F^V(\cdot, h)\right|_{\Gamma^V(h)}=0.
\end{align}
\end{proposition}
\begin{proof}
We prove the convergence by contradiction. The proof is divided into four steps.

\vskip0.05in

{\it Step 1: Lower bounds from contradiction assumption}.  Assume \eqref{equation-forward-stability-L1} does not hold. Then there is a constant $a>0$ and a subsequence $k_j\to\infty$ such that 
\begin{align}\label{contradiction-assumption}
	|g_{k_j}(\cdot, h)|_{\G^V(h)}\geq a.   
\end{align}
Due to the time inhomogeneity, here we do not have non-expansiveness of the sequence $g_k(\cdot, h)$ under $L^1(\G^V(h))$ norm.  Instead, the cocycle property \eqref{eigenvalue-cocycle} implies 
\begin{align}\label{g_k-invariance}
	g_k(w, h) = \frac{P^V_{0, l, h}g_{k-l}(\cdot, \b_{l}h)(w)}{\l_{l, h, V}}, \, \text{ for all }0\leq l\leq k,
\end{align}
which together with the eigenmeasure property from Proposition \ref{Existence-of-Eigenmeasure} yields
\begin{align}\label{g_k-non-increasing}
	|g_k(\cdot, h)|_{\G^V(h)}\leq |g_{k-l}(\cdot, \b_{l}h)|_{\G^V(\b_{l}h)}, \, \text{ for all }0\leq l\leq k, h\in\To.
\end{align}
This ``local" monotonicity along the line $x+y=k$ and \eqref{contradiction-assumption} imply that 
\begin{align*}
	|g_{m}(\cdot, \b_{n}h)|_{\G^V(\b_nh)}\geq a, \, \text{ for all }(m, n)\in K,
\end{align*}
where $K=\bigcup_{j\geq 1}K_j$ with 
\begin{align*}
	K_j=\lp(m, n)\in\mathbb{Z}^2: m, n\geq 0 \text{ and } m+n=k_j\rp.
\end{align*}
Combining with \eqref{vanishing-meaning}, we have 
\begin{align}\label{lower-bound-gpm-integral}
	|g_{m}^{\pm}(\cdot, \b_{n}h)|_{\G^V(\b_nh)}\geq a/2, \, \text{ for all }(m, n)\in K. 
\end{align}
Since $\G^V(h)$ is concentrated on $H_2$ for every $h$, we have from \eqref{concentration-property} and Lemma \ref{lemma-equicontinuity-gk} that 
\begin{align*}
	|g_{m}^{\pm}(\cdot, \b_{n}h)|_{\G^V(\b_nh)} 
	&= \int_{H_2}g_{m}^{\pm}(w, \b_{n}h)\G^V(\b_nh)(dw)\\
	&\leq \int_{B_{L_1}^{H_2}}g_{m}^{\pm}(w, \b_{n}h)\G^V(\b_nh)(dw)+C\int_{H_2\setminus B_{L_1}^{H_2}}\mkm_q(w)\G^V(\b_nh)(dw)\\
	&\leq \int_{B_{L_1}^{H_2}}g_{m}^{\pm}(w, \b_{n}h)\G^V(\b_nh)(dw)+a/4, 
\end{align*}
for some large $L_1$ independent of $m, n, h$.  This and \eqref{lower-bound-gpm-integral} imply the existence of $w_{m, n}^{\pm}\in B_{L_1}^{H_2}$ such that 
\begin{align*}
	g_{m}^{\pm}(w_{m, n}^{\pm}, \b_{n}h)\geq a/4
\end{align*}
for all $(m, n)\in K$.  Let $r_0$ be the radius from \eqref{UI-1} and $\ell_0=\lceil 2\ln r_0\rceil$, and $L_0=L_0(\ell_0, r_0)$ as in Proposition \ref{Uniform-Irreducibility}. Since the family $\lp g_{m}^{\pm}(, \b_{n}h)\rp_{ (m, n)\in K}$ is uniformly equicontinuous on $(B_{L_0}^{H_2}, \dh)$ by Lemma \ref{lemma-equicontinuity-gk}, there is a radius $r=r(L_0)$ independent of $m, n, h$ such that 
\begin{align}\label{lower-bound-gpm}
	g_{m}^{\pm}(w, \b_{n}h)\geq a/8, \text{ for all } w \in B_{r}^{H, B_{L_0}^{H_2}}(w_{m, n}^{\pm}). 
\end{align}
By \eqref{irreducible-large-H2ball} we have $p_0 = p_0(L_1, r)$ independent of $m, n$ such that 
\begin{align*}
	P^V_{0, \ell_0, h}(w_0, B_{r}^{H, B_{L_0}^{H_2}}(w_{m, n}^{\pm}))
	\geq \inf_{(w_0, h, w)\in B_{r_0}^{H_2}\times\To\times B_{L_1}^{H_2}}P^V_{0, \ell_0, h}(w_0, B_{r}^{H,
	B_{L_0}^{H_2}}(w))\geq p_0/2. 
\end{align*}
for all $m, n\in K$. Hence for any $\ell\geq 0$, by the Chapman-Kolmogorov relation, we have 
\begin{align*}
	P^V_{0, \ell+\ell_0, h}(w, B_{r}^{H, B_{L_0}^{H_2}}(w_{m, n}^{\pm})) 
	&= \int_{H}P^V_{\ell, \ell+\ell_0, h}(u, B_{r}^{H, B_{L_0}^{H_2}}(w_{m, n}^{\pm}))P^V_{0, \ell, h}(w, du)\\
	&\geq \int_{B_{r_0}^{H_2}}P^V_{0, \ell_0, \b_{\ell}h}(u, B_{r}^{H, B_{L_0}^{H_2}}(w_{m, n}^{\pm}))P^V_{0, \ell, h}(w, du
	\geq \frac12p_0P^V_{0, \ell, h}(w, B_{r_0}^{H_2})
\end{align*}
for all $w\in H$. As a result, for any $R\geq 1$, by taking $\ell = \lceil 2\ln R\rceil +\ell_0$, one obtains  by  \eqref{UI-1} that 
\begin{align}\label{lower-bound-ell0}
	P^V_{0, \ell, h}(w, B_{r}^{H, B_{L_0}^{H_2}}(w_{m, n}^{\pm})) \geq \frac12p_0P^V_{0, \ell-\ell_0, h}(w,
	B_{r_0}^{H_2})\geq p_0/4, \,\forall (w, h)\in B_R^H\times\To, (m, n)\in K. 
\end{align}

\vskip0.05in

{\it Step 2: A local Harnack type inequality from the lower bounds}. Let $R\geq 1$ and  $\ell = \lceil 2\ln R\rceil +\ell_0$. Then for each $j\geq 1$, we can write $k_j = {l}_j\ell+r_j$ where $l_j, r_j$ are nonnegative integers such that $l_j\to\infty$ as $j\to\infty$ and $ r_j<\ell$. We now establish an estimate which roughly states that the action of the Feynman-Kac evolution operators  on $g_{m}^{\pm}(\cdot, \b_{n}h)$ satisfies a local Harnack type inequality as long as $(m,n)\in K$. See \cite{KS00,JNPS18} in the time homogeneous case.

\vskip0.05in

Fix any $j$ such that $l_j\geq 2$. Note that for any $0\leq p < l_j$, we have $(k_j-(p+1)\ell, (p+1)\ell)\in K_j$. Let $w_0^{\pm} = w^{\pm}_{k_j-(p+1)\ell, (p+1)\ell}$ be the corresponding points as in \eqref{lower-bound-gpm}. Assume without loss of generality that $V\geq 0$ and recall that $\ds M_V = \sup_{w\in H}V(w)$. Then by boundedness of $V$, and estimates \eqref{lower-bound-gpm}, \eqref{lower-bound-ell0},
\begin{align*}
	\begin{split}
		\frac{P^V_{0, \ell, \b_{p\ell}h}g_{k_j-(p+1)\ell}^{\pm}(\cdot, \b_{(p+1)\ell}h)(w)}{\l_{\ell, \b_{p\ell}h, V}}
		&=\l_{\ell, \b_{p\ell}h, V}^{-1}\int_{H}g_{k_j-(p+1)\ell}^{\pm}(u, \b_{(p+1)\ell}h)P^V_{0, \ell, \b_{p\ell}h}(w, du)\\
		&\geq e^{-\ell M_V}\int_{B_r^{H, B_{L_0}^{H_2}}(w_{0}^{\pm})}g_{k_j-(p+1)\ell}^{\pm}(u, \b_{(p+1)\ell}h)P^V_{0, \ell, \b_{p\ell}h}(w, du)\\
		&\geq Ce^{-\ell M_{V}},
	\end{split}
\end{align*}
for all $w\in B_R^H$, where $C=p_0a/32$ is independent of $p, \ell, R, k_j$. On the other hand, from \eqref{Bound-summand-Fk},\eqref{uniform-bound-gk} and the growth condition from Proposition \ref{Growth-Condition} we deduce 
\begin{align}\label{bound-Tlg-by-meta}
	\begin{split}
		\frac{P^V_{0, \ell, \b_{p\ell}h}g_{k_j-(p+1)\ell}^{\pm}(\cdot, \b_{(p+1)\ell}h)(w)}{\l_{\ell, \b_{p\ell}h, V}}
		&\leq C\l_{\ell, \b_{p\ell}h, V}^{-1}P^V_{0, \ell, \b_{p\ell}h}\mkm_{q}(w)\\
		&\leq C\l_{\ell, \b_{p\ell}h, V}^{-1}\|P^V_{0, \ell, \b_{p\ell}h}\id_{H}\|_{R_0}\mkm_q(w)\\
		&\leq C\mkm_q(R),
	\end{split}
\end{align}
for all $w\in B_R^H$, where the constant $C\geq 1$ is independent of $p, \ell, R, k_j$.  As a result, we obtain a local Harnack type inequality, 
\begin{align*}
	\sup_{w\in B^{H}_R}\frac{P^V_{0, \ell, \b_{p\ell}h}g_{k_j-(p+1)\ell}^{\pm}(\cdot, \b_{(p+1)\ell}h)(w)}{\l_{\ell, \b_{p\ell}h, V}}	
	\leq A(R) \inf_{w\in B^{H}_R}\frac{P^V_{0, \ell, \b_{p\ell}h}g_{k_j-(p+1)\ell}^{\pm}(\cdot, \b_{(p+1)\ell}h)(w)}{\l_{\ell, 	
	\b_{p\ell}h, V}},
\end{align*}
where since $\ell = \lceil 2\ln R\rceil +\ell_0$, 
\begin{align}\label{eq-A(R)}
A(R) = C\mkm_{q}(R)e^{\ell M_V}\leq CR^{2M_V}\mkm_{q}(R), \quad C \text{ a constant independent of }p, \ell, R, k_j.
\end{align}
 It then follows by the eigenmeasure property from Proposition \ref{Existence-of-Eigenmeasure} that 
\begin{align*}
	|g_{k_j-(p+1)\ell}^{\pm}(\cdot, \b_{(p+1)\ell}h)|_{\G^V(\b_{(p+1)\ell}h)}
	& = \int_{H}\frac{P^V_{0, \ell, \b_{p\ell}h}g_{k_j-(p+1)\ell}^{\pm}(\cdot, \b_{(p+1)\ell}h)(w)}{\l_{\ell, \b_{p\ell}h, V}}\G^V(\b_{p\ell}h)(dw)\\
	&\leq A(R) \inf_{w\in B^{H}_R}\frac{P^V_{0, \ell, \b_{p\ell}h}g_{k_j-(p+1)\ell}^{\pm}(\cdot, \b_{(p+1)\ell}h)(w)}{\l_{\ell, 	
	\b_{p\ell}h, V}} + I_0,
\end{align*}
where 
\begin{align}\label{epsilon-R}
\begin{split}
	I_0&= \int_{H\setminus B^{H}_R} \frac{P^V_{0, \ell, \b_{p\ell}h}g_{k_j-(p+1)\ell}^{\pm}(\cdot, \b_{(p+1)\ell}h)(w)}{\l_{\ell, \b_{p\ell}h, V}}\G^V(\b_{p\ell}h)(dw)\\
	 &\leq C\sup_{h\in\To}\int_{H\setminus B^{H}_R} \mkm_{q}(w)\G^V(h)(dw):=  \varepsilon_0(R) 
\end{split}
\end{align}
by following the same estimates as in \eqref{bound-Tlg-by-meta}. Hence we arrive at 
\begin{align}\label{lower-bound-Tlg+-}
	\frac{P^V_{0, \ell, \b_{p\ell}h}g_{k_j-(p+1)\ell}^{\pm}(\cdot, \b_{(p+1)\ell}h)(w)}{\l_{\ell, 	
	\b_{p\ell}h, V}} - \frac{|g_{k_j-(p+1)\ell}^{\pm}(\cdot, \b_{(p+1)\ell}h)|_{\G^V(\b_{(p+1)\ell}h)}}{A(R)} +  \frac{\varepsilon_0(R)}{A(R)}\geq 0, 
\end{align}
for all $R\geq 1$, $\ell = \lceil 2\ln R\rceil +\ell_0$, $w\in B_{R}^{H}, j\geq 1, 0\leq p<l_j$, where $k_j=l_j\ell+r_j$. 

\vskip0.05in

{\it Step 3: An iteration scheme based on the local Harnack type inequality}. Now we prove that $|g_{k_j}(\cdot, h)|_{\G^V(h)}$ can be made as arbitrarily small by taking $j$ sufficiently large through an iteration scheme by utilizing the estimate \eqref{lower-bound-Tlg+-}. 
It follows from \eqref{g_k-invariance}, \eqref{vanishing-meaning} and \eqref{epsilon-R} that 
\begin{align}\label{iteration-p=1-0}
\begin{split}
	&|g_{k_j}(\cdot, h)|_{\G^V(h)} = \left|\frac{P^V_{0, \ell, h}g_{k_j-\ell}(\cdot, \b_{\ell}h)}{\l_{\ell, h, V}}\right|_{\G^V(h)}\\
	&\leq \int_{B^{H}_R}\left|\frac{P^V_{0, \ell, h}g_{k_j-\ell}(\cdot, \b_{\ell}h)(w)}{\l_{\ell, h, V}}\right|\G^V(h)(dw) + \varepsilon_0(R)\\
         &\begin{aligned}
         		  =\int_{B^{H}_R}\Big(\frac{P^V_{0, \ell, h}g_{k_j-\ell}^{+}(\cdot, \b_{\ell}h)}{\l_{\ell, h, V}}&-\frac{P^V_{0, \ell, h}g_{k_j-\ell}^{-}(\cdot,
          	\b_{\ell}h)}{\l_{\ell, h, V}}\\
          	&- \frac{|g_{k_j-\ell}^{+}(\cdot, \b_{\ell}h)|_{\G^V(\b_{\ell}h)}}{A(R)} +\frac{|g_{k_j-\ell}^{-}(\cdot, \b_{\ell}h)|_{\G^V(\b_{\ell}h)}}					{A(R)}\Big)\G^V(h)(dw) + \varepsilon_0(R)
	\end{aligned}\\
	&\begin{aligned}
		 \leq\int_{B^{H}_R}\Big(
		 \Big|\frac{P^V_{0, \ell, h}g_{k_j-\ell}^{+}(\cdot, \b_{\ell}h)}{\l_{\ell, h, V}}&- \frac{|g_{k_j-\ell}^{+}(\cdot, \b_{\ell}h)|_{\G^V(\b_{\ell}h)}}{A(R)}\Big|\\
		  & +\Big|\frac{P^V_{0, \ell, h}g_{k_j-\ell}^{-}(\cdot, \b_{\ell}h)}{\l_{\ell, h, V}}- \frac{|g_{k_j-\ell}^{-}(\cdot, \b_{\ell}h)|_{\G^V(\b_{\ell}h)}}{A(R)}\Big|
		  \Big)\G^V(h)(dw) + \varepsilon_0(R).
	\end{aligned}
\end{split}
\end{align}
Since for $c\geq 0$, inequality $a\geq -c$ implies $|a|\leq a+2c$, we have by \eqref{lower-bound-Tlg+-} with $p=0$ that 
\begin{align*}
	&\int_{B^{H}_R}\Big|\frac{P^V_{0, \ell, h}g_{k_j-\ell}^{\pm}(\cdot, \b_{\ell}h)}{\l_{\ell, h, V}}- \frac{|g_{k_j-\ell}^{\pm}(\cdot, \b_{\ell}h)|_{\G^V(\b_{\ell}h)}}{A(R)} \Big|\G^V(h)(dw) \\
	&\leq \int_{B^{H}_R}\Big(\frac{P^V_{0, \ell, h}g_{k_j-\ell}^{\pm}(\cdot, \b_{\ell}h)}{\l_{\ell, h, V}}- \frac{|g_{k_j-\ell}^{\pm}(\cdot, \b_{\ell}h)|_{\G^V(\b_{\ell}h)}}{A(R)} \Big)\G^V(h)(dw) + 2A(R)^{-1}\varepsilon_0(R)\\
	&=\int_{B^{H}_R}\Big(\frac{P^V_{0, \ell, h}g_{k_j-\ell}^{\pm}(\cdot, \b_{\ell}h)}{\l_{\ell, h, V}} \Big)\G^V(h)(dw) - \frac{|g_{k_j-\ell}^{\pm}(\cdot, \b_{\ell}h)|_{\G^V(\b_{\ell}h)}}{A(R)}\G^V(h)(B_R^H) + 2A(R)^{-1}\varepsilon_0(R).
\end{align*}
Combining this estimate with \eqref{iteration-p=1-0}, and the eigenfunction property from Proposition \ref{Existence-of-Eigenmeasure} and \eqref{vanishing-meaning}, we deduce
\begin{align}\label{pre-iteration}
\begin{split}
	|g_{k_j}(\cdot, h)|_{\G^V(h)}
	\leq &\int_{H}\Big(\frac{P^V_{0, \ell, h}g_{k_j-\ell}^{+}(\cdot, \b_{\ell}h)}{\l_{\ell, h, V}}+\frac{P^V_{0, \ell, h}g_{k_j-\ell}^{-}(\cdot, \b_{\ell}h)}{\l_{\ell, h, V}} \Big)\G^V(h)(dw) \\
	& - \frac{\G^V(h)(B_R^H)}{A(R)}\Big(|g_{k_j-\ell}^{+}(\cdot, \b_{\ell}h)|_{\G^V(\b_{\ell}h)}+|g_{k_j-\ell}^{-}(\cdot, \b_{\ell}h)|_{\G^V(\b_{\ell}h)}\Big) +4A(R)^{-1}\varepsilon_0(R)+\varepsilon_0(R)\\
	&\leq a_0(R)|g_{k_j-\ell}(\cdot, \b_{\ell}h)|_{\G^V(\b_{\ell}h)} + \k(R)
\end{split}
\end{align}
where 
\begin{align}\label{kappa-R}
	a_0(R) = 1- A(R)^{-1}\G^V(h)( B^{H}_R), \text{ and } \k(R) = (4A(R)^{-1}+1)\varepsilon_0(R). 
\end{align}
Apply the same estimates as in \eqref{iteration-p=1-0}-\eqref{pre-iteration} to $|g_{k_j-\ell}(\cdot, \b_{\ell}h)|_{\G^V(\b_{\ell}h)}$ by using \eqref{lower-bound-Tlg+-} with $p=1$ we have
\begin{align*}
	|g_{k_j-\ell}(\cdot, \b_{\ell}h)|_{\G^V(\b_{\ell}h)}\leq a_1(R)|g_{k_j-2\ell}(\cdot, \b_{2\ell}h)|_{\G^V(\b_{2\ell}h)}+ \k(R),
\end{align*}
where 
\[a_1(R) = 1- A(R)^{-1}\G^V(\b_{\ell}h)( B^{H}_R). \]
By iteration through using \eqref{lower-bound-Tlg+-} for $0\leq p<l_j$, one obtains
\begin{align}\label{gkj-contraction-1}
	|g_{k_j}(\cdot, h)|_{\G^V(h)}\leq \d_{l_j-1} |g_{r_j}(\cdot, \b_{l_j\ell}h)|_{\G^V(\b_{l_j\ell}h)}+ \eta_{l_j}\k(R)\leq C\d_{l_j-1}+\eta_{l_j}\k(R),
\end{align}
where 
\begin{align*}
	\d_{m} = \prod_{i=0}^{m}a_i, \, \eta_{l_j} = 1+\sum_{m=0}^{l_j-2}\d_{m}, \, \text{ and } a_i = a_i(R) = 1- A(R)^{-1}\G^V(\b_{i\ell}h)( B^{H}_R), 
\end{align*}
and $C$ does not depend on the parameters $k_j, h, \ell, R, l_j, r_j$. Note that in the last inequality of \eqref{gkj-contraction-1}, we used Lemma \ref{lemma-equicontinuity-gk} and the integrability of $\mkm_q$ with respect to $\G^V$ from Proposition \ref{Existence-of-Eigenmeasure}. 

\vskip0.05in

{\it Step 4: Deriving a contradiction}. By the concentration property \eqref{concentration-property}, we can find $R_0>0$ such that for all $R\geq R_0$, 
\[\G^V(h)( B^{H}_R)\geq 1/2, \quad\forall h\in \To. \]
Therefore 
\begin{align*}
	0<a_i\leq 1- \frac12 A(R)^{-1}:=\d(R)<1. 
\end{align*}
Combining this with \eqref{gkj-contraction-1}, the formula of $\k(R)$ from \eqref{kappa-R}, $A(R)$ from \eqref{eq-A(R)} and $\varepsilon_0(R)$ from \eqref{epsilon-R}, we know 
\begin{align}\label{gkj-contraction-2}
\begin{split}
	|g_{k_j}(\cdot, h)|_{\G^V(h)}
	&\leq C\d(R)^{l_j-1}+  \frac{ \k(R)}{1-\d(R)} = C\d(R)^{l_j-1}+  2(4+A(R))\varepsilon_0(R) \\
	&\leq C\left(\d(R)^{l_j-1}+ R^{2M_V}\mkm_{q}(R)\sup_{h\in\To}\int_{H\setminus B_{R}^{H}}\mkm_{q}(w)\G^V(h)(dw)\right), 
\end{split}
\end{align}
for any $R\geq R_0$. By Proposition \ref{Existence-of-Eigenmeasure} and Markov's inequality, we find 
\begin{align}\label{gkj-contraction-3}
\begin{split}
	\int_{H\setminus B_{R}^{H}}\mkm_{q}(w)\G^V(h)(dw)&\leq \left(\int_{H}\mkm_{q}^2(w)\G^V(h)(dw)\right)^{1/2}\left(\G^V(h)(H\setminus B_{R}^{H})\right)^{1/2}\\
	&\leq C\left(\G^V(h)\left\{\|w\|^{6(M_V+q)}\geq R^{6(M_V+q)}\right\}\right)^{1/2}\\
	&\leq \frac{C}{R^{3(M_V+q)}}.
\end{split}
\end{align}
Combining \eqref{gkj-contraction-2} and \eqref{gkj-contraction-3}, for any given $\varepsilon>0$, one can first choose $R$ large to make 
\[R^{2M_V}\mkm_{q}(R)\sup_{h\in\To}\int_{H\setminus B_{R}^{H}}\mkm_{q}(w)\G^V(h)(dw)<\frac{\varepsilon}{2C},\]
then choose $j$ large such that 
\[\d(R)^{l_j-1}<\frac{\varepsilon}{2C},\]
which result in $|g_{k_j}(\cdot, h)|_{\G^V(h)}<\varepsilon$. This contradicts with \eqref{contradiction-assumption}. As a result, convergence \eqref{equation-forward-stability-L1} must be true.  
\end{proof}

\subsubsection{Proof of Theorem \ref{fd-convergence}} \label{subsection-proof-forward-stability}

In this subsection we complete the proof of the forward stability as in Theorem \ref{fd-convergence}, by combining the $L^1$-stability established in Proposition \ref{proof-stability-L1},  the uniform equicontinuity and standard approximation arguments. 

\vskip0.05in
\begin{proof}[Proof of Theorem \ref{fd-convergence}]
 We first show that Proposition \ref{proof-stability-L1} implies \eqref{fd-convergence-BH} for $\phi\in C_{b}^1(H\times\To)$. Then we prove them for $\phi\in C_{\mkm_{\eta}}(H\times\To)$ through approximation arguments. The convergence \eqref{fd-covergence-eigenmeasure} is given at the end. The proof is divided into the following four steps. 

\vskip0.05in

{\it Step 1: Convergence \eqref{fd-convergence-BH} in $C_b(B_R^{H_2}, \dh)$ for $\phi\in C_{b}^1(H\times\To)$}. This is proved by a contradiction argument. Indeed, if there is $R>0$ such that convergence \eqref{fd-convergence-BH} in  $C_b(B_R^{H_2}, \dh)$ does not hold, then there are $a, r>0$ and a subsequence $k_j$ and $\{w_{k_j}\}\subset B_R^{H_2}$ such that 
\[|g_{k_j}(w, h)|\geq a>0,  \forall w\in B_{r}^{H, B_{L}^{H_2}}(w_{k_j}),\quad k\in\mathbb{N}\]
by uniform equicontinuity. Here $r>0$ is independent of $h, k$ and $L$ is any large constant from \eqref{support-property-1}.  The support property \eqref{support-property-1} then implies that 
\begin{align*}
|g_{k_j}(w, h)|_{\G^V(h)}\geq a\G^V(h)(B_{r}^{H, B_{L}^{H_2}}(w_{k_j}))\geq a\inf_{(w, h)\in B_R^{H_2}\times\To}\G^V(h)(B_{r}^{H, B_{L}^{H_2}}(w))>0
\end{align*}
for all $j\geq 1$, which contradicts with the convergence \eqref{fd-convergence-Gh}.

\vskip0.05in

The proof of {\it Step 1} is complete. 

\vskip0.05in

{\it Step 2: Convergence of \eqref{fd-convergence-BH} for $\phi\in C_{b}^1(H\times\To)$ in $C_b(B_R^{H})$}.  First observe that by \eqref{g_k-invariance}, one has for any $R>0$, 
\begin{align*}
	|g_{k+1}(w, h)| &= \left|\int_H g_{k}(u, \b_1h)\frac{P^V_{0, 1, h}(w, du)}{\l_{0, 1,  h}}\right|\\
	&\leq \int_{B_R^{H_2}}\left| g_{k}(u, \b_1h)\right|\frac{P^V_{0, 1, h}(w, du)}{\l_{0, 1,  h}}+\int_{H\setminus B_R^{H_2}}\left| g_{k}(u, \b_1h)\right|\frac{P^V_{0, 1, h}(w, du)}{\l_{0, 1,  h}}\\
	&:=I_1 + I_2.
\end{align*}
By \eqref{bound-summand-Fk1} and the uniform convergence $\ds\sup_{u\in B_R^{H_2}}|g_{k}(u, \b_1h)|\to 0$ from {\it Step 1}, one has for any $r>0$,
\begin{align*}
\sup_{w\in B_r^{H}}I_1
	&\leq \sup_{u\in B_R^{H_2}}|g_{k}(u, \b_1h)|\sup_{w\in B_r^{H}}\frac{P^V_{0, 1, h}\id_H(w)}{\l_{0, 1,  h}}\\
	&\leq C_{R_0}M_q\mkm_q (r)\sup_{u\in B_R^{H_2}}|g_{k}(u, \b_{-k}h)| \to 0, \quad k\to \infty.
\end{align*}
Now it follows from \eqref{uniform-bound-gk}, \eqref{contraction-polynomial} and Markov's inequality that 
\begin{align*}
	I_2&\leq C\int_{H\setminus B_R^{H_2}}\mkm_q(u)\frac{P^V_{0, 1, h}(w, du)}{\l_{0, 1,  h}}\\
	&\leq CR^{-1}\left(\E_w\mkm_q^2(\Phi_{0, 1, \b_1h})\right)^{1/2}\left(\E_{w}\|\Phi_{0, 1, \b_1h}\|_2^2\right)^{1/2}\leq CR^{-1}e^{\eta\|w\|^2}.
\end{align*}
Hence 
\[\sup_{w\in B_r^{H}}I_2\to 0, \quad R\to\infty. \]
As a consequence, we deduce for any $r>0$,
\[\sup_{w\in B_r^{H}}|g_k(w, h)|\to 0, \quad k\to\infty.\]
The proof of {\it Step 2} is complete. 

\vskip0.05in

{\it Step 3: Convergence \eqref{fd-convergence-BH} and \eqref{fd-convergence-Gh} for $\phi\in C_{\mkm_{\eta}}(H\times\To)$}. We only show the proof of these convergences for $\phi\in C_{b}(H\times\To)$.  The proof for $\phi\in C_{\mkm_{\eta}}(H\times\To)$ through a cut-off procedure is similar. 

\vskip0.05in

By Theorem A.8 in \cite{FK06}, the space $C_{b}^1(H)$ is dense in $C_{b}(H)$ when restricting on any compact subsets of $H$. Therefore by viewing $\{\phi(\cdot, \b_kh)\}_{k\geq 0}$ as an element in $\ell^{\infty}(\mathbb{N}, C_b(H))$,  there is a sequence $\phi_m = \{\phi_m(\cdot, \b_kh)\}_{k\geq 0}\in \ell^{\infty}(\mathbb{N}, C_b(H))$ such that $\ds\sup_{m\geq1}\|\phi_m\|_{\ell^{\infty}(\mathbb{N}, C_b(H))}<\infty$ and $\phi_m(\cdot, \b_kh)\to\phi(\cdot, \b_kh)$ uniformly for $k$ and uniformly on compact subsets of $H$ as $m\to\infty$. Denote by $g_{k, m}$ the function obtained by replacing $\phi(\cdot, \b_kh)$  with  $\phi_m(\cdot, \b_{k}h)$ in \eqref{def-gk-1}. And denote by $\ds\|\varphi\|_{r} = \sup_{w\in B_{r}^{H}}|\varphi(w)|$. Then in view of \eqref{def-gk-1}, we have 
\begin{align*}
	\|g_k(\cdot, h)\|_{r}\leq \|g_{k, m}(\cdot, h)\|_{r} &+ \sup_{h\in\To}\|F^V(\cdot, h)\|_{r}\la|\phi(\cdot, \b_kh)-\phi_m(\cdot, \b_kh)|, \G^V(\b_kh)\ra \\
	&+\frac{\|P_{0, k, h}^V(\phi-\phi_m)(\cdot, \b_kh)\|_r}{\l_{k, h, V}}.
\end{align*}
It follows from Proposition \ref{Existence-of-Eigenfunction} that $\ds\sup_{h\in\To}\|F^V(\cdot, h)\|_{r}<\infty$. Also from \eqref{concentration-property}, we have for any $R\geq 1$, 
\begin{align*}
	&\la|\phi(\cdot, \b_kh)-\phi_m(\cdot, \b_kh)|, \G^V(\b_kh)\ra \\
	&= \la|\phi(\cdot, \b_kh)-\phi_m(\cdot, \b_kh)|\id_{B_{R}^{H_2}}, \G^V(\b_kh)\ra  + \la|\phi(\cdot, \b_kh)-\phi_m(\cdot, \b_kh)|\id_{H\setminus B_{R}^{H_2}}, \G^V(\b_kh)\ra \\
	&\leq |\phi(\cdot, \b_kh)-\phi_m(\cdot, \b_kh)|_{C_b(B_{R}^{H_2})} + C\la\id_{H\setminus B_{R}^{H_2}}, \G^V(\b_kh)\ra\\
	&\leq |\phi(\cdot, \b_kh)-\phi_m(\cdot, \b_kh)|_{C_b(B_{R}^{H_2})}+ CR^{-2},
\end{align*}
where $C$ is independent of $k, m, R$. Therefore for any $\varepsilon>0$, by first choosing $R$ large and then choose $m$ large, we can find $M_1>0$ such that for all $m\geq M_1$ and $k\geq0$, 
\begin{align}\label{step3-1}
	\la|\phi(\cdot, \b_kh)-\phi_m(\cdot, \b_kh)|, \G^V(\b_kh)\ra <\varepsilon/3.
\end{align}
Let $L\geq 1$. By the uniform bounds \eqref{Bound-summand-Fk}, \eqref{bound-summand-Fk1} and the growth condition for the function $N_2$ in Proposition \ref{Growth-Condition}, we have for $w\in B_r^H$, 
\begin{align*}
\frac{|P_{0, k, h}^V(\phi-\phi_m)(\cdot, \b_kh)(w)|}{\l_{k, h, V}} 
	& \leq \frac{|P_{0, k, h}^V\id_{B_{L}^{H_2}}(\phi-\phi_m)(\cdot, \b_kh)(w)|}{\l_{k, h, V}}+\frac{|P_{0, k, h}^V\id_{H\setminus B_{L}^{H_2}}(\phi-\phi_m)(\cdot, \b_kh)(w)|}{\l_{k, h, V}}\\
	&\leq  |\phi(\cdot, \b_kh)-\phi_m(\cdot, \b_kh)|_{C_b(B_{L}^{H_2})}\frac{\|P_{0, k, h}^V\id_{H}\|_{r}}{\l_{k, h, V}}+C\frac{\left\|P_{0, k, h}^V\id_{H\setminus B_{L}^{H_2}}\right\|_{r}}{\l_{k, h, V}}\\
	&\leq C\left(|\phi(\cdot, \b_kh)-\phi_m(\cdot, \b_kh)|_{C_b(B_{L}^{H_2})}+ L^{-2}\frac{\left\|P_{0, k, h}^VN_2\right\|_{r}}{\l_{k, h, V}}\right)\\
	&\leq C\left(|\phi(\cdot, \b_kh)-\phi_m(\cdot, \b_kh)|_{C_b(B_{L}^{H_2})}+ L^{-2}\frac{\left\|P_{0, k, h}^V\id_{H}\right\|_{R_0}}{\l_{k, h, V}}\right)\\
	&\leq C(|\phi(\cdot, \b_kh)-\phi_m(\cdot, \b_kh)|_{C_b(B_{L}^{H_2})}+ L^{-2}),
\end{align*}
where the constant $C$ does not depend on $L, k, m$. Hence by first choosing $L$ large and then choosing $m$ large, we can find $M_2>0$ such that for all $m\geq M_2$ and $k\geq 0$
\begin{align}\label{step3-2}
	\sup_{w\in B_r^H}\frac{\|P_{0, k, h}^V(\phi-\phi_m)(\cdot, \b_kh)\|_r}{\l_{k, h, V}}<\varepsilon/3.
\end{align}
The result from {\it Step 2} implies that for $m=\max\{M_1, M_2\}$, one can find $K>0$ such that for all $k\geq K$, 
\[\|g_{k, m}(\cdot, h)\|_{r}<\varepsilon/3.\]
Combining this with \eqref{step3-1} and \eqref{step3-2} we obtain the desired convergence \eqref{fd-convergence-BH} for $\phi\in C_{b}(H\times\To)$.  The convergence \eqref{fd-convergence-Gh} follows from \eqref{fd-convergence-BH} and the concentration property \eqref{concentration-property} and the fact that $\G^V(h)\in \P_{\eta, A}(H)$ from Proposition \ref{Existence-of-Eigenmeasure}. 

\vskip0.05in

The proof of {\it Step 3} is complete. 

\vskip0.05in

{\it Step 4: Proof of convergence \eqref{fd-covergence-eigenmeasure}}. Let $\phi\in C_{\mkm_{\eta}}(H\times\To)$ and $\eta^{\prime}\in(\eta, \eta_0)$. Then by the uniform bound \eqref{Bound-summand-Fk} and growth condition from Proposition \ref{Growth-Condition}, we have 
\begin{align*}
\frac{P_{0, k, h}^V\phi(w, h)}{\l_{k, h, V}}
	&\leq C\frac{P_{0, k, h}^V\mkm_{\eta}(w)}{\l_{k, h, V}}\\
	&=C\frac{P_{0, k, h}^V\mkm_{\eta}(w)}{\|P_{0, k, h}^V\id_{H}\|_{R_0}}\frac{\|P_{0, k, h}^V\id_{H}\|_{R_0}}{\l_{k, h, V}}\leq C\mkm_{\eta}(w). 
\end{align*}
Since $F^V\in L_{\mathfrak{m}_{q}}^{\infty}(H\times\To)$, one has $g_k\in L_{\mkm_{\eta}}^{\infty}(H\times\To)$, where $g_k$ is the function obtained from \eqref{def-gk} with $\phi\in C_{\mkm_{\eta}}(H\times\To)$.  Hence by H\"older's inequality and Markov inequality, one has 
\begin{align}\label{inequality-step4}
\begin{split}
\la |g_k(\cdot, h)|, \mu\ra 
	&= \la |g_k(\cdot, h)|\id_{B_{R}^H}, \mu\ra  +  \la |g_k(\cdot, h)|\id_{H\setminus B_{R}^H}, \mu\ra \\
	&\leq |g_k(\cdot, h)|_{C_b(B_{R}^H)} + C\la \mkm_{\eta}\id_{H\setminus B_{R}^H}, \mu\ra\\
	&\leq |g_k(\cdot, h)|_{C_b(B_{R}^H)} + C\la \mkm_{\eta^{\prime}}, \mu\ra^{\eta/\eta^{\prime}} \la\id_{H\setminus B_{R}^H}, \mu\ra^{1-\eta/\eta^{\prime}} \\
	&\leq |g_k(\cdot, h)|_{C_b(B_{R}^H)} + CR^{-1}
\end{split}
\end{align}
for all $\mu\in\P_{\eta^{\prime}, M}(H)$, where $C$ is a constant independent of $k, h, R$. Therefore by first taking $k\to\infty$ and then $R\to\infty$, we obtain the desired convergence \eqref{fd-covergence-eigenmeasure}.

\vskip0.05in 

The proof of Theorem \ref{fd-convergence} is complete. 
\end{proof}

Following the same idea, one can prove the following pullback stability. Note that  since we do not have uniform in $h\in\To$ convergence in Theorem \ref{fd-convergence}, the following theorem is not a direct consequence of it.  However, since in the pullback case we have a nice global monotonicity structure, the proof is actually simpler than that of Theorem \ref{fd-convergence}.

\begin{theorem}[Pullback stability of the eigen-triple]\label{prop-convergence-discrete-time}
Let $\eta\in (0, \eta_0)$ and $V\in C_b^{1,0}(H\times\To)$. For any $\phi\in C_{\mkm_{\eta}}(H\times\To)$, $M, R>0, h\in\To$, and $\eta^{\prime}\in (\eta, \eta_0)$ we have
\begin{align}\label{convergence-BH}
	&\lim_{k\to\infty}\left| \l_{k, \b_{-k}h, V}^{-1}P_{0, k, \b_{-k}h}^{V}\phi(\cdot, h) - \left\langle \phi(\cdot, h), \G^V(h)\right\rangle F^V(\cdot, \b_{-k}h)\right|_{C_b(B_{R}^H)} = 0,\\\label{convergence-Gh}
	&\lim_{k\to\infty}\left|  \l_{k, \b_{-k}h, V}^{-1}P_{0, k, \b_{-k}h}^{V}\phi(\cdot, h) - \left\langle \phi(\cdot, h), \G^V(h)\right\rangle F^V(\cdot, \b_{-k}h)\right|_{\G^V(\b_{-k}h)} = 0,\\\label{covergence-eigenmeasure}
	&\lim_{k\to\infty}\sup_{\mu\in\P_{\eta^{\prime}, M}(H)}\left|\la  \l_{k, \b_{-k}h, V}^{-1}P_{0, k, \b_{-k}h}^{V*}\mu , \phi(\cdot, h)\ra - \langle F^V(\cdot, \b_{-k}h), \mu\rangle \la\G^V(h), \phi(\cdot, h)\ra\right| = 0.
\end{align}
\end{theorem}
\begin{proof}
Replacing $h$ by $\b_{-k}h$ in \eqref{g_k-non-increasing}, we have 
\begin{align*}
	|g_k(\cdot, \b_{-k}h)|_{\G^V(\b_{-k}h)}\leq |g_{k-1}(\cdot, \b_{-(k-1)}h)|_{\G^V(\b_{-(k-1)}h)}, \, \text{ for all } k,
	h\in\To.
\end{align*}
This means for each fixed $h$, the sequence $|g_k(\cdot, \b_{-k}h)|_{\G^V(\b_{-k}h)}, \, k\in\mathbb{N}$ is non-increasing. This global non-expansive structure implies that if the convergence \eqref{convergence-Gh} does not hold, then there is a constant $a>0$ such that 
\[|g_k(\cdot, \b_{-k}h)|_{\G^V(\b_{-k}h)}\geq a\]
for all $k$.  The rest of the proof for Theorem \ref{prop-convergence-discrete-time} is similar to that of Theorem \ref{fd-convergence}, hence we omit the details. 
\end{proof}

\vskip0.05in

The following corollary states that uniformly for initial measures in 
\[\overline{\G}:=\lp\G^V(h), h\in\To\rp,\]
the trajectory of $S_k^V$ is attracted by $\G^V$, which is useful in proving the continuity of eigenmeasure. 

\begin{corollary}\label{St-to-Gamma}
For each $V\in C_b^{1,0}(H\times\To), \phi\in C_b(H)$, we have the following convergence 
\begin{align*}
\lim_{k\to\infty}\sup_{\mu\in\overline{\G}}\left|\frac{\la  P_{0, k, \b_{-k}h}^{V*}\mu, \phi\ra}{ P_{0, k, \b_{-k}h}^{V*}\mu(H)}-  \la\G^V(h), \phi\ra\right|=0,
\end{align*}
for every $h\in\To$. 
\end{corollary}
\begin{proof}
Let $\phi\in C_b(H)$ and $\eta\in(0, \eta_0)$.  Convergence \eqref{covergence-eigenmeasure} implies that 
\begin{align}\label{St-to-Gamma-1}
	I_1:= \left|\la  \l_{k, \b_{-k}h, V}^{-1}P_{0, k, \b_{-k}h}^{V*}\mu , \phi\ra - \langle F^V(\cdot, \b_{-k}h), \mu\rangle \la\G^V(h), \phi\ra\right| \to 0
\end{align}
uniformly for $\mu\in\P_{\eta, A}(H)$ as $k\to\infty$, where $A$ is from Proposition \ref{Existence-of-Eigenmeasure}. In particular, taking $\phi=\id_H$, then 
\begin{align}\label{St-to-Gamma-2}
	I_2:=\left| \l_{k, \b_{-k}h, V}^{-1}P_{0, k, \b_{-k}h}^{V*}\mu(H) - \la F^V(\cdot, \b_{-k}h), \mu\ra\right| \to 0
\end{align}
uniformly for $\mu\in\P_{\eta, A}(H)$ as well. Combining \eqref{St-to-Gamma-1} and \eqref{St-to-Gamma-2} with triangle inequality, we have 
\begin{align}\label{St-to-Gamma-3}
	\varepsilon_{k, h, \mu}
	:= \left|\frac{\la  P_{0, k, \b_{-k}h}^{V*}\mu, \phi\ra}{\l_{k, \b_{-k}h, V}}-  \frac{P_{0, k, \b_{-k}h}^{V*}\mu(H)}{\l_{k, \b_{-k}h, V}}\la\G^V(h),
	\phi\ra\right|\leq I_1+I_2\left|\la\G^V(h), \phi\ra\right| \to 0
\end{align}
uniformly for $\mu\in\P_{\eta, A}(H)$ as $k\to\infty$. 

\vskip0.05in

We now claim that $\l_{k, \b_{-k}h, V}^{-1}P_{0, k, \b_{-k}h}^{V*}\mu(H)$ can be bounded away from 0 for large $k$ through the convergence \eqref{St-to-Gamma-2} and equicontinuity of the eigenfunction family $F^V(\cdot, h), h\in\To$. Since $\G^V(h)$ is concentrated on $H_2$ for every $h$, by the normalizing equality \eqref{Normalization-property} and the fact $F^V\in L_{\mathfrak{m}_{q}}^{\infty}(H\times\To)$ from Proposition \ref{Existence-of-Eigenfunction}, we have for any $L>0$
\begin{align*}
	1 = \la F^V(\cdot, h), \G^V(h)\ra 
	& = \la F^V(\cdot, h)\id_{B_{L}^{H_2}}, \G^V(h)\ra + \la F^V(\cdot, h)\id_{H_2\setminus B_{L}^{H_2}}, \G^V(h)\ra\\
	&\leq \la F^V(\cdot, h)\id_{B_{L}^{H_2}}, \G^V(h)\ra + \la \mkm_q\id_{H_2\setminus B_{L}^{H_2}}, \G^V(h)\ra\\
	&\leq \la F^V(\cdot, h)\id_{B_{L}^{H_2}}, \G^V(h)\ra + CL^{-1}. 
\end{align*}
In the last inequality we used H\"older inequality, the concentration property \eqref{concentration-property} and $\G^V(h)$ integrability of $\mkm_q^2$, where $C$ is independent of $h$. Taking $L\geq L_0$ sufficiently large where $L_0$ is from \eqref{support-property-1}, we have 
\[\la F^V(\cdot, h)\id_{B_{L}^{H_2}}, \G^V(h)\ra\geq 1/2, \, h\in\To.\]
This implies the existence of $w_h\in B_{L}^{H_2}$ such that $F^V(w_h, h)\geq 1/2$. Proposition \ref{Existence-of-Eigenfunction} implies that the family $F^V(\cdot, h), h\in\To$ is uniformly equicontinuous on $(B_{L}^{H_2}, \dh)$. Hence there is a radius $r=r(L)$ independent of $h$ such that 
\begin{align*}
	F^V(w, h)\geq 1/4, \quad \text{ for all } w\in B_{r}^{H, B_{L}^{H_2}}(w_h), h\in\To. 
\end{align*}
Combining this with \eqref{support-property-1}, we have 
\begin{align}\label{lower-bound-eigenfunction-eigenmeasure}
\la F^V(\cdot, \b_{-k}h), \mu\ra\geq \frac14 \mu(B_{r}^{H, B_{L}^{H_2}}(w_h))\geq p/8, 
\end{align}
for any $\mu\in \overline{\G}$, $k\geq 0, h\in\To$, where $p=p(L)$ is independent of $h, k, \mu$.  In view of \eqref{St-to-Gamma-2}, we obtain the existence of a large $K$ (may depend on $h$)  such that 
\begin{align}
	\inf_{\mu\in\overline{\G}}\l_{k, \b_{-k}h, V}^{-1}P_{0, k, \b_{-k}h}^{V*}\mu(H)\geq p/16. 
\end{align}
for all $k\geq K$. Combining this with \eqref{St-to-Gamma-3} and the fact $\overline{\G}\subset \P_{\eta, A}(H)$, we obtain the desired convergence 
\begin{align*}
	\sup_{\mu\in\overline{\G}}\left|\frac{\la  P_{0, k, \b_{-k}h}^{V*}\mu, \phi\ra}{ P_{0, k, \b_{-k}h}^{V*}\mu(H)}-  \la\G^V(h), \phi\ra\right|
	&=\sup_{\mu\in\overline{\G}}\frac{\l_{k, \b_{-k}h, V}}{P_{0, k, \b_{-k}h}^{V*}\mu(H)}\varepsilon_{k, h, \mu}\\
	&\leq \frac{16}{p}\sup_{\mu\in\overline{\G}}\varepsilon_{k, h, \mu}\to 0,\, k\to\infty.
\end{align*}
The proof is complete. 
\end{proof}

\vskip0.05in

\subsection{Uniqueness and regularity of the eigen-triple}\label{subsection-uniqueness-regularity-discrete-time}

We now prove uniqueness of eigen-triple and corresponding regularity. The uniqueness of eigenmeasure and eigenvalue follows from the pullback convergence in Theorem \ref{prop-convergence-discrete-time}, while the uniqueness of eigenfunction follows from the forward convergence in Theorem \ref{fd-convergence}. The continuous dependence of eigenmeassure on time symbols $h\in\To$ follows from Corollary \ref{St-to-Gamma}, which implies continuity of eigenvalue $\l_{k, h, V}$. The regularity of eigenfunction $F^V$ on $w\in H$ follows from the forward convergence  \eqref{fd-convergence-BH} and uniform Lipschitz property from Proposition \ref{Uniform-Feller}.  Continuity of eigenfunction on time symbols $h\in\To$  will be proved in subsection \ref{subsection-proof-main-theorem-Feynman-Kac}. 
\begin{proposition}\label{proposition-uniqueness-eigen-data}
For $V\in C_b^{1,0}(H\times\To)$, the eigen-triple $(\G^V, \l_{k, h, V}, F^V)$ of the Feynman-Kac evolution operator $P_{0, k, h}^V$ is unique. In addition, 
\[h\to \G^V(h), \, \text{ and } h\to\l_{k, h, V},\]
are continuous for every $k$. Moreover, the family $\lp F^V(\cdot, h), h\in\To\rp$ is uniformly Lipschitz on $B_{R}^{H}$ for any $R>0$ with a uniform bound  $L=L(R, V)$ on the Lipschitz constants independent of $h\in\To$. 
\end{proposition}
\begin{proof}
The proof is divided into three steps. We first show uniqueness at the first step. Then we prove continuity of eigenmeasure and eigenvalue at the second step. The continuity of eigenfunction is given at the last step. 

\vskip0.05in

{\it Step 1: Uniqueness}. Suppose there is another eigen-triple $(\widetilde{\G}^V, \widetilde{\l}_{k, h, V}, \widetilde{F}^V)$. We show the uniqueness by first proving $\G^V= \widetilde{\G}^V$. 

\vskip0.05in

Noting $\widetilde{\G}^V$ must have the same support property as $\G^V$ in Proposition \ref{Existence-of-Eigenmeasure}.  This fact together with the convergence \eqref{convergence-BH} imply that 
\begin{align*}
\lim_{k\to\infty}|g_k(\cdot, \b_{-k}h)|_{\widetilde{\G}^V(\b_{-k}h)}=0, 
\end{align*}
where $g_k$ is given as in \eqref{def-gk}. Since $\left| \la g_k(\cdot, \b_{-k}h), \widetilde{\G}^V(\b_{-k}h)\ra\right|\leq |g_k(\cdot, \b_{-k}h)|_{\widetilde{\G}^V(\b_{-k}h)}$, it follows that 
\begin{align}\label{convergence-tilde-Gamma}
\lim_{k\to\infty}\la g_k(\cdot, \b_{-k}h), \widetilde{\G}^V(\b_{-k}h)\ra =0.
\end{align}
Taking in particular $\phi=\id_{H}$ in $g_k$, using the facts $\widetilde{\G}^V(H) = \G^V(H)=1$ for all $h\in\To$ and $P_{0, k, \b_{-k}h}^{V*}\widetilde{\G}^V(\b_{-k}h)=\widetilde{\l}_{k, \b_{-k}h, V}\widetilde{\G}^V(h)$, we have 
\begin{align}\label{convergence-1H-1}
\la g_k(\cdot, \b_{-k}h), \widetilde{\G}^V(\b_{-k}h)\ra = \frac{\widetilde{\l}_{k, \b_{-k}h, V}}{\l_{k, \b_{-k}h, V}}- \la F^V(\cdot, \b_{-k}h), \widetilde{\G}^V(\b_{-k}h) \ra \to 0, \quad k\to\infty, 
\end{align}
by \eqref{convergence-tilde-Gamma}. Similarly by exchanging the role of eigen-triple we have 
\begin{align}\label{convergence-1H-2}
\frac{\l_{k, \b_{-k}h, V}}{\widetilde{\l}_{k, \b_{-k}h, V}}- \la \widetilde{F}^V(\cdot, \b_{-k}h), \G^V(\b_{-k}h) \ra \to 0, \quad k\to\infty, 
\end{align}

\vskip0.05in

Observe that $F^V$ is strictly positive everywhere by Proposition \ref{Existence-of-Eigenfunction}. We claim that 
\[\liminf_{k\to\infty}\la F^V(\cdot, \b_{-k}h), \widetilde{\G}^V(\b_{-k}h) \ra>0.\]
Indeed if this is not true, then by \eqref{convergence-1H-1} we have 
\[\liminf_{k\to\infty} \frac{\widetilde{\l}_{k, \b_{-k}h, V}}{\l_{k, \b_{-k}h, V}} = 0. \] 
This contradicts with the limit \eqref{convergence-1H-2}, since as an eigenfunction $\widetilde{F}^V\in L_{\mathfrak{m}_{q}}^{\infty}(H\times\To)$ implies that $\la \widetilde{F}^V(\cdot, \b_{-k}h), \G^V(\b_{-k}h) \ra$ is uniformly bounded. Hence the claim is proved. 

\vskip0.05in

Combining the claim with \eqref{convergence-1H-1}, one has
\begin{align*}
\L_{k}:= \frac{\widetilde{\l}_{k, \b_{-k}h, V}}{\l_{k, \b_{-k}h, V}}\la F^V(\cdot, \b_{-k}h), \widetilde{\G}^V(\b_{-k}h) \ra^{-1} \to 1, \quad k\to\infty, 
\end{align*}
which together with \eqref{convergence-tilde-Gamma} imply that for every $\phi\in C_b(H)$, 
\begin{align*}
\left|\la \phi, \widetilde{\G}^V(h) \ra - \la \phi, \G^V(h) \ra \right| &= \lim_{k\to\infty}\left|\L_k\la \phi, \widetilde{\G}^V(h) \ra- \la\phi, \G^V(h) \ra\right| \\
&=\lim_{k\to\infty}\la F^V(\cdot, \b_{-k}h), \widetilde{\G}^V(\b_{-k}h) \ra^{-1}\left|\la g_k(\cdot, \b_{-k}h), \widetilde{\G}^V(\b_{-k}h)\ra\right| = 0. 
\end{align*}
Hence the eigenmeasure, as the fixed point of  the semigroup $S_k^V$ in Proposition \ref{Existence-of-Eigenmeasure},  is unique. The corresponding eigenvalue $\l_{k, h, V}$ is uniquely determined.  Furthermore, the eigenfunction $F^V$ satisfying the normalization condition $\la F^V(h), \G^V(h)\ra =1$ is also uniquely determined by the forward convergence in Theorem \ref{fd-convergence} by taking $\phi=\id_{H}$ in \eqref{fd-convergence-BH}. 

\vskip0.05in

{\it Step 2: Continuity of eigenmeasure and eigenvalue}. Let $\phi$ be any bounded Lipschitz function on $H$ and $h_0\in\To$. We would like to show 
\[\la \G^V(h_n), \phi\ra\to\la \G^V(h_0), \phi\ra, \, \text{ as long as } h_n\to h,\] 
by combining Corollary \ref{St-to-Gamma} and Lipschitz continuity of $h\to P_{0, k, h}^V\phi(w)$ modulo the Lyapunov function $\mkm_{\eta}$ and growth in $k$. 

\vskip0.05in

To be specific, define 
\[\g_{k}^{\varphi}(h_1, h_2) = \la P_{0, k, \b_{-k}h_1}^{V*}\G^V(\b_{-k}h_2), \varphi\ra, \, \varphi\in C_b(H), h_1, h_2\in\To. \]
Note that by Proposition \ref{Existence-of-Eigenmeasure}, one has
\[\g_{k}^{\id_H}(h, h)=\l_{k, \b_{-k}h, V}\,  \text{ and }\, \g_{k}^{\phi}(h, h) = \g_{k}^{\id_H}(h, h)\la\G^V(h), \phi\ra.\]
Also observe that 
\begin{align}\label{continuity-gamma-1}
\begin{split}
	&\left|\la \G^V(h_n), \phi\ra\ - \la \G^V(h_0), \phi\ra\right| \\
	&\leq\left| \la \G^V(h_n), \phi\ra\ - \frac{\g_{k}^{\phi}(h_0, h_n)}{\g_{k}^{\id_H}(h_0, h_n)}\right| + \left|\frac{\g_{k}^{\phi}(h_0, h_n)}{\g_{k}^{\id_H}(h_0, h_n)} - \la \G^V(h_0), \phi\ra\right|: = I_{1, k} + I_{2, k}. 
\end{split}
\end{align}
It follows from Corollary \ref{St-to-Gamma} that for any $\varepsilon>0$, there is a large $K=K(h_0)$ independent of $h_n$ such that for all $k\geq K$, 
\begin{align}\label{continuity-gamma-2}
I_{2, k}<\varepsilon/2.
\end{align}
Since $V\in C_b^{1,0}(H\times\To)$, estimate as in \eqref{Lipschitz-modulo-Lyapunov} and the fact $\G^V(h)\in\P_{\eta, A}(H)$ yield the existence of some modulus of continuity $\k$ and a constant $C=C(K, \phi, V)$ such that 
\begin{align*}
	\left|\g_{K}^{\phi}(h_n, h_n)-\g_{K}^{\phi}(h_0, h_n)\right| &= \left|\la P_{0, K, \b_{-K}h_n}^V\phi(w) -P_{0, K, \b_{-K}h_0}^V\phi(w), \G^V({\b_{-K}h_n})\ra \right|\\
	&\leq C\la\mkm_{\eta}, \G^V(\b_{-K}h_n)\ra(|h_n-h_0|+ \k(|h_n-h_0|))\\
	&\leq C(|h_n-h_0|+ \k(|h_n-h_0|)), 
\end{align*}
for any Lipschitz $\phi\in C_b(H)$, where we used the fact that $|\b_{-K}h_n-\b_{-K}h| = |h_n-h|$. Therefore by noting that $e^{km_V}\leq \g_{k}^{\id_{H}}(h_1, h_2)\leq e^{kM_V}$ and $|\g_{k}^{\phi}(h_1, h_2)|\leq \|\phi\|_{\infty}e^{kM_V}$ for any $h_1, h_2$, we have 
\begin{align*}
	I_{1, K}\leq &\frac{\big|\g_{K}^{\id_H}(h_0, h_n)-\g_{K}^{\id_H}(h_n, h_n)\big|\big|\g_{K}^{\phi}(h_n, h_n)\big|}{\g_{K}^{\id_H}(h_0, h_n)\g_{K}^{\id_H}(h_n, h_n)}+\frac{\left|\g_{K}^{\id_H}(h_n, h_n)\right|\left|\g_{K}^{\phi}(h_n, h_n) -\g_{K}^{\phi}(h_0, h_n)\right|}{\g_{K}^{\id_H}(h_0, h_n)\g_{K}^{\id_H}(h_n, h_n)}\\
	&\leq C(K,\phi, V)(|h_n-h_0|+ \k(|h_n-h_0|)).
\end{align*}
As a result, there is an $N = N(K, \varepsilon, \phi, V)$ such that for all $n\geq N$,  
\begin{align}\label{continuity-gamma-3}
I_{1, K} < \varepsilon/2. 
\end{align}
Hence $\G^V\in C(\To, \P_{\eta, A}(H))$ by combining \eqref{continuity-gamma-1}-\eqref{continuity-gamma-3} and arbitrariness of bounded Lipschitz function $\phi$. The continuity of $h\to\l_{k, h, V}$ follows from the fact that 
\begin{align*}
	\l_{k, h, V} = \la P_{0, k, h}^V\id_H, \G^V(h)\ra
\end{align*}
and continuity of $h\to\G^V(h)$. 

\vskip0.05in

{\it Step 3: Continuity of eigenfunction}. 
The uniform Lipschitz continuity of the eigenfunction family follows from the convergence \eqref{fd-convergence-BH} with $\phi = \id_{H}$, the uniform Lipschitz property from Proposition \ref{Uniform-Feller} with $\phi = \id_{H}$, and the uniform boundedness \eqref{Bound-summand-Fk}. 

\vskip0.05in

The proof of Proposition \ref{proposition-uniqueness-eigen-data} is complete. 
\end{proof}
\subsection{Proof of Theorem \ref{asymptotic-behavior-feynman-kac}}\label{subsection-proof-main-theorem-Feynman-Kac}
\begin{proof}[Proof of Theorem \ref{asymptotic-behavior-feynman-kac}]The proof is divided into three steps. We first show the uniqueness of eigen-triple and then prove their continuity. Finally we  prove the forward and pullback stability  properties. 

\vskip0.05in

{\it Step 1: Uniqueness of eigen-triple for the continuous time family $P^V_{0, t, h}$}. Let $\l_{t, h, V} = P_{0, t, h}^{V*}\G^V(h)(H)$. We first show that  $\G^V$ is the unique eigenmeasureof the continuous time Feynman-Kac evolution operators $P_{0, t, h}^V$ with eigenvalue $\l_{t, h, V}$.  For any $t_0>0$, we apply the process in the previous three subsections to the discrete time Feynman-Kac evolution operators $P_{0, kt_0, h}^V, k\geq 0$, which yields the existence of unique eigen-triple $(\G_{t_0}^V, \widetilde{\l}_{kt_0, h, V}, F_{t_0}^V)$. Note that $\G_{t_0}^V$ is the unique fixed point of the operator $S_{t_0}^V$ as in \eqref{the-semigroup}, i.e.,
\[S_{t_0}^V\G_{t_0}^V=\G_{t_0}^V.\]
By the semigroup property of $S_t^V$, we have 
\begin{align*}
S_1^V(S_{t_0}^V\G^V)=S_{t_0}^V(S_1^V\G^V) = S_{t_0}^V\G^V, 
\end{align*}
implying  that $S_{t_0}^V\G^V$ is a fixed point of $S_1^V$, which is unique by Proposition \ref{proposition-uniqueness-eigen-data}. Hence $S_{t_0}^V\G^V=\G^V$,  which in turn shows that $\G^V$ is a fixed point of $S_{t_0}^V$.  Therefore by uniqueness 
\[\G_{t_0}^V=\G^V, \quad \text{and } \widetilde{\l}_{t_0, h, V} = \l_{t_0, h, V}.\]  
Since $t_0$ is arbitrary, we obtain that $\G^V$, as the unique fixed point of the semigroup $S_t^V$, is the unique eigenmeasure of $P_{0, t, h}^V$ with eigenvalue $\l_{t, h, V}$. 

\vskip0.05in

The uniqueness of eigenfunction is similar. Indeed, let $\widetilde{F}(w, h) = \l_{t_0, h, V}^{-1}P_{0, t_0, h}^VF^V(w, \b_{t_0}h)$. Then by using eigenfunction property of $F^V$ and cocycle property of $\l_{t, h, V}$,  one can verify in a straightforward manner that 
\[\la \widetilde{F}(h), \G^V(h)\ra=1, \text{ and } P_{0, k, h}^V\widetilde{F}(w, \b_kh) = \l_{k, h, V}\widetilde{F}(w, h), \quad (w, h)\in H\times\To.\]
By uniqueness of eigenfunction from Proposition \ref{proposition-uniqueness-eigen-data}, one has $\widetilde{F}=F^V$. Hence 
\[P_{0, t_0, h}^VF^V(w, \b_{t_0}h)=\l_{t_0, h, V}F^V(w, h),\]
which means that $F^V$ is an eigenfunction of the family $\{P_{0, kt_0, h}^V\}$. Therefore $F^V = F_{t_0}^V$ by uniqueness. 

\vskip0.05in

With the invariance (eigenmeasure property) of $\G^V$ in the whole time horizon, we are now able to show the integral representation \eqref{eigen-value-integral-representation} for the eigenvalue through a differentiation procedure. Indeed, since $r\to H_r$ is continuous almost surely,  one has  
\begin{align*}
	\frac{d}{dt}\l_{t, h, V} 
	&= \frac{d}{dt}\int_{H}\E_{(w, h)}\left[e^{\int_0^tV(\H_r)dr}\right]\G^V(h)(dw)\\
	&=\int_{H}\E_{(w, h)}\left[e^{\int_0^tV(\H_r)dr}V(H_t)\right]\G^V(h)(dw)\\
	&= \la V(\cdot, \b_th), P^{V*}_{0, t, h}\G^V(h)\ra = \l_{t, h, V}\la V(\cdot, \b_th), \G^V(\b_th)\ra. 
\end{align*}
Since $\l_{0, h, V}=1$, by integrating the above equation we obtain the desired representation
\begin{align*}
	\l_{t, h, V}  = \exp\left(\int_0^t\l^V(\b_rh)dr\right), \, \l^V(h) = \la V(\cdot, h), \G^V(h)\ra. 
\end{align*}

\vskip0.05in

The proof of {\it Step 1} is complete. 

\vskip0.05in

{\it Step 2: Continuity of the eigen-triple}. It follows from Proposition \ref{proposition-uniqueness-eigen-data} that 
$\G^V\in C(\To, \P(H))$. The continuity  of eigenvalue follows from the fact that $V\in C_b^{1, 0}(H\times\To)$ and the continuity of eigenmeasure once we note that 
\begin{align*}
	\left|\l^V(h_n) - \l^V(h)\right| &=\left|\la V(\cdot, h_n), \G^V(h_n)\ra - \la V(\cdot, h), \G^V(h)\ra\right|\\
	&\leq \left|\la V(\cdot, h_n), \G^V(h_n)\ra - \la V(\cdot, h), \G^V(h_n)\ra\right|+\left|\la V(\cdot, h), \G^V(h_n)\ra - \la V(\cdot, h), \G^V(h)\ra\right|\\
	&\leq C\k(|h_n-h|)+\left|\la V(\cdot, h), \G^V(h_n)\ra - \la V(\cdot, h), \G^V(h)\ra\right|\to 0, 
\end{align*}
as long as $h_n\to h$, where $\k$ is some modulus of continuity coming from the fact $V\in C_b^{1, 0}(H\times\To)$. 

\vskip0.05in

Since $\lp F^V(\cdot, h), h\in\To\rp$ is uniformly Lipschitz on any bounded ball in $H$ by Proposition \ref{proposition-uniqueness-eigen-data},  we only need to show that $h\to F^V(w, h)$ is continuous for every $w\in H$, which will imply that $F^V\in C(H\times\To)$. To that end,  it suffices to show that 
\begin{align}\label{time-continuity-eigenfunction-1}
	t\in \R\to F^V(w, \b_th) \, \text{ is  continuous for every } (w, h)\in H\times\To. 
\end{align}
Indeed, since $\To = [0, 2\pi)$ and $\b_{t}h=h+t\mod 2\pi$, we have $h=\b_h0$ for $h\in [0, 2\pi)$. Hence
\begin{align*}
	h\to F^V(w, h) = F^V(w, \b_h0),
\end{align*}
which is continuous by time continuity \eqref{time-continuity-eigenfunction-1}. 

\vskip0.05in

Therefore it remains to show \eqref{time-continuity-eigenfunction-1}. By the eigenfunction property, we have 
\begin{align}\label{time-continuity-eigenfunction-property-1}
	P_{0, t, \b_{-t}h}^VF^V(\cdot, h)(w) = \l_{t, \b_{-t}h, V}F^V(w, \b_{-t}h), \, \forall h\in\To, t\geq 0. 
\end{align}
By the integral representation \eqref{eigen-value-integral-representation}, we see that $\l_{t, \b_{-t}h, V}$ is continuous in time. Hence we only need to show that $P_{0, t, \b_{-t}h}^VF^V(\cdot, h)(w)$ is time continuous. 

\vskip0.05in

Let $\tau\geq 0$, for notational simplicity denoting $F(w) = F^V(w, h), w\in H$ and 
\begin{align*}
	\Xi_{t, h}^V :=\exp\left( \int_0^tV(\H_r(w, h))dr\right) = \exp\left( \int_0^tV(\Phi_{0, r, h}(w), \b_rh)dr\right). 
\end{align*}
Note that with $\E = \E_{w}$, 
\begin{align}\label{time-continuity-1-1}
\begin{split}
	&P_{0, t, \b_{-t}h}^VF^V(\cdot, h)(w) - P_{0, \tau, \b_{-\tau}h}^VF^V(\cdot, h)(w)\\
	&=\E\Xi_{t, \b_{-t}h}^VF(\Phi_{0, t, \b_{-t}h})  - \E\Xi_{\tau, \b_{-\tau}h}^VF(\Phi_{0, \tau, \b_{-\tau}h}) \\
	&=\E\left[(\Xi_{t, \b_{-t}h}^V-\Xi_{\tau, \b_{-\tau}h}^V)F(\Phi_{0, t, \b_{-t}h})\right]  + \E\left[\Xi_{\tau, \b_{-\tau}h}^V(F(\Phi_{0, t, \b_{-t}h}) -F(\Phi_{0, \tau, \b_{-\tau}h}))\right]\\
	&:=I_1 + I_2. 
\end{split}
\end{align}
Since $F^V\in L^{\infty}_{\mkm_q}(H\times\To)$ and is strictly positive everywhere by Proposition \ref{Existence-of-Eigenfunction}, one has 
\begin{align*}
	|I_1|&\leq \E\left[\left|\Xi_{t, \b_{-t}h}^V-\Xi_{\tau, \b_{-t}h}^V\right|F(\Phi_{0, t, \b_{-t}h})\right]+\E\left[\left|\Xi_{\tau, \b_{-t}h}^V-\Xi_{\tau, \b_{-\tau}h}^V\right|F(\Phi_{0, t, \b_{-t}h})\right]\\
	&\leq C\left(\E\left[\left|\int_{\tau}^tV(\H_r(w, \b_{-t}h)) dr\right|\mkm_q\right]+ \E\left[\int_{0}^{\tau}\left|V(\H_r(w, \b_{-t}h))-V(\H_r(w, \b_{-\tau}h))\right| dr\mkm_q\right]\right)\\
	&\leq C\left(|t - \tau| \E\mkm_q + \k(|\b_{-t}h -\b_{-\tau}h |) \E\mkm_q  + \int_0^{\tau}\E\left[\mkm_q\|\Phi_{0, r, \b_{-t}h}-\Phi_{0, r, \b_{-\tau}h}\|\right] dr\right)\\
	&\leq C\left(|t - \tau| +\k(|\b_{-t}h -\b_{-\tau}h |) + |\b_{-t}h -\b_{-\tau}h | \right)
\end{align*}
for some modulus of continuity $\k$ coming from the fact $V\in C_b^{1, 0}(H\times\To)$ and the constant $C=C(t, \tau, V)>0$. Here in the second inequality we used the boundedness of $V, t, \tau$ and the Lipschitz continuity of the exponential function $e^x$ on some bounded interval determined by $m_V, M_V, t, \tau$.  In the third inequality we used the fact $V\in C_b^{1, 0}(H\times\To)$. In the last step we used H\"older's inequality and \eqref{contraction-polynomial}, and \eqref{continuousonhull} together with the Lipschitz continuity of $f$. Hence $I_1\to 0$ as $t\to\tau$. 

\vskip0.05in

Now observe that 
\begin{align*}
	|I_2|&\leq e^{\tau\|V\|_{\infty}}\E\left|F(\Phi_{0, t, \b_{-t}h}) -F(\Phi_{0, \tau, \b_{-\tau}h})\right|\\
	&\leq C\left(\E\left|F(\Phi_{0, t, \b_{-t}h}) -F(\Phi_{0, t, \b_{-\tau}h})\right|+\E\left|F(\Phi_{0, t, \b_{-\tau}h}) -F(\Phi_{0, \tau, \b_{-\tau}h})\right|\right)\\
	&:=C(I_{2,1}+I_{2,2}). 
\end{align*}
Let $\O_{R} = \lp\o\in \O: \Phi_{0, t, \b_{-t}h}, \Phi_{0, t, \b_{-\tau}h}\in B_{R}^H\rp$. It follows from Proposition \ref{proposition-uniqueness-eigen-data} that $F$ is Lipschitz continuous on $B_{R}^H$ with some Lipschitz constant $L(R)$. Therefore by $F^V\in L^{\infty}_{\mkm_q}(H\times\To)$,  and \eqref{continuousonhull} together with the Lipschitz continuity of $f$, and noting that 
\[1\leq R^{-2q}\left(\|\Phi_{0, t, \b_{-t}h}\|^{2q}+\|\Phi_{0, t, \b_{-\tau}h}\|^{2q}\right)\leq R^{-2q}\left(\mkm_q(\Phi_{0, t, \b_{-t}h}) + \mkm_q(\Phi_{0, t, \b_{-\tau}h})\right), \text{ on }\O_{R}^c, \]
one obtains 
\begin{align*}
	I_{2,1} &= \E\left[\id_{\O_R}\left|F(\Phi_{0, t, \b_{-t}h}) -F(\Phi_{0, t, \b_{-\tau}h})\right|\right] + \E\left[\id_{\O_R^c}\left|F(\Phi_{0, t, \b_{-t}h}) -F(\Phi_{0, t, \b_{-\tau}h})\right|\right]\\
	&\leq L(R)\E\left\|\Phi_{0, t, \b_{-t}h} -\Phi_{0, t, \b_{-\tau}h}\right\| + C\E\left[\id_{\O_R^c}\left(\mkm_q(\Phi_{0, t, \b_{-t}h}) + \mkm_q(\Phi_{0, t, \b_{-\tau}h})\right)\right]\\
	&\leq C\Big(L(R)|\b_{-t}h -\b_{-\tau}h|+R^{-2q}\E\left(\mkm_{2q}(\Phi_{0, t, \b_{-t}h}) + \mkm_{2q}(\Phi_{0, t, \b_{-\tau}h})\right)\Big)\\
	&\leq C\Big(L(R)|\b_{-t}h -\b_{-\tau}h|+R^{-2q}\mkm_{2q}(w)\Big),
\end{align*}
where we used \eqref{contraction-polynomial} at the last step. Therefore by first choosing $R$ large and then letting $t\to\tau$, we have $I_{2,1}\to 0 $. 

\vskip0.05in

Since each sample path of the solution $\Phi_{0, t, \b_{-\tau}h}, t\geq 0$ is  continuous in $t$, and $F^V\in L^{\infty}_{\mkm_q}(H\times\To)$ implies 
\[ \E\left|F(\Phi_{0, t, \b_{-\tau}h}) -F(\Phi_{0, \tau, \b_{-\tau}h})\right| \leq \E\left(\mkm_q(\Phi_{0, t, \b_{-\tau}h}) + \mkm_q(\Phi_{0, \tau, \b_{-\tau}h})\right)\leq C\mkm_q(w), \, t\geq0, \]
therefore by Lebesgue's dominated convergence theorem, we have $I_{2, 2}\to 0$ as $t\to\tau$. As a result $I_2\to 0$ as $t\to\tau$.

\vskip0.05in

In view of \eqref{time-continuity-1-1} we obtain the time continuity of $P_{0, t, \b_{-t}h}^VF^V(\cdot, h)(w)$. Therefore by \eqref{time-continuity-eigenfunction-property-1}, the function 
\begin{align}\label{time-continuity-2-1}
	t\to F^V(w, \b_{-t}h) =  \l_{t, \b_{-t}h, V}^{-1}P_{0, t, \b_{-t}h}^VF^V(\cdot, h)(w), \, t\geq 0
\end{align}
is continuous for every $(w, h)\in H\times\To$. The continuity \eqref{time-continuity-eigenfunction-1} for any $t\in\R$ follows by shifting (replacing $h$ by $\b_sh$ for any $s>0$ in \eqref{time-continuity-2-1}). 

\vskip0.05in

The proof of {\it Step 2} is complete. 

\vskip0.05in

{\it Step 3: Proof of the stability properties of Theorem \ref{asymptotic-behavior-feynman-kac}}. Let $g_t(w, h)$ be the function obtained from \eqref{def-gk} by replacing $k$ with $t$. The proof of the forward convergence in this case is exactly the same as the proof of Theorem \ref{fd-convergence}. Namely, we first show as in Proposition \ref{proof-stability-L1}  the convergence of $g_t$ for $\phi\in C_b^1(H\times\To)$ in $L^1(\G^V(h))$ by utilizing the unique eigen-triple and the local monotonicity \eqref{g_k-non-increasing} which is also valid for continuous time. Then we proceed as in subsection \ref{subsection-proof-forward-stability}. 

\vskip0.05in

The pullback convergence in $L^1(\G^V(\b_{-t}h))$ follows from the convergence \eqref{convergence-Gh} in discrete time and the global monotonicity of $t\to \la|g_t(\cdot, \b_{-t}h)|, \G^V(\b_{-t}h)\ra$, where the later follows  by the semigroup property of $\P_t^V$, and cocycle property of the eigenvalue, as well as the eigenmeasure property.  The other pullback stability  properties in Theorem \ref{asymptotic-behavior-feynman-kac} then follow from  convergence in $L^1(\G^V(\b_{-t}h))$ and similar arguments as in subsection \ref{subsection-proof-forward-stability}. 

\vskip0.05in

The proof of Theorem \ref{asymptotic-behavior-feynman-kac} is complete. 

\end{proof}

\section{A Donsker-Varadhan type large deviation principle}\label{Proof-main-theorems}

This section is devoted to the proof of Theorem \ref{LDP}. We prove this theorem in the first subsection \ref{subsection-proof-theorem-ldp}. Then in subsection \ref{subsection-proof-zero-invariant-measure} we prove several lemmas that are used in the proof of Theorem \ref{LDP}. 

\subsection{Proof of Theorem \ref{LDP}}\label{subsection-proof-theorem-ldp}
In this subsection, we prove the main result on the large deviation bounds in Theorem \ref{LDP} by combining the results of Theorem \ref{asymptotic-behavior-feynman-kac} and the Kifer type criterion in Theorem \ref{Kifer-Type-Criterion} from the Appendix, where the later was established in \cite{JNPS18}. 

\vskip0.05in

More specifically, we first give a proof of  Theorem \ref{LDP}  in subsection \ref{subsection-proof-ldp} by assuming the following Theorem \ref{Kifer-condition}, which gives a verification of the conditions in Kifer's Theorem \ref{Kifer-Type-Criterion}. The proof of Theorem \ref{Kifer-condition} is given in subsection \ref{subsection-proof-verify-kifer}, which is  based on Theorem \ref{asymptotic-behavior-feynman-kac}. 
\begin{theorem}\label{Kifer-condition}
Assume $A_{\infty} = H$ and $f\in C_b(\R, H_2)$ is Lipschitz.  Let $\eta\in(0, \eta_0)$ and $M>0$. 
\begin{enumerate}
\item For any $V\in C_b(H\times\To)$, the following limit called the pressure function exists
\begin{align}\label{LDP-condition-1}
	Q(V) := \lim_{t\to\infty}\frac{1}{t}\sup_{\s\in \P_{\eta, M}(H)}\log\E_{\s}\exp\la tV, \zeta_{t, h}\ra,
\end{align}
and independent of $h, \eta, M$. Also $Q(V) $ is a convex $1$-Lipschitz function such that $Q(C) =C$ for any $C\in\mathbb{R}$. In particular for $V\in C_b^{1,0}(H\times\To)$, 
\begin{align}\label{Q(V)-when-V-regular}
	Q(V) = \int_{\To}\l^V(h)m(dh), 
\end{align}
where  $\l^V\in C(\To, \R)$ is the eigenvalue obtained in Theorem \ref{asymptotic-behavior-feynman-kac}. 

\item There is a function $R: H\times\To\to [0, \infty]$ with compact level sets $\{(w, h)\in H\times\To: R(w, h)\leq \alpha\}$ for any $\alpha\geq 0$ such that 
\begin{align}\label{exponential-tightness}
\E_{\s}\exp\la tR, \zeta_{t, h}\ra\leq Ce^{ct}, \quad \s\in \P_{\eta, M}(H), 
\end{align}
for some positive constants $c, C$. 

\item For each $V\in C_b^{1,0}(H\times\To)$,  there is a unique equilibrium state $\mu_V\in\mathcal{P}(H\times\To)$, i.e., 
\begin{align}\label{definition-equilibrium-state}
Q(V) = \langle V, \mu_V\rangle - I(\mu_V).
\end{align}
Here $I$ is the Legendre transform of $Q$ as in \eqref{rate-function}.

\end{enumerate}
\end{theorem}

\subsubsection{Proof of Theorem \ref{LDP}}\label{subsection-proof-ldp}

\begin{proof}The proof is divided into three steps. 

\vskip0.05in

{\it Step 1: Link Theorem \ref{Kifer-condition} to Theorem \ref{Kifer-Type-Criterion}}. Note that the three conditions of Kifer's Theorem \ref{Kifer-Type-Criterion} are verified by  Theorem \ref{Kifer-condition}, where we take the directed set 
\[\Theta=(0, \infty)\times \P_{\eta, M}(H), \]
equipped with an order $\prec$ defined by $(t_1, \s_1)\prec(t_2, \s_2)$ if $t_1\leq t_2$.  And  for each fixed $h\in\To$, we set 
\begin{align*}
	\z_{\theta} = \zeta_{t, h}, \quad r(\theta)= t, \, \text{ for every } \theta = (t, \s)\in\Theta, 
\end{align*}
where $\z_{t, h}$ is the occupation measure defined as \eqref{occupation-measure}, which is a family of random probability measures defined on the probability space $(\O_{\theta}, \mathcal{F}_{\theta}, \mathbf{P}_{\theta}) = (\O, \mathcal{F}, \mathbf{P}_{\s})$.  Hence  Theorem \ref{Kifer-Type-Criterion} implies the desired large deviation bounds with the good rate function 
\begin{align}\label{rate-function-proof}
	I(\mu)=
	\left\{\begin{array}{cc}
		\ds\sup_{V\in C_b(H\times\To)}\left(\la V, \mu\ra - Q(V)\right), &\mu\in\P(H\times\To),\\
		\qquad+\infty, &\text{otherwise.}
	\end{array}\right.
\end{align}
It remains to show that the good rate function is nontrivial. 

\vskip0.05in

{\it Step 2: Prove the non-triviality of the rate function}.  Assume by contradiction that $I(\g) = 0$ and $I(\mu) =+\infty$ for $\mu\in \P_{m}(H\times\To)\setminus\{\g\}$, where $\g=\G(h)m(dh)$ is the unique invariant measure of the homogenized process $\mathcal{H}_t$, and $\G\in C(\To, \P(H))$ is the continuous periodic invariant measure of $\Phi_{0, t, h}$\cite{LL22}. Note that $\G$ is also the unique eigenmeasure of $P_{0, t, h}^V$ in the case when the potential $V=0$. Since $I$ is the Legendre transform of $Q$, it follows from \eqref{rate-function-proof} that 
\begin{align*}
	Q(V) = \sup_{\mu\in\P(H\times\To)}\left(\la V, \mu\ra-I(\mu)\right),  \quad V\in C_b(H\times\To).
\end{align*}
The assumption of triviality of $I$ implies that 
\begin{align}\label{nontrivial-rate-function-1}
	Q(V) = \la V, \g\ra \text{ for any }V\in C_b(H\times\To). 
\end{align}

\vskip0.05in

To each $V\in C_b^1(H)$ we associate 
\[V_{\g}(w, h):= V(w) - \la V, \G(h)\ra.\]
It has been shown in \cite{LL22} that $\G\in C^{r}(\To, \P(H))$ for some $r\in (0, 1)$, where $\P(H)$ is equipped with a Wasserstein metric weighted by the Lyapunov function $\mkm_{\eta}$. Hence $V_{\g}\in C_{b}^{1, 0}(H\times\To)$ for each $V\in C_b^1(H)$.  Since $\la  V_{\g}, \g\ra = 0$, by \eqref{Q(V)-when-V-regular} and \eqref{nontrivial-rate-function-1}, we have 
\begin{align}\label{nontrivial-rate-function-2}
	\int_{\To}\l^{V_{\g}}(h)m(dh) = 0, \, V\in C_b^1(H). 
\end{align}
Combining this with the integral representation \eqref{eigen-value-integral-representation} and Birkhoff's ergodic theorem, we have 
\begin{align}\label{nontrivial-rate-function-3}
	\lim_{t\to\infty}\frac1t\log\l_{t, h, V_{\g}} = \lim_{t\to\infty}\frac1t\int_0^t\l^{V_{\g}}(\b_{rh})dr = \int_{\To}\l^{V_{\g}}(h)m(dh) = 0, 
\end{align}
which implies that $\l_{t, h, V_{\g}}$ has a subexponential growth in $t$.

\vskip0.05in

The forward convergence \eqref{main-thereom-forward-mxing} in Theorem \ref{asymptotic-behavior-feynman-kac} with $\phi=\id_{H}$ and $\mu=\d_{w}$ implies that 
\begin{align}\label{nontrivial-rate-function-4}
	\lim_{t\to\infty}\l_{t, h, V_{\g}}^{-1} P^{V_{\g}}_{0, t, h}\id_H(w) = F^{V_{\g}}(w, h), \, w\in H, 
\end{align}
where we note that 
\begin{align}\label{nontrivial-rate-function-5}
	P^{V_{\g}}_{0, t, h}\id_H(w)= \mathbf{E}_{w}\exp\left(\int_0^t\big(V(\Phi_{0, r, h}) - \la V, \G(\b_rh)\ra\big) dr\right). 
\end{align}
On the other hand, the central limit theorem proved in  \cite{LL22} shows that 
\begin{align*}
	\frac{1}{\sqrt{t}}\int_0^t\big(V(\Phi_{0, r, h}) - \la V, \G(\b_rh)\ra\big) dr \text{ is asymptotically Gaussian with variance } \s_V^2. 
\end{align*}
Hence if one can show the existence of $V\in C_b^1(H)$ such that $\s_V^2\neq 0$, then \eqref{nontrivial-rate-function-5} will imply that $P^{V_{\g}}_{0, t, h}\id_H(w)$ grows exponentially fast in $t$, which will contradicts with the subexponential growth \eqref{nontrivial-rate-function-3} and the convergence \eqref{nontrivial-rate-function-4}.  This contradiction then shows that the rate function $I$ is nontrivial. 

\vskip0.05in

{\it Step 3: Prove the existence of an observable function with nonzero asymptotic variance.} The rest of the proof is devoted to show the existence of such a function $V$ that makes $\s_V^2\neq 0$ through a contradiction argument as in \cite{KS12} Proposition 4.1.4. For any $R>0$, we consider a $C_b^1$ function $V_{R}: H\to \R$ defined as follows: $0\leq V_{R}\leq 1$ and 
\begin{align}\label{define-V_R}
	V_{R}(w) = 1, \, \text{ when } w\in B_{R}^H, \quad \text{ and } V_{R}(w) = 0, \, \text{ when } w\in H\setminus B_{R+1}^H. 
\end{align}
Since $\G$ is the unique eigenmeasure corresponding to $V=0$, it also has the support property as in Proposition \ref{Existence-of-Eigenmeasure}. Hence there is a $0<p=p(R)<1$ does not depend on $h\in\To$ such that 
\begin{align*}
	\G(h)(B_{R+1}^H)\leq p<1, \, h\in\To, 
\end{align*}
which implies $ \la V_{R}, \G(h)\ra\leq p$. As a result, one has 
\begin{align}\label{lower-bound-V_R}
	V_{\g, R}(w, h):= V_{R}(w) - \la V_{R}, \G(h)\ra\geq 1-p>0, \, \forall w\in B_{R}^H, h\in\To.
\end{align}
 
\vskip0.05in

Assume by contradiction that  $\s_{V_{\g, R}}^2=0$.  Recall the martingale approximation as in \cite{LL22}, 
\begin{align}\label{nontrivial-rate-function-8}
	\int_{0}^NV_{\g, R}(\H_r)dr = M_N - \chi(\H_N) + \chi(\H_0), 
\end{align}
where $M_N  =\int_{0}^NV_{\g, R}(\H_r)dr +\chi(\H_N) - \chi(\H_0) $ is the Dynkin martingale and the corrector $\chi$ is defined as (Note that $\P_{t} = \P_{t}^V$  with $V=0$)
\[\chi(w, h) = \int_{0}^{\infty}\P_{t}V_{\g, R}(w, h)dr, \, (w, h)\in H\times\To, \]
which is bounded modulo $e^{\eta\|w\|^2}$, see \cite{LL22}. We claim that the assumption $\s_{V_{\g, R}}^2=0$ implies
\begin{align}\label{nontrivial-rate-function-7}
	\mathbf{P}_{\g}\lp M_N =0,\, \forall N\geq 0 \rp=1. 
\end{align}
Indeed, by boundedness of $\chi$ and integrability of  $e^{\eta\|w\|^2}$, we have 
\begin{align*}
	\E_{\g}M_N^2 &\leq 3\E_{\g}\left(\int_{0}^NV_{\g, R}(\H_r)dr\right)^2 + 3\E_{\g}|\chi(\H_N)|+ 3\E_{\g}|\chi(\H_0)|\\
	&\leq C\left(\E_{\g}\left(\int_{0}^NV_{\g, R}(\H_r)dr\right)^2+1\right).
\end{align*}
Therefore by Proposition 4.12 in \cite{LL22} on the asymptotic variance $\s_{V_{\g, R}}^2$, we have 
\begin{align}\label{nontrivial-rate-function-6}
	\lim_{N\to\infty}\frac1N\E_{\g}M_N^2 = 0, \text{ since } \lim_{N\to\infty}\frac1N\E_{\g}\left(\int_{0}^NV_{\g, R}(\H_r)dr\right)^2 = \s_{V_{\g, R}}^2 =0. 
\end{align}
Since $\g$ is the invariant measure of the homogenized process, for the martingale difference sequence $Z_k = M_k-M_{k-1}$, one has $\E_{\g}Z_k^2 = \E_{\g}Z_1^2$ for all $k\geq 1$ and  
\begin{align*}
	\E_{\g}M_N^2 = \E_{\g}\left(\sum_{k=1}^NZ_k\right)^2 = \E_{\g}\sum_{k=1}^NZ_k^2 = N\E_{\g}Z_1^2. 
\end{align*}
Combining this with \eqref{nontrivial-rate-function-6} we obtain $\E_{\g}Z_k^2=0$ for all $k$. Hence $M_N=0$ almost surely for all $N$ and the claim \eqref{nontrivial-rate-function-7} is proved. 

\vskip0.05in

We are now ready to derive a contradiction. Note that by deterministic version of the estimate \eqref{contraction-polynomial}, there is a constant $C>1$ independent of $h\in\To$ such that the deterministic solution $\Phi_{0, t, h}^d(w)$ to equation \eqref{Settings-NS-family} without noise satisfies (see also \cite{CV02a})
\begin{align*}
	\|\Phi_{0, t, h}^d(w)\|^2\leq e^{-\nu t}\|w\|^2 + C, \, \forall w\in H, h\in\To. 
\end{align*}
Hence by choosing $R=2C$ in the definition \eqref{define-V_R} of $V_{R}$, we have that for any initial condition $(w, h)\in B_1^H\times \To$, the deterministic solution $\Phi_{0, t, h}^d(w)\in B_{R}^H$ for all $t\geq 0$. Therefore for any initial $(w, h)\in B_1^H\times \To$ and $N\geq 0$, there is a small positive probability (corresponding to those Brownian sample paths in a small neighborhood of the origin), such that the solution $\Phi_{0, t, h}(w)$ to the noisy system \eqref{Settings-NS-family} always stays in $B_{R}^H$ for $t\in [0, N]$. Namely, we have 
\begin{align*}
	\mathbf{P}_{(w, h)}\lp \H_r\in B_{R}^H\times\To, \text{ for } 0\leq r\leq N \rp>0, \, (w, h)\in B_1^H\times \To. 
\end{align*}
Since $\g(B_1^H\times \To)>0$, one has 
\begin{align*}
	\mathbf{P}_{\g}\lp \H_r\in B_{R}^H\times\To, \text{ for } 0\leq r\leq N \rp>0. 
\end{align*}
In view of \eqref{lower-bound-V_R}, we obtain for any $N\geq 0$, 
\begin{align*}
	\mathbf{P}_{\g}\lp \int_{0}^NV_{\g, R}(\H_r)dr \geq N(1-p)\rp>0.
\end{align*}
This combined with \eqref{nontrivial-rate-function-7}, and boundedness of the corrector $\chi$ over $B_{R}^H\times\To$, imply that \eqref{nontrivial-rate-function-8} cannot hold with $\mathbf{P}_{\g}$ probability $1$. This contradiction shows that  $\s_{V_{\g, R}}^2\neq0$. 

\vskip0.05in 

The proof is complete. 
\end{proof}

\vskip0.05in 

\subsubsection{Proof of Theorem \ref{Kifer-condition}}\label{subsection-proof-verify-kifer}
We now supply a proof of Theorem \ref{Kifer-condition} based on the results of Theorem \ref{asymptotic-behavior-feynman-kac}. 

\vskip0.05in 

{\it Proof of Items (1) and (2) in Theorem \ref{Kifer-condition}}. 
Note that Item (2) in Theorem \ref{Kifer-condition} directly follows from \eqref{enstrophy} with $R(w, h) = \eta\nu\|w\|_1^2$. The level set $\{(w, h)\in H\times\To: R(w, h)\leq \alpha\}$ for any $\alpha\geq 0$ is compact since $\To$ is compact and the embedding $H_1\to H$ is compact. 

\vskip0.05in 

Item (1) of Theorem \ref{Kifer-condition} follows from the forward convergence in Theorem \ref{asymptotic-behavior-feynman-kac}, together with Birkhoff's ergodic theorem and a $buc$-approximation argument. Indeed, taking $\phi = \id_H$ in the forward convergence \eqref{main-thereom-forward-mxing} in Theorem \ref{asymptotic-behavior-feynman-kac}, and noting the strict positivity of eigenfunction $F^V$,  we have for $V\in C_b^{1,0}(H\times\To)$, 
\begin{align}\label{item-1-1}
	\lim_{t\to\infty}\left|\frac1t\log\l_{t, h, V}-\frac1t\log\la P_{0, t, h}^{V}\id_H, \mu\ra \right| = 0, 
\end{align}
uniformly for $\mu\in\P_{\eta, M}(H)$. The integral representation \eqref{eigen-value-integral-representation} and Birkhoff's ergodic theorem, as well as the continuity of $h\to\l^V(h)$ imply that 
\begin{align}\label{item-1-2}
	\lim_{t\to\infty}\frac1t\log\l_{t, h, V} = \lim_{t\to\infty}\frac1t\int_0^t\l^V(\b_rh)dr = \int_{\To}\l^V(h)m(dh). 
\end{align}
A combination of \eqref{item-1-1} and \eqref{item-1-2} verifies  \eqref{LDP-condition-1} for $V\in C_b^{1,0}(H\times\To)$ and \eqref{Q(V)-when-V-regular} in Item (1) of Theorem \ref{Kifer-condition}. The existence of the limit in \eqref{LDP-condition-1} for general $V\in C_b(H\times\To)$ follows from a $buc$-approximation argument as in \cite{JNPS18}, where the required  exponential tightness has been verified in Item (2). The Lipschitz continuity is derived from the limit definition of $Q$. Indeed, note that for any $\s\in\P_{\eta, M}(H)$, 
\begin{align*}
	\E_{\s}\exp\la tV_1, \zeta_{t, h}\ra &= \E_{\s}\exp\big(\la t(V_1-V_2), \zeta_{t, h}\ra+\la tV_2, \zeta_{t, h}\ra\big)\\
	&\leq e^{t\|V_1-V_2\|_{\infty}}\E_{\s}\exp\la tV_2, \zeta_{t, h}\ra. 
\end{align*}
By exchanging the role of $V_1, V_2$, we have 
\begin{align*}
	\left|\frac1t\log\E_{\s}\exp\la tV_1, \zeta_{t, h}\ra-\frac1t\log\E_{\s}\exp\la tV_2, \zeta_{t, h}\ra\right| 
	\leq \|V_1-V_2\|_{\infty}, 
\end{align*}
which by taking limit, implies the  Lipschitz continuity of $Q$. The convexity follows from H\"older's inequality. The proof of Item (1) is then complete. 

\vskip0.05in 

{\it Proof of Item (3) in Theorem \ref{Kifer-condition}}. The remainder of this subsection is then devoted to prove Item (3) on uniqueness of equilibrium state in Theorem \ref{Kifer-condition}, which is accomplished through the following three steps. Fix any $V\in C_b^{1, 0}(H\times\To)$. 

\vskip0.05in

{\it Step 1: The auxiliary Markov evolution operators defined through Doob's $h$-transform is uniquely ergodic}. 
Define the following auxiliary evolution family of Markov operators $T^V_{0, t, h}$ acting on $C_b(H)$, 
\begin{align}\label{definition-Doob-Markov-cocycle}
	T^V_{0, t, h}\phi(w) = \l_{t, h, V}^{-1}F^V(w, h)^{-1}P_{0, t, h}^V(\phi F^V(\cdot, \b_th))(w), \, \text{ for } (w, h)\in H\times\To, \phi\in C_b(H). 
\end{align}
This is an analogue of the Doob's $h$-transform in our current setting. See \cite{BGL14} and references therein for an introduction about the Doob's $h$-transform, and \cite{JNPS18,Ner19} for the application in proving the uniqueness of equilibrium state in the time homogeneous case. 

\vskip0.05in

The operator $T^V_{0, t, h}$ is Markov since $T^V_{0, t, h}\id_H=\id_H$ by the eigenfunction property of $F^V$. It follows from the integral representation \eqref{eigen-value-integral-representation} that the eigenvalue satisfies the cocycle property 
\[\l_{t+s, h, V} = \l_{t, \b_sh, V}\l_{s, h, V}, \, \text{ for } s, t\geq 0, h\in \To.\]
Base on this and the cocycle property \eqref{Feynman-Kac-cocycle} of the Feynman-Kac operators, it is straightforward to check that the Markov operators $T^V_{0, t, h}$ satisfy the cocycle property as well, 
\begin{align}\label{Doob-Markov-cocycle}
		T_{0, s+t, h}^V = T_{0, s, h}^VT_{0, t, \b_sh}^V, \, s, t\geq 0, h\in\To.
\end{align}
In addition, one directly verifies that the induced operator $\mct^V_t$ acting on $C_{b}(H\times\To)$ by 
\begin{align}\label{Doob-Markov-semigroup}
	\mct^V_t\phi(w, h) = T^V_{0, t, h}\phi(\cdot, \b_th)(w), \, \phi\in C_{b}(H\times\To), 
\end{align}
is a Markov semigroup, whose restriction on $C_b(H)$ yields $T^V_{0, t, h}$. 

\vskip0.05in

Note that the action of $\mct^V_t$ on $C_b(\To)$ is just the Koopman operator for the circle rotation. Hence any invariant measure $\mu$ of $\mct^V_t$ must have marginal $m$ on $\To$.  By disintegration theorem, the disintegration $\mu(h), h\in\To$ of $\mu$ is uniquely determined up to an $m$-null set and $\mu = \mu(h)m(dh)$. One can check that $\mu(h)$ is an invariant measure of the evolution family $T_{0, t, h}^{V*}$, where $T_{0, t, h}^{V*}$ is the dual of $T_{0, t, h}^{V}$. Namely, for every $t\geq 0$, 
\begin{align}\label{invariant-measure-Doob-Markov-cocycle}
	T_{0, t, h}^{V*}\mu(h)  = \mu(\beta_th), \quad m\text{-a.e } h\in \To. 
\end{align}
Conversely, any measurable family of probability measures $\mu(h), h\in\To$ satisfying \eqref{invariant-measure-Doob-Markov-cocycle} gives an invariant measure  $\mu(h)m(dh)$ of $\mct^V_t$. In this vein, unique ergodicity of Markov evolution family \eqref{definition-Doob-Markov-cocycle} is equivalent to the unique ergodicity of its homogenization \eqref{Doob-Markov-semigroup} and hence we will not distinguish them. 

\vskip0.05in

The following lemma says that the semigroup $\mct^V_t$  is uniquely ergodic with unique invariant measure $F^V(h)\G^V(h)m(dh)\in \P(H\times\To)$, which is a consequence of the uniqueness of eigenmeasure from Theorem \ref{asymptotic-behavior-feynman-kac}. 

 \begin{lemma}\label{invarinat-measure}
Let $\g_V(h) = F^V(\cdot, h)\G^V(h)$. Then $\g_V(h)m(dh)$ is the unique invariant measure of the semigroup $\mct^V_t$. 
\end{lemma}
\begin{proof}
We first verify that $\g_V(h)m(dh)$ is indeed invariant under $\mct^V_t$.  Since $\G^V$ is the eigenmeasure of $P_{0, t, h}^V$, one has 
\begin{align*}
	\la T_{0, t, h}^{V}g, \g_V(h)\ra 
	& =\la \l_{t, h, V}^{-1}F^V(\cdot, h)^{-1}(P_{0, t, h}^Vg(\cdot)F^V(\cdot, \b_th)), F^V(\cdot, h)\G^V(h)\ra \\
	& = \la\l_{t, h, V}^{-1}g(\cdot)F^V(\cdot, \b_th), P_{0, t, h}^{V*}\G^V(h)\ra = \la g, F^V(\cdot, \b_th)\G^V(\b_th)\ra,
\end{align*}
for every $g\in C_b(H)$. Hence $\g_V$ is an invariant measure for $T_{0, t, h}^{V*}$, i.e., 
\[T_{0, t, h}^{V*}\g_V(h)  = \g_V(\beta_th).\]
Therefore $\g_V(h)m(dh)$  is an invariant measure for the semigroup $\mct^V_t$. 

\vskip0.05in

We next show that this invariant measure is unique. Suppose that $\mu\in \P(H\times\To)$ is a different invariant measure  of the semigroup $\mct^V_t$,  then it must have marginal $m$ on $\To$ and its unique disintegration which we denote by $\g(h)$ satisfies \eqref{invariant-measure-Doob-Markov-cocycle}, i,e., 
\[T_{0, t, h}^{V*}\g(h)  = \g(\beta_th), \quad m\text{-a.e } h\in \To.\]
We claim that $\g$ is necessarily equal to $\g_V$ up to a null set. Let $\g_0(h) = F^V(\cdot, h)^{-1}\g(h)$ and $\phi\in C_b(H)$. It follows from the invariance of $\g$ that 
\begin{align*}
	\la\g_0(\b_th), \phi F^V(\cdot, \b_th)\ra = \la \g(\b_th), \phi\ra &= \la T_{0, t, h}^{V*}\g(h) , \phi\ra\\
	&= \l_{t, h, V}^{-1}\la F^V(\cdot, h)^{-1}\g(h) , P_{0, t, h}^{V}\left(\phi(\cdot) F^V(\cdot, \b_th)\right)\ra\\
	&=  \l_{t, h, V}^{-1}\la P_{0, t, h}^{V*}\g_0(h) , \phi(\cdot) F^V(\cdot, \b_th)\ra. 
\end{align*}
 Since $\phi$ is arbitrary and $F^V$ is strictly positive, we deduce 
 \[P_{0, t, h}^{V*}\g_0(h) = \l_{t, h, V}\g_0(\b_th), \quad m\text{-a.e } h\in \To.\]
 The uniqueness of eigenmeasure then implies $\g_0 = \G^V$  up to a null set, which in turn shows that $\g = \g_V$ for almost every $h\in\To$. Hence $\g_V(h)m(dh)$ is the unique invariant measure of the semigroup $T_{t}^{V}$. 
 
\vskip0.05in
 
The proof of Lemma \ref{invarinat-measure} is complete. 
\end{proof}

{\it Step 2: Correspondence between equilibrium states and zeros of the auxiliary rate function obtained from the auxiliary Markov evolution operators \eqref{definition-Doob-Markov-cocycle}}. 
To show the uniqueness of equilibrium state associated with $V\in C_b^{1, 0}(H\times\To)$ as in \eqref{definition-equilibrium-state}, we shall establish a correspondence between equilibrium states and the zeros of an auxiliary rate function, where the later is connected with the invariant measure of $\mct_{t}^{V}$. The uniqueness of the invariant measure proved in Lemma \ref{invarinat-measure} thus will imply  the uniqueness of equilibrium state. 

\vskip0.05in

More precisely, consider the auxiliary Feynman-Kac operators $R_{0, t, h}^{V, U}$ associated with $T_{0, t, h}^{V}$  and $U\in C_b(H\times\To)$ given by 
\[R_{0, t, h}^{V, U}g(w) = \l_{t, h, V}^{-1}F^V(w, h)^{-1}(P_{0, t, h}^{V+U}g(\cdot)F^V(\cdot, \b_th))(w), \quad g\in C_b(H), (w, h)\in H\times\To,\]
which induces as in \eqref{Doob-Markov-semigroup} its homogenization $\mcr_{t}^{V, U}$ acting on $C_{b}(H\times\To)$:
\begin{align}\label{auxiliary-Feynman-Kac-semigroup}
	\mcr_{t}^{V, U}g(w, h) = \l_{t, h, V}^{-1}F^V(w, h)^{-1}(P_{0, t, h}^{V+U}g(\cdot, \b_th)F^V(\cdot, \b_th))(w), \quad g\in C_{b}(H\times\To). 
\end{align}
We claim that the following limit is well defined, which gives the corresponding auxiliary pressure function, 
\begin{align}\label{definition-auxiliary-pressure}
Q^V(U) := \lim_{t\to\infty}\frac1t\log(R_{0, t, h}^{V, U}\id_H)(w), \quad U\in C_b(H\times\To), w\in H, h\in\To, 
\end{align}
which is independent of $w, h$.

\vskip0.05in

Indeed, for $U\in C_b^{1, 0}(H\times\To)$, the forward convergence \eqref{main-thereom-forward-mxing} in Theorem \ref{asymptotic-behavior-feynman-kac} implies that 
\begin{align}\label{auxiliary-pressure-1}
	\lim_{t\to\infty}\left|\l_{t, h, V+U}^{-1}P_{0, t, h}^{V+U}F^V(\cdot, \b_th)(w) -  F^{V+U}(w, h)\la\G^{V+U}(\b_th), F^V(\cdot, \b_th)\ra\right| = 0.
\end{align}
Estimating as in \eqref{lower-bound-eigenfunction-eigenmeasure} and using the properties of eigenfunction $F^V$ and eigenmeasure $\G^{V+U}$ as in Proposition \ref{Existence-of-Eigenfunction} and Proposition \ref{Existence-of-Eigenmeasure}, one can show that 
\begin{align*}
	p_1\leq \liminf_{t\to\infty}\la\G^{V+U}(\b_th), F^V(\cdot, \b_th)\ra\leq \limsup_{t\to\infty}\la\G^{V+U}(\b_th), F^V(\cdot, \b_th)\ra \leq p_2,
\end{align*}
for some constants $p_1, p_2>0$. Therefore convergence \eqref{auxiliary-pressure-1} implies
\begin{align*}
	\lim_{t\to\infty}\left|\l_{t, h, V+U}^{-1}\la\G^{V+U}(\b_th), F^V(\cdot, \b_th)\ra^{-1}P_{0, t, h}^{V+U}F^V(\cdot, \b_th)(w) -  F^{V+U}(w, h)\right| = 0.
\end{align*}
Combining this with Birkhoff's ergodic theorem and continuity of $h\to\l^{V+U}(h)$ and integral representation \eqref{eigen-value-integral-representation} for potential $V+U$, we have 
\begin{align}\label{auxiliary-pressure-2}
\begin{split}
	\lim_{t\to\infty}\frac1t\log P_{0, t, h}^{V+U}F^V(\cdot, \b_th)(w) 
	&= \lim_{t\to\infty}\frac1t\log \l_{t, h, V+U}\\
	&=\lim_{t\to\infty}\frac1t\int_0^t\l^{V+U}(\b_rh)dr = \int_{\To}\l^{V+U}(h)m(dh) = Q(U+V). 
\end{split}
\end{align}
In view of \eqref{definition-auxiliary-pressure} and the definition of $R_{0, t, h}^{V, U}$, limits \eqref{item-1-2} and  \eqref{auxiliary-pressure-2} yield
\begin{align}\label{auxiliary-pressure-3}
 	Q^V(U) = Q(U+V) - Q(V), \, \text{ for } U\in C_b^{1, 0}(H\times\To). 
\end{align}
Therefore the limit \eqref{definition-auxiliary-pressure} is well defined for $U\in C_b^{1, 0}(H\times\To)$. The existence of the limit \eqref{definition-auxiliary-pressure} for general $U\in C_b(H\times\To)$ follows from a $buc$-approximation argument as in the proof of Item (1). The relation \eqref{auxiliary-pressure-3} extends to $U\in C_b(H\times\To)$ as well by Lipschitz continuity of the pressure functions. 

\vskip0.05in

Denote by $I^V(\mu): = \ds\sup_{U\in C_b(H\times\To)}\big(\la U, \mu\ra - Q^V(U)\big)$ the Legendre transformation of $Q^V$, which is the auxiliary rate function. Then relation \eqref{auxiliary-pressure-3} implies
\[I^V(\mu) = I(\mu)+Q(V)-\langle V, \mu\rangle.\]
Hence a measure $\mu_V$ is an equilibrium state of the potential $V$, i.e, 
\[Q(V) = \langle V, \mu_V\rangle - I(\mu_V)\]
if and only if  $I^V(\mu_V)=0$. 

\vskip0.05in

{\it Step 3: The zeros of the auxiliary rate function are invariant measures of the Markov evolution operators \eqref{definition-Doob-Markov-cocycle}}. 
Following the scheme as in \cite{MN18}, one can prove the following lemma that connects the zeros of the auxiliary rate function $I^V$ with the invariant measure of $\mct_{t}^{V}$. The major difference is in proving the time continuity and Feller property of the semigroup $\mct_t^V$ since we are working with a semigroup obtained by coupling the evolution operator \eqref{definition-Doob-Markov-cocycle} with the circle rotation.  The proof is postponed in  the next subsection \ref{subsection-proof-zero-invariant-measure} for the reader's convenience.  
\begin{lemma}\label{zero-is-invariant-measure}
If $I^V(\mu) = 0$ then $\mu$ is an invariant measure of the semigroup  $\mct_{t}^{V}$. 
\end{lemma}

\vskip0.05in

Now since $I$ is a good rate function, there is at least one equilibrium state for $V$. Hence the zeros of $I^V$ is non-empty. In view of Lemma \ref{invarinat-measure} and Lemma \ref{zero-is-invariant-measure}, we know that the equilibrium state for $V$ is unique, which is the unique invariant measure $F^V(h)\G^V(h)m(dh)$ of $\mct_t^V$. This proves Item (3) in Theorem \ref{Kifer-condition}. 

\vskip0.05in

The proof of Theorem \ref{Kifer-condition} is complete. \qed

\subsection{Proof of auxiliary lemmas}\label{subsection-proof-zero-invariant-measure}

We give a proof of Lemma \ref{zero-is-invariant-measure} in this subsection. We first establish time continuity and Feller property of the semigroup $\mct_t^V$ in subsection \ref{Time-continuity-Feller-property-T}, which is the part that is different from the time homogeneous setting. Then in subsection \ref{proof-of-lemma-zero-is-invariant-measure} we prove Lemma \ref{zero-is-invariant-measure}, by combining the time continuity and Feller property, following the scheme 
as in \cite{MN18}. Recall that $V\in C_{b}^{1, 0}(H\times\To)$ and the semigroup $\mct_{t}^V$ is defined as in \eqref{Doob-Markov-semigroup} through Doob's $h$-transform \eqref{definition-Doob-Markov-cocycle}.

\subsubsection{Time continuity and Feller property of $\mct_t^V$}\label{Time-continuity-Feller-property-T}

The following lemma establishes the time continuity and Feller property of the semigroup. 
\begin{lemma}\label{time-continuity-T_t}
The semigroup $\mct_{t}^V$ has the following two properties:
\begin{enumerate}
\item For any $\phi\in C_b(H\times\To)$, and $(w, h)\in H\times\To$, the function $t\to \mct_{t}^V\phi(w,h)$ is continuous. 
\item For any $t\geq 0$, $\mct_{t}^V$  maps $C_b(H\times\To)$ to itself. 
\end{enumerate}
\end{lemma}
 \begin{proof}
 The proof is divided into the following two steps. 
 
\vskip0.05in

{\it Step 1: Proof of time continuity}. It follows from the integral representation \eqref{eigen-value-integral-representation} for the eigenvalue that $\l_{t, h, V}$ is continuous in time. In view of definition \eqref{definition-Doob-Markov-cocycle} and  \eqref{Doob-Markov-semigroup}, we only need to show the time continuity of the function 
\begin{align}\label{time-continuity-0}
	t\in [0, \infty)\to P^V_{0, t, h}(\phi(\cdot, \b_th)F^V(\cdot, \b_th))(w).
\end{align}

\vskip0.05in

Let's write for convenience $F_{\phi}=\phi F^V$. Since $\phi\in C_b(H\times\To)$, and $F^V\in L_{\mkm_q}^{\infty}(H\times\To)$, we know $F_{\phi}\in L_{\mkm_q}^{\infty}(H\times\To)$ as well. In addition, since $F^V\in C(H\times\To)$ by Theorem \ref{asymptotic-behavior-feynman-kac},  we know that 
\[(w, \tau)\to F^V(w, \b_\tau h) \, \text{ is jointly continuous everywhere}.\]
Therefore so is $F_{\phi}$. 

\vskip0.05in

By triangle inequality, denoting $\Xi_{t}^V = \exp\left(\int_0^tV(\H_r)dr\right)$, 
\begin{align}\label{time-continuity-3}
\begin{split}
	&\Big|P^V_{0, t, h}(\phi(\cdot, \b_th)F^V(\cdot, \b_th))(w) - P^V_{0, \tau, h}(\phi(\cdot, \b_{\tau}h)F^V(\cdot, \b_{\tau}h))(w) \Big|\\
	&\leq \E_{(w, h)}\Big|\Xi_{t}^VF_{\phi}(\H_t) -  \Xi_{\tau}^VF_{\phi}(\H_{t}) \Big|+\E_{(w, h)}\Big|\Xi_{\tau}^VF_{\phi}(\H_{t})  -  \Xi_{\tau}^VF_{\phi}(\H_{\tau}) \Big|\\
	&: = J_1 + J_2. 
\end{split}
\end{align}	
Since $F_{\phi}\in L_{\mkm_q}^{\infty}(H\times\To)$, we obtain 
\begin{align}\label{time-continuity-plus}
\begin{split}
	J_1&\leq C\E_{(w, h)}\left[\Big|\Xi_{t}^V-  \Xi_{\tau}^V \Big|\mkm_q(\Phi_{0, t, h})\right]\\
	&\leq C\E_{(w, h)}\left[\Big|\int_{\tau}^tV(\H_r) dr\Big|\mkm_q(\Phi_{0, t, h})\right]\\
	&\leq C|t-\tau|\mkm_{q}(w),
\end{split}
\end{align}
where in the second step we used the fact $V\in C_{b}^{1, 0}(H\times\To)$, and the Lipschitz continuity of $e^x$ on some bounded interval determined by $m_V, M_V, t, \tau$ and in the last inequality we used \eqref{contraction-polynomial}. Hence $J_1\to0$ as $t\to \tau$.	

\vskip0.05in

On the other hand, 
\begin{align*}
	J_2\leq e^{\tau\|V\|_{\infty}}\E_{w}\Big|F_{\phi}(\Phi_{0, t, h}, \b_th)  -  F_{\phi}(\Phi_{0, \tau, h}, \b_{\tau}h) \Big|
\end{align*}
By time continuity of sample paths of the solution $\Phi_{0, t, h}(w)$ and joint continuity of $(w, t)\to F_{\phi}(w, \b_th)$, we know that 
\begin{align*}
	|F_{\phi}(\Phi_{0, t, h}, \b_th)  -  F_{\phi}(\Phi_{0, \tau, h}, \b_{\tau}h) \Big|\to 0 
\end{align*}
almost surely as $t\to\tau$. In addition, the fact $F_{\phi}\in L_{\mkm_q}^{\infty}(H\times\To)$ implies 
\begin{align*}
	\E_{w}\Big|F_{\phi}(\Phi_{0, t, h}, \b_th)  -  F_{\phi}(\Phi_{0, \tau, h}, \b_{\tau}h) \Big|\leq C\mkm_q(w),
\end{align*}
by \eqref{contraction-polynomial}. Hence by Lebesgue's dominated convergence theorem, we have $J_2\to0$ as $t\to \tau$. 

\vskip0.05in

In view of \eqref{time-continuity-3}, we obtain the desired continuity of \eqref{time-continuity-0}. The proof of {\it Step 1} is complete. 

\vskip0.05in

{\it Step 2: Proof of Feller property}. Fix $t\geq 0$. We first show that for any Lipschitz $\phi\in C_b(H\times\To)$, 
\begin{align}\label{Feller-0}
	(w, h)\in H\times\To\to \mct_t^V\phi(w, h) \, \text{ is continuous. }
\end{align}
Since $(w, h)\to \frac{1}{\l_{t, h, V}F^V(w, h)}$ is continuous by Theorem \ref{asymptotic-behavior-feynman-kac} ,  in view of definitions \eqref{definition-Doob-Markov-cocycle} and  \eqref{Doob-Markov-semigroup}, we only need to show the continuity of the function 
\begin{align}\label{Feller-0.5}
	(w, h)\in H\times\To\to \P_{t}^VF_{\phi}(w, h), 
\end{align}
where $F_{\phi}=\phi F^V$. The main difficulties here are possible unboundedness of $F^V$ and coupling of the Navier-Stokes family \eqref{Settings-NS-family} with the circle rotation. 

\vskip0.05in

Denoting $\Xi_t^V(w, h) = \exp\left(\int_0^tV(\H_r(w, h))dr\right)$, we write 
\begin{align}\label{Feller-I1+I2}
\begin{split}
	\P_{t}^VF_{\phi}(w_n, h_n) - \P_{t}^VF_{\phi}(w, h) &= \E_{(w_n, h_n)}\Xi_t^VF_{\phi}(\H_t) -  \E_{(w, h)}\Xi_t^VF_{\phi}(\H_t)\\
	& = \E\left[\left(\Xi_t^V(w_n, h_n)-\Xi_t^V(w, h)\right)F_{\phi}(\H_t(w_n, h_n))\right]\\
	 &\qquad+ \E\left[\Xi_t^V(w, h)\big(F_{\phi}(\H_t(w_n, h_n)) -F_{\phi}(\H_t(w, h)) \big)\right]\\
	 &:= I_1 + I_2.
\end{split}
\end{align}
We can assume $w_n, w\in B_{R_0}^H$ for some $R_0>0$ and all $n\geq 1$. Note that since $V\in C_b^{1, 0}(H\times\To)$, by the same reasoning as in \eqref{time-continuity-plus}, and the fact $F_{\phi}\in L_{\mkm_q}^{\infty}(H\times\To)$ as well as \eqref{contraction-polynomial}, we have 
\begin{align}\label{Feller-1}
\begin{split}
	|I_1| &\leq C\E\left[\int_0^t\left|V(\H_r(w_n, h_n))-V(\H_r(w, h))\right|dr\mkm_q(\Phi_{0, t, h_n}(w_n))\right]\\
	&\leq C\left(\E\int_0^t\left|V(\H_r(w_n, h_n))-V(\H_r(w, h))\right|^2dr\right)^{\frac12}\left(\E\mkm_{2q}(\Phi_{0, t, h_n}(w_n))\right)^{\frac12}\\
	&\leq C\left(\E\int_0^t\left|V(\H_r(w_n, h_n))-V(\H_r(w, h))\right|^2dr\right)^{\frac12}.
\end{split}
\end{align}
By $V\in C_b^{1, 0}(H\times\To)$ and triangle inequality, we have the existence of some modulus of continuity $\k$ such that 
\begin{align*}
	&\E\int_0^t\left|V(\H_r(w_n, h_n))-V(\H_r(w, h))\right|^2dr \\
	&\leq C\left(\int_0^t\k^2(|\b_rh_n - \b_rh|)dr +\E\int_0^t|\Phi_{0, r, h_n}(w_n) -\Phi_{0, r, h}(w)|^2 dr\right)\\
	&\leq C\left(\k^2(|h_n - h|)+\E\int_0^t|\Phi_{0, r, h_n}(w_n) -\Phi_{0, r, h_n}(w)|^2 dr + \E\int_0^t|\Phi_{0, r, h_n}(w) -\Phi_{0, r, h}(w)|^2 dr\right)\\
	&\leq C\left(\k^2(|h_n - h|) + \|w_n-w\|^2 + |h_n-h|^2\right), 
\end{align*}
where we used \eqref{continuousonhull} with Lipschitz continuity of $f$ and \eqref{continuous-initial-condition} in the last inequity. Combining this with \eqref{Feller-1}, we have 
\begin{align}\label{Feller-I1}
	I_1\leq C\left(\k(|h_n - h|) + \|w_n-w\| + |h_n-h|\right). 
\end{align}

\vskip0.05in

We now estimate $I_2$. The boundedness of $V$ yields 
\begin{align}\label{Feller-I2}
\begin{split}
	e^{-t\|V\|_{\infty}} |I_2|\leq &\E\left|F_{\phi}(\H_t(w_n, h_n)) -F_{\phi}(\H_t(w, h)) \right|\\
	&\leq \E\left|F_{\phi}(\Phi_{0, t, h_n}(w_n), \b_th_n) -F_{\phi}(\Phi_{0, t, h}(w), \b_th_n) \right| \\
	&\qquad +\E\left|F_{\phi}(\Phi_{0, t, h}(w), \b_th_n) -F_{\phi}(\Phi_{0, t, h}(w), \b_th) \right|\\
	&:= I_{21}+I_{22}. 
\end{split}
\end{align}
Note that $\phi\in C_b(H\times\To^d)$ is Lipschitz,  and  by Proposition \ref{proposition-uniqueness-eigen-data}, for any $R>0$, the family $ F^V(\cdot, h), h\in\To$ is uniformly Lipschitz on $B_{R} := B_{R}^H$ with a uniform bound $L=L(R)$ on the Lipschitz constants, therefore so is the the family $F_{\phi}(\cdot, h)=(\phi F^V)(\cdot, h), h\in\To$. In addition, \eqref{continuousonhull} with Lipschitz continuity of $f$ and \eqref{continuous-initial-condition} imply that 
\begin{align}\label{Feller-I21-1}
\begin{split}
	\E\|\Phi_{0, t, h_n}(w_n)-\Phi_{0, t, h}(w)\|&\leq \E\|\Phi_{0, t, h_n}(w_n)-\Phi_{0, t, h_n}(w)\|+\E\|\Phi_{0, t, h_n}(w)-\Phi_{0, t, h}(w)\|\\
	&\leq C\left(\|w_n-w\| + |h_n-h|\right),
\end{split}
\end{align}
with constant $C=C(R_0)$, where we used the fact that $w_n, w\in B_{R_0}^H$ for all $n\geq 1$. Since $F_{\phi} = \id_{B_R}F_{\phi} + \id_{B_R^c}F_{\phi}\in L_{\mkm_q}^{\infty}(H\times\To)$,  we have 
\begin{align}\label{Feller-I21-2}
\begin{split}
	I_{21} &\leq \E\left|(\id_{B_R}F_{\phi})(\Phi_{0, t, h_n}(w_n), \b_th_n) -(\id_{B_R}F_{\phi})(\Phi_{0, t, h}(w), \b_th_n) \right| \\
	&\qquad + \E\left|(\id_{B_R^c}F_{\phi})(\Phi_{0, t, h_n}(w_n), \b_th_n) -(\id_{B_R^c}F_{\phi})(\Phi_{0, t, h}(w), \b_th_n) \right|\\
	&\leq L(R)\E\|\Phi_{0, t, h_n}(w_n)-\Phi_{0, t, h}(w)\|+CP_{0, t, h_n}\left(\id_{B_R^c}\mkm_q\right)(w_n)+ CP_{0, t, h}\left(\id_{B_R^c}\mkm_q\right)(w).
\end{split}
\end{align}
By H\"older's inequality and Markov inequality, as well as \eqref{contraction-polynomial}, we have 
\begin{align}\label{Feller-I21-3}
\begin{split}
	P_{0, t, \mathfrak{h}}\left(\id_{B_R^c}\mkm_q\right)(u)
	&\leq \left(P_{0, t, \mathfrak{h}}\id_{B_R^c}(u)\right)^{1/2}\left(P_{0, t, \mathfrak{h}}\mkm_{2q}(u)\right)^{1/2}\\
	&\leq CR^{-1}P_{0, t, \mathfrak{h}}\|\cdot\|(u)\leq CR^{-1}, 
\end{split}
\end{align}
for all $(u, \mathfrak{h})\in\lp (w, h), (w_n, h_n): n\geq 1\rp$, where $C=C(R_0)$. Hence by combining \eqref{Feller-I21-1}-\eqref{Feller-I21-3}, we obtain 
\begin{align*}
	I_{21}\leq C(R)\left(\|w_n-w\| + |h_n-h|\right) + CR^{-1}, 
\end{align*}
which tends to zero by first making $R$ large and then $n$ large so that $\|w_n-w\| + |h_n-h|\to0$. 

\vskip0.05in

We now look at $I_{22}$. Since for each $u\in H$, we have 
\begin{align*}
	F_{\phi}(u, \b_th_n)\to F_{\phi}(u, \b_th), \, \text{ as }h_n\to h. 
\end{align*}
and $\left|F_{\phi}(u, \b_th_n)- F_{\phi}(u, \b_th)\right|\leq C\mkm_q(u)$ by $L_{\mkm_q}^{\infty}(H\times\To)$, which is $P_{0, t, h}(w, du)$ integrable by \eqref{contraction-polynomial},  hence by Lebesgue's dominated convergence theorem, we have 
\begin{align*}
	I_{22} =\int_{H} \left|F_{\phi}(u, \b_th_n)- F_{\phi}(u, \b_th)\right|P_{0, t, h}(w, du)\to 0\, \text{ as }h_n\to h. 
\end{align*}
In view of \eqref{Feller-I2} we have $|I_2|\to 0$ as long as $(w_n, h_n)\to(w, h)$. 

\vskip0.05in

Combining this with \eqref{Feller-I1+I2} and \eqref{Feller-I1}, we see that continuity of \eqref{Feller-0.5} and hence \eqref{Feller-0} is true for any Lipschitz $\phi\in C_b(H\times\To)$. 

\vskip0.05in

Now for any $\phi\in C_b(H\times\To)$, we can find a sequence $\phi_m$ of bounded Lipschitz functions such that 
\[\phi_{M}: = \sup_{m}\|\phi_m\|_{\infty}<\infty,  \text{ and } \sup_{(w, h)\in K\times\To}|\phi_m(w, h)-\phi(w, h)|\to 0\]
as $m\to \infty$ for any compact $K\subset H$. For any $(w_0, h_0)\in H\times\To$ and $(w_n, h_n)\to (w_0, h_0)$ as $n\to\infty$,  denote $z_n = (w_n, h_n)$ for $n\geq 0$. By triangle inequality we have 
\begin{align}\label{Feller-Cb-0}
\begin{split}
	&\left|\mct_t^V\phi(z_n) - \mct_t^V\phi(z_0)\right|\\
	& \leq \left|\mct_t^V\phi(z_n) - \mct_t^V\phi_m(z_n)\right|+\left|\mct_t^V\phi_m(z_n) - \mct_t^V\phi_m(z_0)\right|+\left|\mct_t^V\phi_m(z_0) - \mct_t^V\phi(z_0)\right|. 
\end{split}
\end{align}
We can assume the existence of $R_0>0$ such that $z_n\in B_{R_0}^H\times\To$ for all $n$. For any $z = (u, \mathfrak{h})\in\lp z_n: n\geq 0\rp$, and $K_{R}: = B_{R}^{H_2}$ which is compact in $H$, we have 
\begin{align}\label{Feller-Cb-1}
\begin{split}
	&\left|\mct_t^V\phi(z) - \mct_t^V\phi_m(z)\right|\\
	&\leq \left|\mct_t^V\left(\id_{K_R}\phi\right)(z) - \mct_t^V\left(\id_{K_R}\phi_m\right)(z)\right|+ \left|\mct_t^V\left(\id_{K_R^c}\phi\right)(z) - \mct_t^V\left(\id_{K_R^c}\phi_m\right)(z)\right|\\
	&\leq \sup_{(w, h)\in K_R\times T}|\phi(w, h) - \phi_m(w, h)| + \frac{C(\|\phi\|_{\infty},\phi_{M},\|V\|_{\infty})}{\l_{t, \mathfrak{h}, V}F^V(u, \mathfrak{h})}P_{0, t, \mathfrak{h}}(\id_{K_R^c}\mkm_q(u)). 
\end{split}
\end{align}
By H\"older's inequality and Markov inequity as in \eqref{Feller-I21-3}, and using \eqref{contraction-polynomial}, we have 
\begin{align}\label{Feller-Cb-2}
	P_{0, t, \mathfrak{h}}(\id_{K_R^c}\mkm_q(u))\leq CR^{-1}, 
\end{align}
with constant $C = C(R_0)$ independent of $u, \mathfrak{h}, R$. Since $h\to \l_{t, h, V}>0$ is continuous, $F^V$ is strictly positive everywhere, and $z_n\to z_0$, we have 
\begin{align*}
	\inf_{(u, \mathfrak{h})\in \lp z_n: n\geq 0\rp}\l_{t, \mathfrak{h}, V}F^V(u, \mathfrak{h})>0. 
\end{align*}
Combining this with \eqref{Feller-Cb-1} and \eqref{Feller-Cb-2}, we have
\begin{align*}
	\left|\mct_t^V\phi(z) - \mct_t^V\phi_m(z)\right|\leq \sup_{(w, h)\in K_R\times T}|\phi(w, h) - \phi_m(w, h)| + CR^{-1}
\end{align*}
for all $z \in\lp z_n: n\geq 0\rp$, where the constant $C$ is independent of $z$ and $R$. In view of \eqref{Feller-Cb-0}, for any $\varepsilon>0$,  we can first choose $R$ large and then take $m$ large so that 
\[\left|\mct_t^V\phi(z_n) - \mct_t^V\phi(z_0)\right|\leq \left|\mct_t^V\phi_m(z_n) - \mct_t^V\phi_m(z_0)\right| + \frac{\varepsilon}{2},\]
for all $n$. Then using continuity of \eqref{Feller-0} for Lipschitz $\phi_m$, we find $N=N(m, R, \varepsilon)$ such that for all $n\geq N$, 
\[\left|\mct_t^V\phi(z_n) - \mct_t^V\phi(z_0)\right|< \varepsilon, \]
which completes the proof of Feller property. 

\vskip0.05in 

The proof of Lemma \ref{time-continuity-T_t} is complete. 

\end{proof}

\subsubsection{Proof of Lemma \ref{zero-is-invariant-measure}}\label{proof-of-lemma-zero-is-invariant-measure}

In this subsection we prove Lemma \ref{zero-is-invariant-measure} by following the scheme in \cite{MN18}. Consider the infinitesimal generator $\mbl_V$ of the Markov semigroup $\mct_{t}^V$. The domain $\mcd(\mbl_V)$ of the generator consists of those functions $\phi\in C_b(H\times\To)$ so that there is some $\varphi\in C_b(H\times\To)$ such that the following is true, 
\begin{align}\label{define-infinitesimal-generator}
	\mct_{t}^V\phi(w, h) = \phi(w, h) + \int_0^t\mct_r^V\varphi(w, h)dr, \, t\geq 0, (w, h)\in H\times\To. 
\end{align}
Therefore for every $\phi\in \mcd(\mbl_V)$, it follows from the time continuity in Lemma \ref{time-continuity-T_t} that the associated function $\varphi$ in the representation \eqref{define-infinitesimal-generator} is unique and it is equal to the following limit, 
\begin{align}
	\varphi(w, h) = \lim_{t\to 0}\frac{\mct_{t}^V\phi(w, h)-\phi(w, h)}{t}, \, (w, h)\in H\times\To. 
\end{align}
We then define $\varphi$ as the image of $\phi$ under the action of  $\mbl_V$ and write $\varphi = \mbl_V\phi$. Hence the integral equation \eqref{define-infinitesimal-generator} can be rewritten as the differential equation
\begin{align*}
	\partial_t\mct_{t}^V\phi = \mct_{t}^V\mbl_V\phi, \, \phi\in \mcd(\mbl_V). 
\end{align*}
It also follows from the definition \eqref{define-infinitesimal-generator} as well as Lemma \ref{time-continuity-T_t} that if $\phi\in \mcd(\mbl_V)$ with $\varphi = \mbl_V\phi$ then $\mct_s^V\phi\in \mcd(\mbl_V)$ with 
\[\mbl_V\mct_s^V\phi = \mct_s^V\varphi =\mct_s^V\mbl_V\phi.\]
The following Lemma \ref{lemma-special-zero-for-auxiliary-pressure} is useful. Its proof is given after we show the proof of Lemma \ref{zero-is-invariant-measure} below. 
\begin{lemma}\label{lemma-special-zero-for-auxiliary-pressure}
Recall the auxiliary pressure function $Q^V$ defined as in \eqref{definition-auxiliary-pressure}.  
\begin{enumerate}
\item The set 
\[\mcd_{+}:=\lp \phi\in\mcd(\mbl_V):\inf_{(w, h)\in H\times\To} \phi(w, h)>0\rp,\]
is determining for $\P(H\times\To)$, i.e., for any $\mu_1, \mu_2\in\P(H\times\To)$, if $\la\phi, \mu_1\ra =\la\phi, \mu_2\ra$ for all $\phi\in\mcd_{+}$ then $\mu_1 = \mu_2$. 
\item For any $\phi\in\mcd_{+}$, we have $Q^V(U_{\phi})=0$ where $U_{\phi} = -\frac{\mbl_V\phi}{\phi}$. 
\end{enumerate}
\end{lemma}
\begin{proof}[Poof of Lemma \ref{zero-is-invariant-measure}]
Assume $I^V(\mu) =0$. Since $Q^V(U_{\phi})=0$ and $U_{\phi}\in C_b(H\times\To)$, one has 
\begin{align*}
	0 = I^V(\mu) = \sup_{U\in C_b(H\times\To)}\left(\la U, \mu\ra - Q^V(U)\right)\geq \sup_{\phi\in\mcd_+}\la U_{\phi}, \mu\ra
\end{align*}
This is to say 
\begin{align*}
	 \inf_{\phi\in\mcd_+}\la \frac{\mbl_V\phi}{\phi}, \mu\ra \geq 0. 
\end{align*}
Since $\mct_t^V$ is Markov, we have $\mbl_V\id = 0$, which implies that $\a = $ is a local minimum of the function 
\begin{align*}
	J_{\mu}(\a) := \la \frac{\mbl_V\left(1+\a\phi\right)}{1+\a\phi}, \mu \ra, \, \text{ for any }\phi\in\mcd_+. 
\end{align*}
Therefore there must holds 
\begin{align}\label{proof-zero-is-invariant-measure}
	0 = J_{\mu}'(0) =  \la \mbl_V\phi, \mu\ra, 
\end{align}
for all $ \phi\in\mcd_+$. Since $\mct_s^V\phi\in \mcd_+$ for any $ \phi\in\mcd_+$ and $s\geq 0$, replacing $\phi$  in \eqref{proof-zero-is-invariant-measure} with $\mct_s^V\phi$ and integrate from $0$ to $t$ we obtain 
\begin{align*}
	\la \mct_t^V\phi, \mu\ra = \la \phi, \mu\ra 
\end{align*}
for all $t\geq 0$ and $ \phi\in\mcd_+$. Since $\mcd_+$ is determining by Lemma \ref{lemma-special-zero-for-auxiliary-pressure}, we have $\mct_t^{V*}\mu = \mu$, i.e.,  any $\mu$ satisfying $I^V(\mu) =0$ is invariant under $\mct_t^V$. 

\vskip0.05in

The proof of Lemma \ref{zero-is-invariant-measure} is complete. 

\end{proof}

\begin{proof}[Proof of Lemma \ref{lemma-special-zero-for-auxiliary-pressure}]
The proof is divided into two steps. 

\vskip0.05in

{\it Step 1: Proof of Item (1)}. Let $\mu_1, \mu_2\in\P(H\times\To)$ and assume 
\begin{align}\label{determining-measure-1}
	\la\phi, \mu_1\ra =\la\phi, \mu_2\ra,\,  \text{ for all }\phi\in\mcd_{+}.
\end{align}
It is known that the following set 
\begin{align*}
	C_b^{+}(H\times \To): = \lp U\in C_b(H\times \To): \inf_{(w, h)\in H\times\To}U(w, h)>0\rp
\end{align*}
is determining. We will show that $\mcd_+$ is rich enough to approximate any function in $ C_b^{+}(H\times \To)$.  

\vskip0.05in

 Take any $U\in C_b^{+}(H\times \To)$.  We claim that $\overline{U}_s:=\frac1s\int_0^s\mct_r^VUdr\in \mcd_{+}$ for any $s>0$.  Indeed, since $\mct_r^V$ is conservative, i.e., $\mct_r^V\id = \id$, we have 
\begin{align*}
	0<m_U\leq \overline{U}_s(w, h)\leq M_U,\, (w, h)\in H\times\To, 
\end{align*}
where $m_U, M_U$ are the infimum and supremum of $U$. On the other hand, one has 
\begin{align*}
	\mct_{t}^V\overline{U}_s - \overline{U}_s &= \frac1s\int_0^s\mct_{t+r}^VUdr - \frac1s\int_0^s\mct_r^VUdr\\
	& = \frac1s\int_s^{s+t}\mct_{r}^VUdr - \frac1s\int_0^t\mct_r^VUdr = \int_0^t\mct_r^V\widehat{U}^s dr,
\end{align*}
where $\widehat{U}_s = \frac{\mct_s^VU- U}{s}$. Since $U\in C_b(H\times \To)$, it follows from Lemma \ref{time-continuity-T_t} that $\widehat{U}_s\in C_b(H\times\To)$.  Hence $\overline{U}_s\in \mcd_{+}$  and the claim is proved. 

\vskip0.05in

It then follows from \eqref{determining-measure-1} that 
\begin{align}\label{determining-measure-2}
	\la\overline{U}_s, \mu_1\ra =\la\overline{U}_s, \mu_2\ra,\,  \text{ for all }s>0.
\end{align}
Again the fact $U\in C_b(H\times \To)$ and Lemma \ref{time-continuity-T_t} yield 
\[\lim_{s\to 0}\overline{U}_s(w, h) = U(w, h), \, (w, h)\in H\times\To.\]
Therefore by Lebesgue's dominated convergence theorem and \eqref{determining-measure-2}, one has 
\begin{align*}
	\la U, \mu_1\ra =\la U, \mu_2\ra. 
\end{align*}
Since $U\in C_b^{+}(H\times \To)$ is arbitrary, we conclude $\mu_1 = \mu_2$. 

\vskip0.05in

{\it Step 2: Proof of Item (2)}. We first claim that for any $U, \phi\in C_{b}(H\times\To)$ the trajectory $\phi_t: = \mcr_{t}^{V, U}\phi$ of the semigroup $\mcr_{t}^{V, U}$ defined as in \eqref{auxiliary-Feynman-Kac-semigroup} is the solution to  the following equation
\begin{align}\label{claim-intergal-duhamel}
	\phi_t = \mct_t^V\phi + \int_0^t \mct_r^V(U\phi_{t-r})  dr.
\end{align}
Note that by \eqref{definition-Doob-Markov-cocycle}, \eqref{Doob-Markov-semigroup} and \eqref{auxiliary-Feynman-Kac-semigroup}, denoting $\Xi_{s, t}^V = \exp\left(\int_s^tV(\H_{\tau})d\tau\right)$, and $\E=\E_{(w, h)}$, 
\begin{align*}
	\phi_t - \mct_t^V\phi &= \mcr_{t}^{V, U}\phi - \mct_t^V\phi\\
	& =  \frac{1}{\l_{t, h, V}F^V}\E\left[\Xi_{0,t}^{V}(\Xi_{0, t}^{U}-1)(\phi F^V)(\H_t)\right]\\
	& = \frac{1}{\l_{t, h, V}F^V}\E\left[\Xi_{0,t}^{V}\left(\int_0^t-\frac{d}{dr}\Xi_{r, t}^Udr\right)(\phi F^V)(\H_t)\right]\\
	& = \frac{1}{\l_{t, h, V}F^V}\int_0^t\E\left[\Xi_{0,t}^{V}U(H_r)\Xi_{r, t}^U(\phi F^V)(\H_t)\right]dr.\\
\end{align*}
Therefore by the Markov property, and time homogeneity of $\P_{t}^{V+U}$, as well as cocycle property of the eigenvalue $\l_{t, h, V}$ from the integral representation \eqref{eigen-value-integral-representation}, we deduce 
\begin{align*}
	\phi_t - \mct_t^V\phi &= \frac{1}{\l_{t, h, V}F^V}\int_0^t\E\left[\Xi_{0, r}^{V}U(H_r)\E\left[\Xi_{r, t}^{V+U}(\phi F^V)(\H_t)|\F_{r}\right]\right]dr\\
	&=\frac{1}{\l_{t, h, V}F^V}\int_0^t\E\left[\Xi_{0, r}^{V}U(H_r)\E\left[\Xi_{r, t}^{V+U}(\phi F^V)(\H_t)|\F_{r}\right]\right]dr\\
	&=\int_0^t\frac{1}{\l_{r, h, V}F^V}\E\left[\Xi_{0, r}^{V}U(H_r)F^V(H_r)\frac{1}{\l_{t-r, \b_rh, V} F^V(H_r)}\P_{t-r}^{V+U}(\phi F^V)(\H_r)\right]dr\\
	&=\int_0^t\frac{1}{\l_{r, h, V}F^V}\E\left[\Xi_{0, r}^{V}U(H_r)F^V(H_r)\mcr_{t-r}^{V, U}(\H_r)\right]dr\\
	& = \int_0^t \mct_r^V(U\phi_{t-r})  dr 
\end{align*}
which completes the proof of the claim. 

\vskip0.05in

Now for any $\phi\in\mcd_{+}$, we have 
\begin{align*}
	\mct_{t}^V\phi(w, h) = \phi(w, h) + \int_0^t\mct_r^V\mbl_V\phi(w, h)dr. 
\end{align*}
Combining this and with $U_\phi = -\frac{\mbl_V\phi}{\phi}$ it is straightforward to verify that $\phi_t\equiv \phi$ solves the equation 
\[\phi_t = \mct_t^V\phi + \int_0^t \mct_r^V(U_{\phi}\phi_{t-r})  dr.\]
Hence we have $\mcr_{t}^{V, U_{\phi}}\phi=\phi$ for all $t\geq 0$.  Since $\phi\in\mcd_{+}$, in view of \eqref{auxiliary-Feynman-Kac-semigroup}, there are constants $m_{\phi}, M_{\phi}>0$ such that 
\begin{align*}
	m_{\phi}R_{0, t, h}^{V, U_{\phi}}\id_{H}(w)\leq \mcr_{t}^{V, U_{\phi}}\phi(w, h)\leq M_{\phi}R_{0, t, h}^{V, U_{\phi}}\id_{H}(w).
\end{align*}
Also note 
\begin{align*}
	 \lim_{t\to\infty}\frac1t\log\mcr_{t}^{V, U_{\phi}}\phi(w, h) =  \lim_{t\to\infty}\frac1t\log\phi(w, h)=0. 
\end{align*}
Hence we obtain the desired Item (2)
\begin{align*}
	Q^V(U_{\phi}) = \lim_{t\to\infty}\frac1t\log R_{0, t, h}^{V, U_{\phi}}\id_{H}(w) = 0. 
\end{align*}

\vskip0.05in

The proof is complete. 

\end{proof}

\section*{Appendix}
\setcounter{section}{0}
\renewcommand{\thesection}{\Alph{section}}
\counterwithin{equation}{section}
\renewcommand{\theequation}{\thesection.\arabic{equation}}
\counterwithin{theorem}{section}
\renewcommand{\thetheorem}{\thesection.\arabic{theorem}}

\section{A Kifer type criterion} 

The following Kifer type theorem was established in \cite{JNPS18}, which is based on the work of Kifer \cite{Kif90} in the compact setting. Suppose that $X$ is a polish space, $\Theta$ is a directed set and $\{\zeta_{\theta}, \theta\in\Theta\}$ is a family of random probability measures on $X$ that are defined on some probability spaces $(\O_{\theta}, \mathcal{F}_{\theta}, \mathbf{P}_{\theta})$. $r:\Theta\to\mathbb{R}$ is a given positive function such that $\lim_{\theta\in\Theta}r(\theta)=+\infty$. 
\begin{theorem}[Kifer Type Criterion \cite{JNPS18}]\label{Kifer-Type-Criterion}
Assume that the family $\{\zeta_{\theta}, \theta\in\Theta\}$ satisfies the following two properties: 
\begin{enumerate}
\item The following limit exists for any $V\in C_b(X)$:
\[Q(V) = \lim_{\theta\in\Theta}\frac{1}{r(\theta)}\log\int_{\O_{\theta}}\exp(r(\theta)\la V, \zeta_{\theta}\ra)\mathbf{P}_{\theta}(d\o).\]
\item There is a function $R: X\to [0, \infty]$ with compact level sets $\{x\in X: R(x)\leq a\}$ for all $a\geq 0$ such that 
\[\int_{\O_{\theta}}\exp(r(\theta)\la R, \zeta_{\theta}\ra)\mathbf{P}_{\theta}(d\o)\leq Ce^{cr(\theta)}, \, \theta\in\Theta\]
where $C$ and $c$ are positive constants.
\end{enumerate}
Then $Q$ is $1$-Lipschitz and convex, and its Legendre transform $I$ given by 
\[I(\mu)=\left\{\begin{array}{cc}
\ds\sup_{V\in C_b(X)}\left(\la V, \mu\ra-V(Q)\right), &\mu\in \P(X),\\
\qquad+\infty, &\text{otherwise.}
\end{array}\right.\]
is a good rate function, for which the following upper bound holds for any closed subset $F\in \P(X)$:
\begin{align*}
\limsup_{\theta\in\Theta}\frac{1}{r(\theta)}\log\mathbf{P}_{\theta}\{\zeta_{\theta}\in F\}\leq -\inf_{\mu\in F}I(\mu).
\end{align*}
If in addition, the following property is satisfied: 
\begin{enumerate}[resume]
\item Suppose there exists a vector space $\mathcal{V}\subset C_b(X)$ such that the restrictions of its functions to any compact set $K\subset X$ form a dense subspace in $C(K)$ and for any $V\in\mathcal{V}$ there is a unique equilibrium state, i.e., a unique $\mu_V\in \P(H)$ such that 
\[Q(V) = \la V, \mu_V\ra - I(\mu_V).\]
\end{enumerate}
Then the following lower bound holds for any open subset $G\in\P(X)$: 
\begin{align*}
\liminf_{\theta\in\Theta}\frac{1}{r(\theta)}\log\mathbf{P}_{\theta}\{\zeta_{\theta}\in G\}\geq -\inf_{\mu\in G}I(\mu).
\end{align*}
\end{theorem}

\section{A priori estimates of solutions}
Several standard estimates about the solution  $\Phi_{0, t, h}(w_0)$ of the stochastic Navier-Stokes equation \eqref{Settings-NS-family} with time symbol $h\in\To$ are collected in the following Proposition \ref{bounds}. Note that the constant $C$ in the bounds does not depend on $h$ due to the boundedness of the periodic force $f$. Let $\|f\|_{\infty}=\sup_{t\in\R}\|f\|$ and  $\mathcal{B}_0 = \sum_{k=1}^{d}\|g_k\|^2$ for $f\in C_b(\R, H)$  and $g_k\in H$. 

\begin{proposition}\label{bounds}
There is an $\eta_0=\eta_0(\nu, \|f\|_{\infty}, \mathcal{B}_0)$ so that for any $\eta\in(0,\eta_0]$ and $h\in\To$, there is a constant $C>0$ independent of $h$ such that 
\begin{enumerate}
\item The following estimates hold: 
	\begin{align}\label{eq: est1}
		&\mathbf{E}_{w_0} \exp \left(\eta\left\|\Phi_{0, t, h}\right\|^{2}\right) \leq C \exp \left(\eta e^{-\nu t}\left\|w_0\right\|^{2}\right),\\\label{enstrophy}
		&\mathbf{E}_{w_0} \exp \left(\eta \sup _{t \geq 0}\left(\left\|\Phi_{0, t, h}\right\|^{2}+\nu \int_{0}^{t}\left\|\Phi_{0, r, h}\right\|_{1}^{2} d r-Ct\right)\right)\leq C \exp \left(\eta\left\|w_0\right\|^{2}\right),\\\label{mixing-moments}
		&\mathbf{E}_{w_0} \left[\|\Phi_{0, t, h}\|^2 \exp \left(\eta\left\|\Phi_{0, t, h}\right\|^{2}\right)\right]\leq C(t)\exp \left(\eta \left\|w_0\right\|^{2}\right)
	\end{align}
\item For any $h_1, h_2\in \To$,  we have 
	\begin{align}\label{continuousonhull}
		\mathbf{E}_{w_0}\left\|\Phi_{0, t, h_1} - \Phi_{0, t, h_2}\right\|^2\leq Ce^{r t}\exp\left({\eta\|w_0\|^2}\right)\sup_{t\in\R}\|f(\b_th_1)-f(\b_th_2)\|^2.
	\end{align}
\item For any $R>0$ and $w_1, w_2\in B_R^H$, we have 
	\begin{align}\label{continuous-initial-condition}
		\mathbf{E}\|\Phi_{0, t, h}(w_1)-\Phi_{0, t, h}(w_2)\|^2\leq C\|w_1 - w_2\|^{2}e^{\eta R^2+rt}.
	\end{align}
\item For any $m\geq 1$, 
	\begin{align}\label{contraction-polynomial}
		\begin{aligned}
		&\mathbf{E}_{w_0}\|\Phi_{0, t, h}\|^{2m}\leq e^{-\nu mt}\|w_0\|^{2m}+ C(m), \\
		&\mathbf{E}_{w_0}\|\Phi_{0, t, h}\|_{2}^{2m}\leq C (m, t)(1+\|w_0\|^{12m}).
		\end{aligned}
	\end{align}
\end{enumerate}
\end{proposition}
Estimates \eqref{eq: est1} and \eqref{enstrophy} can be proved as in \cite{HM06}. Inequalities \eqref{continuousonhull} and \eqref{continuous-initial-condition} can be proved as in \cite{HM08}. One can prove \eqref{mixing-moments} and \eqref{contraction-polynomial} as \cite{Ner19}. See also \cite{KS12}. 

\vskip0.05in

\addtocontents{toc}{\protect\setcounter{tocdepth}{1}}
\section{Growth condition, uniform irreducibility and uniform Feller property}
\subsection{The Growth Condition}
The following growth condition is useful when working with the non-conservative Feynman-Kac operator $P^V_{0, t, h}$ in a non-compact phase space, which was first introduced in \cite{JNPS18}  in the time homogeneous setting. 
\begin{proposition}\label{Growth-Condition}
Assume $V\in B_b(H\times\To)$. Recall that $\mathfrak{m}_{\eta}(w) = e^{\eta\|w\|^2}$ and $\mkm_q(w) = 1+\|w\|^{2q}$.  Then there are $R_0>0$ and $q=q(V)\geq 1$ such that
\begin{align*}
	&M_{\eta}: = \sup_{t\geq 0, h\in\To}\frac{\|P_{0,t,h}^V\mathfrak{m}_{\eta}\|_{L_{\mathfrak{m}_{\eta}}^{\infty}(H)}}{\|P_{0, t, h}^V\id_H\|_{R_0}}<\infty.\\
	&M_{q}: = \sup_{t\geq 0, h\in\To}\frac{\|P_{0,t,h}^V\mkm_{q}\|_{L_{\mkm_{q}}^{\infty}(H)}}{\|P_{0, t, h}^V\id_H\|_{R_0}}<\infty.\\
	&M_{H_2}: = \sup_{t\geq 0, h\in\To}\frac{\|P_{0,t,h}^VN_2\|_{L_{\mkm_{\eta}}^{\infty}(H)}}{\|P_{0, t, h}^V\id_H\|_{R_0}}<\infty, \quad N_2(w) = \|w\|_2^2,
\end{align*}
for all $\eta\in(0, \eta_0]$. Here $\dh_{R_0}$ is the $L^{\infty}$ norm on $B_{R_0}^{H_2}$.  In particular, we have 
\begin{align*}
	P_{0,t,h}^V\mathfrak{m}_{\eta}(w)\leq M_{\eta}\|P_{0, t, h}^V\id_H\|_{R_0}\mathfrak{m}_{\eta}(w), \,\text{ for all } t\geq0, h\in\To, w\in H. 
\end{align*}
\end{proposition}
The proof of Proposition \ref{Growth-Condition} is based on \cite{Ner19} where the time homogeneous case was proved, which is based on the following hyper-exponential recurrence property of the solutions to the Navier-Stokes system, see also \cite{JNPS18, KS12}.
\begin{proposition}\label{prop-hyperexponential-recurrence}
For any $R>0, h\in\To$, let $\tau_{h}(R)$ be the first hitting time of the solution $\Phi_{0, t, h}$ to the set $B_R^{H_2}$:
\begin{align*}
	\tau_{h}(R) = \inf\lp t\geq 0:  \Phi_{0, t, h}\in B_R^{H_2}\rp.
\end{align*}
 Then for any $\g>0$, there is $q_0=q_0(\g)>\g/\nu$ so that for any $q\geq q_0$, there are positive numbers $ R, C$ depending on $q$ such that
 \begin{align}\label{hyper-H2-0}
	\sup_{h\in\To}\E_{w}\exp\left(\g\tau_{h}(R)\right)\leq C\mkm_q(w), \, w\in H. 
\end{align}
Here $\nu>0$ is the viscosity. 
\end{proposition}
 For the sake of completeness, we provide proofs for both propositions in the current time inhomogeneous setting. The translation identity \eqref{translation-identity} and uniform in $h$ bounds from Proposition \ref{bounds} are the  crucial ingredients when dealing with time inhomogeneity throughout the proof. We first give the proof of Proposition \ref{Growth-Condition} based on Proposition \ref{prop-hyperexponential-recurrence}, whose proof is given at the end of this subsection. 
\begin{proof}[Proof of Proposition \ref{Growth-Condition}]
Since we are estimating the ration of $P_{0, t, h}^{V}$, we can assume 
\[m_V=\inf_{(w, h)\in H\times\To} V(w, h)\geq 0.\]
 Otherwise we can replace it by $V- m_V$. The proof is divided into four steps. 

\vskip0.05in

{\it Step 1: Proof of $M_q<\infty$ when $\mkm_q$ is replaced by $\id_H$}. Here we aim to show the existence of $q_0, R_0\geq 1$ such that 
\begin{align}\label{M-idH-bound}
	\sup_{t\geq 0, h\in\To}\frac{\|P_{0,t,h}^V\id_H\|_{L_{\mkm_{q}}^{\infty}(H)}}{\|P_{0, t, h}^V\id_H\|_{R_0}}<\infty.
\end{align}
In view of Proposition \ref{prop-hyperexponential-recurrence}, we pick any $q, R_0$ the numbers corresponding to $\g =\|V\|_{\infty}$, with $\tau_{h}(R_0)$ be the first hitting time of the ball $B_{R_0}^{H_2}$. Now we write 
\begin{align}\label{M-idH-bound-1}
	P_{0, t, h}^V\id_H(w) = \E_{(w, h)}\Xi_{t}^V =  \E_{(w, h)}\id_{G_{t, h}}\Xi_{t}^V + \E_{(w, h)}\id_{G_{t, h}^c}\Xi_{t}^V := I_1 + I_2, 
\end{align}
where $\Xi_{t}^V=\exp\left(\int_0^tV(H_r)dr\right)$ and $G_{t, h} = \lp\tau_{h}(R_0)>t\rp$. Since $V\geq 0$ and $\tau_{h}(R_0)>t$ on $G_{t, h}$, we obtain 
\begin{align}\label{M-idH-bound-2}
	 I_1\leq \E_{(w, h)}\Xi_{\tau_{h}(R_0)}^V\leq \E_{w}\exp\left(\g\tau_{h}(R_0)\right)\leq C\mkm_q(w)\leq C\mkm_q(w)\|P_{0, t, h}^V\id_H\|_{R_0}, 
\end{align}
by Proposition \ref{prop-hyperexponential-recurrence}. By the strong Markov inequality and the fact that $\Phi_{0, \tau_h(R_0), h}\in B_{R_0}^{H_2}$, we have 
\begin{align}\label{M-idH-bound-3}
\begin{split}
	I_2 &= \E_{(w, h)}\left[\id_{G_{t, h}^c}\Xi_{\tau_h(R_0)}^V\E_{(w, h)}\left[\exp\left(\int_{\tau_h(R_0)}^tV(H_r)dr\right)\Big|\mathcal{F}_{\tau_h(R_0)}\right]\right]\\
	& = \E_{(w, h)}\left[\id_{G_{t, h}^c}\Xi_{\tau_h(R_0)}^VP_{\tau_h(R_0), t, h}^V\id_{H}(\Phi_{0, \tau_h(R_0), h})\right]\\
	&\leq \E_{w}\left[\id_{G_{t, h}^c}\exp(\g\tau_h(R_0))P_{0, t, h}^V\id_{H}(\Phi_{0, \tau_h(R_0), h})\right]\\
	&\leq \E_{w}\left[\exp(\g\tau_h(R_0))\right]\|P_{0, t, h}^V\id_{H}\|_{R_0}\leq C\mkm_{q}(w)\|P_{0, t, h}^V\id_{H}\|_{R_0}, 
\end{split}
\end{align}
 by Proposition \ref{prop-hyperexponential-recurrence} again. By combining \eqref{M-idH-bound-1}-\eqref{M-idH-bound-3} we obtain the desired \eqref{M-idH-bound}. 

\vskip0.05in

The proof of {\it Step 1} is complete. 

\vskip0.05in

{\it Step 2: Proof of $M_q<\infty$}.  We first show the estimate for discrete time and then pass to the continuous time case by evolution property.  Noting that $\d:= e^{\g-\nu q}<1$  by Proposition \ref{prop-hyperexponential-recurrence}, hence by translation identity \eqref{translation-identity} one has  
\begin{align*}
	P_{0, k, h}^V\mkm_q(w) &= P_{0, k-1, h}^V\left(P_{0, 1,\b_{k-1}h}^V\mkm_q\right)(w)\\
	&\leq P_{0, k-1, h}^V\left( e^{\g}P_{0, 1,\b_{k-1}h}\mkm_q\right)(w) \leq e^{\g}\left(e^{- \nu q}P_{0, k-1, h}^V\mkm_q(w) + C(q) P_{0, k-1, h}^V\id_{H}\right)\\
	&\leq  \d P_{0, k-1, h}^V\mkm_q(w)+C(q,\g)P_{0, k, h}^V\id_{H}(w)
\end{align*}
where we used estimate \eqref{contraction-polynomial} in the penultimate step and monotonicity of $P_{0, k, h}^V\id_{H}$ in $k$ in the last step. Iterating this inequality we obtain 
\begin{align*}
	P_{0, k, h}^V\mkm_q(w)\leq \d^k\mkm_q(w) + \frac{C}{1-\d}P_{0, k, h}^V\id_{H}(w), 
\end{align*}
where the constant $C$ does not depend on $w, h, k$. Combining this with \eqref{M-idH-bound}, we obtain $M_q<\infty$ in the discrete time case. 

\vskip0.05in

For the continuous time case, denoting $[t]$ the integer part of $t$, by \eqref{contraction-polynomial} and \eqref{translation-identity}, we have 
\begin{align*}
	P^V_{[t], t, h}\mkm_q(w) = P^V_{0, t-[t], \b_{[t]} h}\mkm_q(w) \leq e^{\g}P_{0, t-[t], \b_{[t]} h}\mkm_q(w)\leq  e^{\g}\mkm_q(w)+C, \, w\in H, h\in\To.
\end{align*}
Hence 
\begin{align*}
	P^V_{0, t, h}\mkm_q(w) &= P^V_{0, [t], h}\left(P^V_{[t], t, h}\mkm_q\right)(w) \leq e^{\g}P^V_{0, [t], h}\mkm_q(w) + CP^V_{0, [t], h}\id_{H}(w)\\
	&\leq C\mkm_q(w)\|P_{0, [t], h}^V\id_{H}\|_{R_0}\leq C\mkm_q(w)\|P_{0, t, h}^V\id_{H}\|_{R_0}
\end{align*}
by the proved discrete time case, estimate \eqref{M-idH-bound}, and monotonicity of $t\to\|P_{0, t, h}^V\id_{H}\|_{R_0}$. Therefore $M_q<\infty$ in the continuous time as well.

\vskip0.05in

The proof of {\it Step 2} is complete. 

\vskip0.05in 

{\it Step 3: Proof of $M_{\eta}<\infty$}. As in the proof of the previous step, we only need to show the discrete time case. The passage to the continuous time is similar by invoking \eqref{eq: est1}. 

\vskip0.05in

Let $A>0$ and note that 
\begin{align}\label{M-eta-bound-1}
\begin{split}
	P_{0, k, h}^V\mkm_{\eta}(w) &= \E_{(w, h)}\left[\id_{\|\Phi_{0, k, h}\|^2\leq A}\Xi_{k}^V\mkm_{\eta}(\Phi_{0, k, h})\right] +  \E_{(w, h)}\left[\id_{\|\Phi_{0, k, h}\|^2\geq A}\Xi_{k}^V\mkm_{\eta}(\Phi_{0, k, h})\right]\\
	&:= I_{k, h} + J_{k, h}.
\end{split}
\end{align}
It follow from \eqref{M-idH-bound} that 
\begin{align*}
	P_{0, k, h}^V\id_H(w) \leq C\mkm_{q}(w)\|P_{0, k, h}^V\id_H\|_{R_0}\leq C\mkm_{\eta}(w)\|P_{0, k, h}^V\id_H\|_{R_0}, \, (w, h)\in H\times\To. 
\end{align*}
Therefore 
\begin{align}\label{M-eta-bound-2}
	I_{k, h}\leq e^{\eta A}P_{0, k, h}^V\id_H(w)\leq C\mkm_{\eta}(w)\|P_{0, k, h}^V\id_H\|_{R_0}, 
\end{align}
where $C$ is independent of $k, h, w$. By \eqref{translation-identity}, \eqref{mixing-moments} and denoting $F(w)=\|w\|^2\mkm_{\eta}(w)$ for $w\in H$, we have 
\begin{align}\label{M-eta-bound-3}
\begin{split}
	J_{k, h}\leq A^{-1}P_{0, k, h}^VF(w) &=  A^{-1}P_{0, k-1, h}^V\left( P_{0, 1,\b_{k-1}h}^VF\right) (w)\\
	&\leq e^{\g} A^{-1}P_{0, k-1, h}^V\left( P_{0, 1,\b_{k-1}h}F\right) (w)\\
	&\leq A^{-1}CP_{0, k-1, h}^V\mkm_{\eta}(w). 
\end{split}
\end{align}
Combining \eqref{M-eta-bound-1}-\eqref{M-eta-bound-3} and choosing $A$ large such that $\d:= A^{-1}C<1$, we obtain 
\begin{align*}
	P_{0, k, h}^V\mkm_{\eta}(w)\leq C\mkm_{\eta}(w)\|P_{0, k, h}^V\id_H\|_{R_0} + \d P_{0, k-1, h}^V\mkm_{\eta}(w). 
\end{align*}
By iteration, one has 
\begin{align*}
	P_{0, k, h}^V\mkm_{\eta}(w)\leq \frac{C}{1-\d}\mkm_{\eta}(w)\|P_{0, k, h}^V\id_H\|_{R_0} + \d^q \mkm_{\eta}(w). 
\end{align*}
Hence $M_{\eta}<\infty$.

\vskip0.05in

The proof of {\it Step 3} is complete. 

\vskip0.05in

{\it Step 4: Proof of $M_{H_2}<\infty$}. This follows from \eqref{contraction-polynomial} and $M_{\eta}<\infty$ proved in the previous step. Indeed, since $\mkm_q\leq C\mkm_{\eta}$, one has  
\begin{align*}
	P_{0, k, h}^VN_2(w) &= P_{0, k-1, h}^V(P_{0, 1, \b_{k-1} h}^VN_2)(w)\leq e^{\g}P_{0, k-1, h}^V(P_{0, 1, \b_{k-1} h}N_2)(w) \\
	&\leq CP_{0, k-1, h}^V\mkm_{\eta}(w)\leq C\mkm_{\eta}(w)\|P_{0, k-1, h}^V\id_H\|_{R_0}\leq C\mkm_{\eta}(w)\|P_{0, k, h}^V\id_H\|_{R_0}. 
\end{align*}
Hence $M_{H_2}<\infty$, which finishes the proof of {\it Step 4}. 

\vskip0.05in

The proof of Proposition \ref{Growth-Condition} is complete. 
\end{proof}

\vskip0.05in

\begin{proof}[Proof of Proposition \ref{prop-hyperexponential-recurrence}]
The proof is divided into two steps. First we prove hyper-exponential recurrence in $H$ and then we show the desired recurrence in $H_2$. 

\vskip0.05in

{\it Step 1:  Hyper-exponential recurrence in $H$.} For any $r>0$, denote $\overline{\tau}_h(r)$ the first hitting time 
\begin{align*}
	\overline{\tau}_h(r)= \inf\lp t\geq 0:  \Phi_{0, t, h}\in B_r^{H}\rp.
\end{align*}
We aim to show that for any $\g>0$, there is $q_0=q_0(\g)$ so that for any $q\geq q_0$, there are positive numbers $ r, \overline{C}$ depending on $q$ such that
 \begin{align}\label{hyper-H-0}
	\sup_{h\in\To}\E_{w}\exp\left(\g\overline{\tau}_h(r)\right)\leq \overline{C}\mkm_q(w), \, w\in H. 
\end{align}
Denote by $N_0^{2q}(w) = \|w\|^{2q}, \,w\in H$. It follows from \eqref{contraction-polynomial} that 
\begin{align}\label{hyper-H-1}
	P_{0, 1, h}N_0^{2q}(w) = \E_{w}\|\Phi_{0, 1, h}\|^{2q}\leq \d(\|w\|^{2q}\vee r), \, (w, h)\in H\times\To, 
\end{align}
where $\d = e^{-\nu q}<1$ and $r=e^{\nu q}C>1$, with the constant $C=C(q)$ from \eqref{contraction-polynomial}. Let 
\begin{align*}
	R_{k, h}(w) = \E_w\left(\id_{\{\overline{\tau}_h(r)>k\}}\|\Phi_{0, k, h}\|^{2q}\right).
\end{align*}
Note that $\id_{\{\overline{\tau}_h(r)>k\}}\geq \id_{\{\overline{\tau}_h(r)>k+1\}}$ and on the event $\{\overline{\tau}_h(r)>k\}$ one has $\|\Phi_{0, k, h}\|^{2q}>r^{2q}>r$. Therefore by Markov inequality, \eqref{translation-identity} and \eqref{hyper-H-1}, we have 
\begin{align*}
	R_{k+1, h}(w) &\leq  \E_w\left(\id_{\overline{\tau}_h(r)>k}\|\Phi_{0, k+1, h}\|^{2q}\right) 
	=  \E_w\left(\id_{\overline{\tau}_h(r)>k}\E_w\left[\|\Phi_{0, k+1, h}\|^{2q}|\mathcal{F}_k\right]\right)\\
	& = \E_w\left(\id_{\overline{\tau}_h(r)>k}P_{k, k+1, h}N_0^{2q}(\Phi_{0, k, h})\right) 
	= \E_w\left(\id_{\overline{\tau}_h(r)>k}P_{0, 1, \b_kh}N_0^{2q}(\Phi_{0, k, h})\right) \\
	&\leq \d\E_w\left(\id_{\overline{\tau}_h(r)>k}\|\Phi_{0, k, h}\|^{2q}\right)=\d R_{k, h}(w).
\end{align*}
By iteration, one has 
\begin{align*}
	R_{k, h}(w) \leq \d^{k}\|w\|^{2q},
\end{align*}
which in turn implies
\begin{align*}
	\mathbf{P}_{w}(\overline{\tau}_h(r)>k) \leq r^{-2q}R_{k, h}(w) \leq \d^{k}r^{-2q}\|w\|^{2q}, \, k\geq 0, w\in H.  
\end{align*}
Hence for any $\g>0$, by choosing $q>\g/\nu$ large such that $e^{\g}\d=e^{\g- \nu q}<1$, we obtain the desired \eqref{hyper-H-0}:
\begin{align}\label{hyper-H-1}
\begin{split}
	\E_{w}\exp\left(\g\overline{\tau}_h(r)\right)
	&\leq 1+\sum_{k=1}^{\infty}e^{\g k}\mathbf{P}_{w}(\overline{\tau}_h(r)>k-1)\\
	&\leq 1+r^{-2q}\|w\|^{2q}\sum_{k=1}^{\infty}e^{\g k}\d^{k-1}\leq C\mkm_q(w), 
\end{split}
\end{align}
for some constant $C$ independent of $h\in\To$. 

\vskip0.05in

{\it Step 2:  Hyper-exponential recurrence in $H_2$}. First it follows from \eqref{contraction-polynomial} and Markov inequality that for any $p\in (0, 1)$, there is a large $R>0$ such that 
\begin{align}\label{hyper-H2-1}
	P_{0, 1, h}\id_{H\setminus B_R^{H_2}}(w) = \mathbf{P}_{w}\left(\|\Phi_{0, 1, h}\|_2\geq R\right)\leq R^{-2}\E_{w}\|\Phi_{0, 1, h}\|_2^2\leq p, 
\end{align}
for all $h\in\To$ and $w\in B_r^H$, where $r$ is the radius from \eqref{hyper-H-0} with corresponding $\overline{\tau}_h(r)$. Define recursively the following sequence of stopping times 
\begin{align*}
	\tau_{0, h}^{\prime} = \overline{\tau}_h(r), \, \tau_{n, h}^{\prime} = \inf\lp t\geq \tau_{n-1, h}^{\prime}+1: \Phi_{0, t, h}\in B_r^H\rp,  \, \text{ and } \tau_{n, h} = \tau_{n, h}^{\prime}+1, \, \text{ for } n\geq 1. 
\end{align*}
Note that $\tau_{0, h}^{\prime}$ is the first hitting time of the ball $B_r^H$, and $\tau_{n, h}^{\prime}$ is the first recurrence time after $\tau_{n-1, h}^{\prime}+1$.  Also let 
\begin{align*}
	\hat{n} = \min\lp n\geq 0: \Phi_{0, \tau_{n, h}, h}\in B_{R}^{H_2} \rp
\end{align*}
be the first time the solution entering the ball $B_{R}^{H_2}$ among the sequence $\tau_{n, h}$. Note that for any $k, M\geq 1$, since $\lp\tau_h(R)>\tau_{k, h}\rp\subset \lp \hat{n}>k\rp$, we have 
\begin{align}\label{hyper-H2-4}
\begin{split}
	\mathbf{P}_{w}(\tau_h(R)\geq M) 
	&= \mathbf{P}_{w}(\tau_h(R)\geq M, \tau_{k, h}< M) + \mathbf{P}_{w}(\tau_h(R)\geq M, \tau_{k, h}\geq M)\\
	&\leq \mathbf{P}_{w}(\tau_h(R)>\tau_{k, h}) +  \mathbf{P}_{w}(\tau_{k, h}\geq M)\\
	&\leq \mathbf{P}_{w}(\hat{n}>k) +  \mathbf{P}_{w}(\tau_{k, h}\geq M).
\end{split}
\end{align}
In what follows we estimate the probabilities on the right hand side. 
\vskip0.05in

It follows from translation identity \eqref{translation-identity}, estimate \eqref{hyper-H2-1}, the fact $\Phi_{0, \tau_{k, h}^{\prime}, h}\in B_r^H$ and the strong Markov property that 
\begin{align*}	
	\mathbf{P}_{w}(\hat{n}>k)
	&=\E_{w}\left[\id_{\lp\hat{n}>k-1\rp}\id_{H\setminus B_R^{H_2}}(\Phi_{0, \tau_{k, h}, h})\right]\\
	&=\E_{w}\left[\id_{\lp\hat{n}>k-1\rp}\E_{w}\left[\id_{H\setminus B_R^{H_2}}(\Phi_{0, \tau_{k, h}, h})|\mathcal{F}_{\tau_{k, h}^{\prime}}\right]\right]\\
	&=\E_{w}\left[\id_{\lp\hat{n}>k-1\rp}P_{\tau_{k, h}^{\prime}, \tau_{k, h},h}\id_{H\setminus B_R^{H_2}}(\Phi_{0, \tau_{k, h}^{\prime}, h})\right]\\
	&= \E_{w}\left[\id_{\lp\hat{n}>k-1\rp}P_{0, 1, \b_{\tau_{k, h}^{\prime}}h}\id_{H\setminus B_R^{H_2}}(\Phi_{0, \tau_{k, h}^{\prime}, h})\right]\leq p\E_{w}\left[\id_{\lp\hat{n}>k-1\rp}\right] = p  \mathbf{P}_{w}(\hat{n}>k-1), 
\end{align*}
which, by iteratively using the strong Markov property, implies 
\begin{align}\label{hyper-H2-2}
	\mathbf{P}_{w}(\hat{n}>k)\leq p^k, \, \text{ for } w\in H, k\geq 0. 
\end{align}
In a similar fashion, by Markov's inequality, the recurrence \eqref{hyper-H-0} with constant $\overline{C}$, and the strong Markov property, we have for any $M\geq 1$, 
\begin{align*}
	\mathbf{P}_{w}(\tau_{k, h}\geq M)
	&\leq e^{-\g M}\E_w \exp(\g\tau_{k, h}) 
	= e^{-\g M}\E_w\left[\E_w\left[\exp(\g(\tau_{k, h}^{\prime}+1))|\mathcal{F}_{\tau_{k-1, h}}\right] \right]\\
	& = e^{-\g M}e^{\g}\E_w\left[\exp(\g\tau_{k-1, h})\E_{\Phi_{0, \tau_{k-1,h},h}}\exp(\g\overline{\tau}_{\b_{\tau_{k-1,h}}h}(r))\right]\\
	&\leq e^{-\g M}\overline{C}e^{\g}\E_w\left[\exp(\g\tau_{k-1, h})\mkm_q(\Phi_{0, \tau_{k-1,h},h})\right]\\
	& = e^{-\g M}\overline{C}e^{\g}\E_w\left[\exp(\g\tau_{k-1, h})\E_{\Phi_{0, \tau_{k-1,h}^{\prime},h}}\mkm_q(\Phi_{0, \tau_{k-1,h},h})\right]\\
	&\leq e^{-\g M}\overline{C}e^{\g}(e^{- \nu q}r^{2q}+C(q))\E_w\exp(\g\tau_{k-1, h}),
\end{align*}
where we used the fact that $\tau_{k-1, h}$ is $\mathcal{F}_{ \tau_{k-1,h}^{\prime}}$ measurable in the last equality, and 
\eqref{contraction-polynomial} with $\Phi_{0, \tau_{k-1,h}^{\prime},h}\in B_r^H$ in the last inequality.  Hence by iteratively using the strong Markov inequality, we obtain 
\begin{align}\label{hyper-H2-3}
	\mathbf{P}_{w}(\tau_{k, h}\geq M)\leq e^{-\g M}C^k\mkm_q(w),\, w\in H, h\in\To, 
\end{align}
with $C=\overline{C}e^{\g}(e^{- \nu q}r^{2q}+C(q))$ independent of $w, M, h$. 

\vskip0.05in

Now combining \eqref{hyper-H2-4}-\eqref{hyper-H2-3}, one has 
\begin{align*}
	\mathbf{P}_{w}(\tau_h(R)\geq M) \leq p^k + e^{-\g M}C^k\mkm_q(w). 
\end{align*}
Choosing $k = \frac{\g M}{2\ln C}$  and $R$ large such that $p = 1/C$ by \eqref{hyper-H2-1}, then we have 
\begin{align}\label{hyper-H2-5}
	\mathbf{P}_{w}(\tau_h(R)\geq M) \leq 2e^{-\g M/2}\mkm_q(w)
\end{align}
for any $\g>0, M\geq 1$ and $(w, h)\in H\times\To$. This implies the desired hyper-exponential recurrence \eqref{hyper-H2-0} by following the same reasoning as in \eqref{hyper-H-1}. Note that to obtain \eqref{hyper-H2-0} with any $\g>0$ we can replace the $\g$ in \eqref{hyper-H2-5} by $4\g$. 

\vskip0.05in

The proof of Proposition \ref{prop-hyperexponential-recurrence} is complete. 
\end{proof}

\vskip0.05in

\subsection{Uniform  Irreducibility}
In this subsection we prove the uniform irreducibility, which follows from the well-known controllability results of Agrachev and Sarychev \cite{AS05,AS06,GHM18}, the cocycle property \eqref{Feynman-Kac-cocycle} and standard compactness arguments. 

\begin{proposition}[Uniform  Irreducibility ]\label{Uniform-Irreducibility}
Let $V\in B_b(H\times\To)$. Given any  $\r, r>0, R\geq 1$, as long as 
 $\ell\geq 2\ln R$, then there is $p=p(V, \r, r, \ell)>0$ such that 
\begin{align}\label{eq-Uniform-Irreducibility}
\inf_{(w_0, h, w)\in B_{R}^{H}\times\To\times B_{\r}^{H_2}}P_{0, \ell, h}^V(w_0, B^{H,H_2}_{r}(w))\geq p
\end{align}
Moreover, there is a large $L_0=L_0(\ell, R)$ such that for any $L\geq L_0$, there holds 
\begin{align}\label{irreducible-large-H2ball}
\inf_{(w_0, h, w)\in B_{R}^{H}\times\To\times B_{\r}^{H_2}}P_{0, \ell, h}^V(w_0, B^{H, B_{L}^{H_2}}_{r}(w))\geq p/2.
\end{align}
If $\ds m_V=\inf_{w\in H} V(w)\geq 0$, then one can choose $p = p(\r, r)$ that is independent of $V, \ell$. 
\end{proposition}
\begin{proof}
First of all, we claim the existence a large time $t_0$ depending on $R$ and a large number $r_0>0$ independent of $w_0, R, h, t_0$, such that 
\begin{align} \label{UI-1}
\inf_{(w_0, h)\in B^{H}_{R}\times\To}P_{0, t_0, h}(w_0, B_{r_0}^{H_2})\geq 1/2.
\end{align}
 Let $N(w) = \|w\|^2$ and $N_2(w) = \|w\|_2^2$.  By estimate \eqref{contraction-polynomial}, we have for any $t\geq 1$, and $w_0\in B^{H}_{R}$, 
\begin{align*}
P_{0, t, h}N_2(w_0) &= P_{0, t-1, h}(P_{0, 1,\b_{t-1}h}N_2)(w_0) \leq C\left(P_{0, t-1, h}N^6(w_0)+1\right)\\
&\leq C(e^{-6(t-1)}\|w_0\|^{12}+1)\leq C(e^{-6(t+1)}R^{12}+1)
\end{align*}
where $C$ is a constant independent of $t, h, R$. Since $P_{0, t, h}(w_0, \cdot)$ is concentrated on $H_2$, by choosing $t_0$ large so that $e^{-6(t_0+1)}R^{12}\leq 1$ and choosing a constant $r_0>2\sqrt{C}$, we have by Markov inequality that 
\begin{align}\label{Br0H2}
P_{0, t_0, h}(w_0, B_{r_0}^{H_2})&=1-P_{0, t_0, h}(w_0, H_2\setminus B_{r_0}^{H_2})
\geq 1- \frac{(P_{0, t_0, h}N_2(w_0))^2}{r_0^2}\geq 1/2.
\end{align}

\vskip0.05in 

Next we show the existence of $p_0>0$, depending only on $r_0, r, \r$ such that 
\begin{align}\label{UI-2}
\inf_{(w_0, h, w)\in B_{r_0}^{H_2}\times\To\times B_{\r}^{H_2}}P_{0,1, h}(w_0, B^{H,H_2}_{r}(w))\geq p_0.
\end{align}
 Let 
\[\O_1 =\{X\in C([0,1], \mathbb{R}^d) : X(0) =0\}\]
be the restriction of $\O$ on $[0, 1]$.  In view of the controllability results of Agrachev and Sarychev \cite{AS05,AS06,GHM18} (see \cite{L22} for a verification in the time inhomogeneous case), the Navier-Stokes system \eqref{NS} is approximately controllable. That is, for any $(w_0, v, h)\in H^2\times\To$ and $r>0$, if we replace $f(t, x)$ by $f(\b_th, x)$ in $\eqref{NS}$, then there is a piecewise smooth map $X = X_{w_0, v, h, r}\in \O_1$ such that $\Phi_{0, 1,h}(w_0, X) \in \ob_{r/4}^{H}(v),$
where $\Phi_{0, t, h}(w_0, X)$ is the solution of \eqref{NS} with symbol $h$ and $W(t)$ replaced by $X(t)$.  
Since $\Phi_{0, 1, h}(w_0, \cdot): \O_1\to H$ is continuous, there is an $\varepsilon = \varepsilon(w_0, v, h, r)>0$ such that for any $Y\in\ob_{\varepsilon}^{\O_1}(X)$ one has 
\[\|\Phi_{0, 1,h}(w_0, Y) - \Phi_{0, 1,h}(w_0, X)\|<r/4.\]
Hence by triangle inequality and regularizing property of the Navier-Stokes systems,  we obtain that $Y\in\ob_{\varepsilon}^{\O_1}(X)$ implies $\Phi_{0, 1,h}(w_0, Y) \in \ob_{r/2}^{H,H_2}(v)$. 

\vskip0.05in 

Since $B^{H_2}_{\rho}$ is compact in $H$, its open cover $\lp\ob_{r/2}^{H, H_2}(v)\rp_{v\in B^{H_2}_{\rho}}$ has a finite subcover $\lp\ob_{r/2}^{H, H_2}(v_k)\rp_{k=1}^K$ where the positive integer $K=K(r, \r)$. Let $\varepsilon_k = \varepsilon(w_0, v_k, h, r)$ and $X_k$ be the path from the controllability result corresponding to $v_k$. Note that $\mathbf{P}(\ob_{\varepsilon_k}^{\O_1}(X_k))>0$ for each $1\leq k\leq K$ since the Wiener measure has a full support. Hence 
\[p(w_0, h, r, \r): = \min_{1\leq k\leq K}\mathbf{P}(\ob_{\varepsilon_k}^{\O_1}(X_k))>0.\]
Now for any $w\in B^{H_2}_{\rho}$, there is some $v_k$ such that $w\in \ob_{r/2}^{H, H_2}(v_k)$. Therefore $\ob_{r/2}^{H, H_2}(v_k)\subset B^{H, H_2}_{r}(w)$. It then follows that 
\begin{align*}
P_{0,1, h}(w_0, B^{H,H_2}_{r}(w))\geq P_{0,1, h}(w_0,\ob_{r/2}^{H, H_2}(v_k)) \geq \mathbf{P}(\ob_{\varepsilon_k}^{\O_1}(X_k))\geq p(w_0, h, r, \r), \quad (w_0, h)\in H\times\To. 
\end{align*}
This combined with a continuity and compactness argument similar to the proof of Lemma 4.4 in \cite{L22} implies the existence of a number $p_0=p_0(r_0, r,\r)>0$ such that 
\[\inf_{(w_0, h)\in B_{r_0}^{H_2}\times\To}P_{0,1, h}(w_0, B^{H,H_2}_{r}(w))\geq p_0, \]
where $B_{r_0}^{H_2}$  is from \eqref{Br0H2}. 

\vskip0.05in 

Lastly, it follows from the Chapman-Kolmogorov relation and inequalities \eqref{UI-1} and  \eqref{UI-2} that 
\begin{align}\label{irreducible-P}
\begin{split}
P_{0, t_0+1, h}(w_0, B^{H,H_2}_{r}(w)) &= \int_{H}P_{0, 1,\b_{t_0}h}(u, B^{H,H_2}_{r}(w)) P_{0, t_0, h}(w_0, du)\\
&\geq  \int_{B_{r_0}^{H_2}}P_{0, 1,\b_{t_0}h}(u, B^{H,H_2}_{r}(w)) P_{0, t_0, h}(w_0, du)\geq p_0/2
\end{split}
\end{align}
for any $(w_0, h, w)\in B_{R}^{H}\times\To\times B_{\r}^{H_2}$. 
Hence 
\[P^V_{0, t_0+1, h}(w_0, B^{H,H_2}_{r}(w))\geq e^{(t_0+1)m_V}P_{0, t_0+1, h}(w_0, B^{H,H_2}_{r}(w)) \geq \frac{p_0e^{(t_0+1)m_V}}{2}.\]
The inequality \eqref{eq-Uniform-Irreducibility} then follows with $\ell = t_0+1\geq 2\ln R$ and $p = \frac{p_0e^{(t_0+1)m_V}}{2}$.  

\vskip0.05in

For any $L>0$, we have by  estimate \eqref{contraction-polynomial}
\begin{align*}
P_{0, t_0+1, h}(w_0, H\setminus B_{L}^{H_2})\leq L^{-2}\E_{w_0}\|\Phi_{0,t_0+1, h}\|_2^2\leq CL^{-2},
\end{align*}
 where the constant $C=C(R, t_0)$ does not depend on $h$. There is $L_0=L_0(R)$ such that for all $L\geq L_0$, 
 \[\sup_{(w_0, h)\in B_{R}^H\times\To}P_{0, t_0+1, h}(w_0, H\setminus B_{L}^{H_2})\leq p_0/4.\]
 Hence by \eqref{irreducible-P}, we deduce 
\begin{align*}
P_{0, t_0+1, h}(w_0, B^{H,B_{L}^{H_2}}_{r}(w)) 
&= P_{0, t_0+1, h}(w_0, B^{H, H_2}_{r}(w)) - P_{0, t_0+1, h}(w_0, B^{H,H_2\setminus B_{L}^{H_2}}_{r}(w))\\
&\geq p_0/2-p_0/4=p_0/4, \quad \text{ for all } (w_0, h, w)\in B_{R}^{H}\times\To\times B_{\r}^{H_2}. 
\end{align*}
Inequality \eqref{irreducible-large-H2ball} then follows. The proof is complete. 
\end{proof}
\subsection{Uniform Feller Property}
In this subsection we prove the uniform Feller property. We follow the arguments presented in \cite{NPX23} and use the cocycle property of the inhomogeneous Feynman-Kac family.
\begin{proposition}[Uniform Feller Property]\label{Uniform-Feller}
Assume $A_{\infty}=H, f\in C_b(\R, H_2)$ in \eqref{Settings-NS-family}. For any $V\in C_{b}^{1,0}(H\times\To)$ and $\phi\in C_b^1(H\times\To)$ there is an $R_0=R_0(V)>0$ such that for any $R>0$ the family 
\[\left\{\frac{P_{0, t, h}^V\phi(\cdot, \b_th)}{\|P_{0, t, h}^V\id_H\|_{R_0}}, t\geq0, h\in\To\right\}\]
is uniformly Lipschitz on $B_R^{H}$ with a uniform bound on the Lipschitz constants depending on 
\[R, \|\phi\|_{\infty}, \|V\|_{\infty}, \sup_{h\in\To}\|\nabla\phi(\cdot, h)\|_{\infty}, \sup_{h\in\To}\|\nabla V(\cdot, h)\|_{\infty}.\]
 In particular, the family is uniformly equicontinuous on $(B_R^{H_2}, \|\cdot\|)$ for any $R>0$.  
\end{proposition}
Here $\nabla$ is the gradient in the $H$ component. For convenience we will denote 
\[\|\nabla V\|_{\infty} := \sup_{h\in \To}\|\nabla V(\cdot, h)\|_{\infty}, \, \text{ for } V\in C_{b}^{1,0}(H\times\To). \]
Proposition \ref{Uniform-Feller} is a direct consequence of the following estimate.
\begin{proposition}\label{gradient-estimate-modulo-growth-by-potential}
Assume $A_{\infty}=H, f\in C_b(\R, H_2)$ in \eqref{Settings-NS-family}. Then for any $V\in C_{b}^{1,0}(H\times\To), \phi\in C_b^1(H\times\To)$ and $c, \eta>0$, there are $R_0=R_0(V)$ and a constant $C$ independent of $h$ such that 
\begin{align}\label{prop-gradient}
	\|\nabla \P^V_t\phi(w, h)\|\leq C\exp(\eta\|w\|^2)\|P_{0, t, h}^V\id\|_{R_0}\left(\|\phi\|_{\infty}+e^{-ct}\|\nabla\phi\|_{\infty}\right), 
\end{align}
for any $(w, h)\in H\times\To$ and $t\geq 0$.  
\end{proposition}
This estimate is proved through Malliavin calculus. In the time homogeneous case this type of gradient estimate was proved in \cite{NPX23} recently. Since the current work is in a time inhomogeneous setting,  we include a detailed proof below, which basically follows from the scheme in \cite{NPX23}.   The time inhomogeneity is overcome by the translation identity \eqref{translation-identity} and the cocycle property \eqref{Feynman-Kac-cocycle}. The proof of Proposition \ref{gradient-estimate-modulo-growth-by-potential} is giving in subsection \ref{Proof-of-gradient-estimate}, after we give some basic facts about Malliavin calculus in subsection \ref{Malliavin-calculus-facts}. 

\subsubsection{Malliavin calculus preliminaries}\label{Malliavin-calculus-facts}
Recall that $\Phi_{0, t, h}(w_0, W) = \Phi_{0, t, h}(w_0, W\id_{[0, t]})$ is the solution of the stochastic Navier-Stokes equation at time $t$, where $W$ is the Wiener process. The Malliavin derivative of this random variable in the direction of $v\in \mathrm{L}^2([0,\infty),\mathbb{R}^d)$ is 
\begin{align*}
	D^v\Phi_{0, t, h}(w_0) : = \lim_{\varepsilon\rightarrow 0}\frac{\mathrm{\Phi}_{0, t, h}(w_0, W+\varepsilon V)-\mathrm{\Phi}_{0, t, h}(w_0, W)}{\varepsilon}, 
\end{align*}
where $V(t) = \int_0^tv(r)dr$. The above limit exists almost surely and one has 
\begin{align*}
	D^v\Phi_{0, t, h}(w_0) = \int_0^t J_{r, t, h}Gv(r)dr, 
\end{align*}
where $J_{r, t, h}$ is the Jacobian that solves the linearized equation 
\begin{align*}
	 \partial_tJ_{r, t, h}\xi = \nu\mathrm{\Delta} J_{r, t, h}\xi+\widetilde{B}(\Phi_{0, t, h} ,J_{r, t, h}\xi),\quad  t>r, \quad J_{r ,r, h}\xi = \xi,
\end{align*}
for $\xi\in H$. See \cite{HM11, Nua06} and references therein for more about Malliavin calculus. 

\vskip0.05in

For any $0\leq \tau \leq t$, the Malliavin operator is defined as $M_{\tau, t, h}: = A_{\tau, t, h}A_{\tau, t, h}^*$, where the random operator  $A_{\tau, t, h}: \mathrm{L}^2([0,\infty),\mathbb{R}^d)\to H$ is given by
\begin{align*}
	A_{\tau, t, h}v:=  \int_{\tau}^t J_{r, t, h}Gv(r)dr,  \, v\in \mathrm{L}^2([0,\infty),\mathbb{R}^d), 
\end{align*}
and $A_{\tau, t, h}^*$ is its adjoint by the relation 
\[\la A_{\tau, t, h}v, u\ra_{H}=\la v, A_{\tau, t, h}^*u\ra_{\mathrm{L}^2([0,\infty),\mathbb{R}^d)}, \, u\in H, v\in \mathrm{L}^2([0,\infty),\mathbb{R}^d).\]
It is known that the Tikhonov regularization  $\widetilde{M}_{\tau, t, h}:= M_{\tau, t, h}+\beta$, is  invertible for any $\beta>0$. 

\vskip0.05in

Denote $\Xi_{t}^V=\exp\left(\int_0^tV(H_r)dr\right)$, and $v_{0, t}=v\id_{[0, t]}$, and set   
\[\mathfrak{R}_{t, h} = J_{0, t, h} \xi-A_{0, t, h} v_{0, t},\] 
which is the error between the infinitesimal variation on the Wiener path $W$ by $v$ and the variation on the initial condition by $\xi$ of the solution process.  Applying the integration by parts formula, chain rule and product rule for the Malliavin derivative \cite{Nua06}, for any  $w, \xi\in H$ with $\|\xi\|=1$, and $\E = \E_{(w, h)}$,  we have
\begin{align} \label{gradient-set-up}
\begin{split}
\Big\langle\nabla &\P_{t}^V \phi(w, h), \xi\big\rangle
	=\E\left[\Xi_{t}^V\phi(H_t)\int_0^t\nabla V(H_r)J_{0, r, h}\xi dr + \Xi_{t}^V\nabla\phi(H_t)J_{0, t, h}\xi\right]\\
	&\begin{aligned} 
		&= \E\left[\Xi_{t}^V\phi(H_t)\int_0^t\nabla V(H_r)D^v\Phi_{0, r, h}dr + \Xi_{t}^V\nabla\phi(H_t)D^v\Phi_{0, t, h}\right]\\
		&\qquad+\E\left[\Xi_{t}^V\phi(H_t)\int_0^t\nabla V(H_r)\mathfrak{R}_{r, h} dr + \Xi_{t}^V\nabla\phi(H_t)\mathfrak{R}_{t, h}\right]
	\end{aligned}\\
	& \begin{aligned}
		&= \E\left[D^v\left(\Xi_{t}^V\phi(H_t)\right)\right]+\E\left[\Xi_{t}^V\phi(H_t)\int_0^t\nabla V(H_r)\mathfrak{R}_{r, h} dr\right]+\E\left[\Xi_{t}^V\nabla\phi(H_t)\mathfrak{R}_{t, h}\right]
	\end{aligned}\\
	& \begin{aligned}
		&= \E\left[\Xi_{t}^V\phi(H_t)\int_0^tv(r)dW(r)\right]+\E\left[\Xi_{t}^V\phi(H_t)\int_0^t\nabla V(H_r)\mathfrak{R}_{r, h} dr\right]+ \E\left[\Xi_{t}^V\nabla\phi(H_t)\mathfrak{R}_{t, h}\right]
	\end{aligned}\\
	&:= I_1+I _2+ I_3, 
\end{split}
\end{align}
where we used the fact that $D^v\Phi_{0, t, h} = A_{0, t}v_{0, t}$. The gradient inequality \eqref{prop-gradient} will be established if one can choose an appropriate control $v$ to have desired bounds on $I_j, j=1, 2, 3$. Here we choose essentially the same control $v$ as in \cite{HM11}, 
\begin{align*}
v(r) = \left\{
\begin{array}{ll}
A_{2n}^*\widetilde{M}_{2n}^{-1}J_{2n}\mathfrak{R}_{2n, h}&\text{ for } r\in[2n, 2n+1), n\geq 0,\\
0&\text{ for }r\in[2n+1, 2n+2), n\geq 0,
\end{array}\right.
\end{align*}
where $J_{n}=J_{n, n+1, h}$,  $A_{n}=A_{n, n+1, h}$, $M_{n}=A_{n} A_{n}^{*}$, $\widetilde{M}_{n}=\beta+M_{n}$. Note that the control is not always adapted, hence the integral in $I_1$ is understood as a Skorokhod integral. 

\vskip0.05in

The following result follows from the framework developed in \cite{HM11}. We refer interested readers to \cite{L22} for a verification in the current setting. 
\begin{proposition}\label{errorbounds}
For any  $\eta>0$  and $a>0$ there are constants $\b>0$ and $C = C(\eta, a)$, so that
\begin{align}\label{error-decay}
	&\mathbf{E}\|\mathfrak{R}_{t}\|^2 \leq C\exp(\eta\|w\|^2)e^{-2at}, \, \forall t\geq 0,\\\label{skorokhod}
	&\mathbf{E}\left|\int_{2n}^{t} v(r) d W(r)\right|^{2} \leq C\exp(\eta\|w\|^2)e^{-2an}, \, \forall t\in [2n, 2n+2], 
\end{align}
for  all $n\geq 0$, $h\in\To, w\in H$.

\end{proposition}

\subsubsection{Proof of Proposition \ref{gradient-estimate-modulo-growth-by-potential}}\label{Proof-of-gradient-estimate}
We now prove Proposition \ref{gradient-estimate-modulo-growth-by-potential} by combining inequality \eqref{gradient-set-up}, Proposition \ref{errorbounds} and the growth estimates from Proposition \ref{Growth-Condition}, through giving desired estimates on $I_j, j=1,2,3$ in inequality \eqref{gradient-set-up}. Let $R_0$ and $q$ be the numbers as in Proposition \ref{Growth-Condition} and assume $V\geq 0$ without loss of generality. The proof is divided into three steps. 

\vskip0.05in

{\it Step 1: Estimate on $I_1$}. The idea is to use the decay in \eqref{skorokhod} to control the growth resulted by the potential $V$ on each single interval.  Let $N_t$ be the largest integer such that $2N_t\leq t$. We can rewrite $I_1$ as follows, 
\begin{align*}
	I_1 = \E\left[\Xi_{t}^V\phi(H_t)\int_0^tv(r)dW(r)\right] 
	&= \sum_{n=1}^{N_t}\E\left[\Xi_{t}^V\phi(H_t)\int_{2n-2}^{2n}v(r)dW(r)\right]+ \E\left[\Xi_{t}^V\phi(H_t)\int_{2N_t}^tv(r)dW(r)\right]\\
	&:= I_{1, 1} + I_{1, 2}. 
\end{align*}
By the construction of $v$, we know that for each $1\leq n\leq N_t$, the integral $\int_{2n-2}^{2n}v(r)dW(r)$ is $\F_{2n}$ measurable. Therefore by the Markov property, we have 
\begin{align}\label{I11-0}
\begin{split}
	\E\left[\Xi_{t}^V\phi(H_t)\int_{2n-2}^{2n}v(r)dW(r)\right] 
	&= \E\left[\int_{2n-2}^{2n}v(r)dW(r)\E\left[\Xi_{t}^V\phi(H_t)|\F_{2n}\right]\right] \\
	& = \E\left[\Xi_{2n}^V\int_{2n-2}^{2n}v(r)dW(r)\P_{t-2n}^{V}\phi(\H_{2n})\right].
\end{split}
\end{align}

By Proposition \ref{Growth-Condition},
\begin{align*}
	\left|\P_{t-2n}^{V}\phi(w, h)\right| = \left|P^V_{0, t-2n, h}\phi(\cdot, \b_{t-2n}h)(w)\right|
	&\leq \|\phi\|_{\infty}P^V_{0, t-2n, h}\id_{H}(w)\\
	&\leq C\|\phi\|_{\infty}\mkm_q(w)\|P_{0, t-2n, h}^V\id_{H}\|_{R_0}.
\end{align*}
Therefore by $V\geq 0$ and translation identity \eqref{translation-identity}, 
\begin{align}\label{I11-special}
\begin{split}
	\left|\P_{t-2n}^{V}\phi(\H_{2n})\right|\leq C\|\phi\|_{\infty}\mkm_q(\Phi_{0, 2n, h})\|P_{0, t-2n, \b_{2n}h}^V\id_{H}\|_{R_0}
	&= C\|\phi\|_{\infty}\mkm_q(\Phi_{0, 2n, h})\|P_{2n, t, h}^V\id_{H}\|_{R_0}\\
	&\leq C\|\phi\|_{\infty}\mkm_q(\Phi_{0, 2n, h})\|P_{0, t, h}^V\id_{H}\|_{R_0}.
\end{split}
\end{align}

Combining this with \eqref{I11-0}, estimate \eqref{contraction-polynomial}, H\"older's inequality, and \eqref{skorokhod} we obtain 
\begin{align*}
	\Bigg|\E\Big[\Xi_{t}^V\phi(H_t)\int_{2n-2}^{2n}v(r)&dW(r)\Big]\Bigg|
	\leq C\|\phi\|_{\infty}e^{2\|V\|_{\infty}n}\|P_{0, t, h}^V\id_{H}\|_{R_0}\left|\E\left[\int_{2n-2}^{2n}v(r)dW(r)\mkm_q(\Phi_{0, 2n, h})\right]\right|\\
	&\leq C\|\phi\|_{\infty}e^{2\|V\|_{\infty}n}\|P_{0, t, h}^V\id_{H}\|_{R_0}\left(\E\left|\int_{2n-2}^{2n}v(r)dW(r)\right|^2\right)^{\frac12}\left(\E\mkm_{2q}(\Phi_{0, 2n, h})\right)^{\frac12}\\
	&\leq C\|\phi\|_{\infty}e^{(2\|V\|_{\infty}-a)n}\|P_{0, t, h}^V\id_{H}\|_{R_0}\mkm_{\eta}(w).
\end{align*}
As a result, 
\begin{align*}
	\left|I_{1, 1}\right| \leq C\|\phi\|_{\infty}\|P_{0, t, h}^V\id_{H}\|_{R_0}\mkm_{\eta}(w)\sum_{n=1}^{N_t}e^{(2\|V\|_{\infty}- a)n}. 
\end{align*}
On the other hand, again by $V\geq 0$ and estimate \eqref{skorokhod}, one has 
\begin{align*}
	\left|I_{1, 2}\right|&=\left|\E\left[\Xi_{t}^V\phi(H_t)\int_{2N_t}^tv(r)dW(r)\right]\right|\\
	&\leq \|\phi\|_{\infty}e^{\|V\|_{\infty}t}\E\left|\int_{2N_t}^tv(r)dW(r)\right|\leq Ce^{(\|V\|_{\infty}-a)n}. 
\end{align*}
Therefore we obtain 
\begin{align}\label{I1}
 \left|I_1\right|\leq \left|I_{1, 1}\right| + \left|I_{1, 2}\right|\leq C\|\phi\|_{\infty}\|P_{0, t, h}^V\id_{H}\|_{R_0}\mkm_{\eta}(w)\sum_{n=1}^{N_t}e^{(2\|V\|_{\infty}- a)n}. 
\end{align}
The proof of {\it Step 1} is complete. 

\vskip0.05in

{\it Step 2: Estimate on $I_2$}. The calculation of $I_2$ is similar to that of $I_1$. The difference is that now we will use \eqref{error-decay} and $\F_{r*}$ measurability of $\mathfrak{R}_{r, h}$, where $r^*=r$ if $\lceil r\rceil$, the smallest integer greater than or equal to $r$, is even, and $r^*=\lceil r\rceil$ if $\lceil r\rceil$ is odd. We first rewrite 
\begin{align*}
	I_2& = \E\left[\Xi_{t}^V\phi(H_t)\int_0^t\nabla V(H_r)\mathfrak{R}_{r, h} dr\right] \\
	&=  \E\left[\Xi_{t}^V\phi(H_t)\int_0^{N_t}\nabla V(H_r)\mathfrak{R}_{r, h} dr\right]+\E\left[\Xi_{t}^V\phi(H_t)\int_{N_t}^t\nabla V(H_r)\mathfrak{R}_{r, h} dr\right]\\
	&:= I_{2,1} + I_{2, 2}, 
\end{align*}
where $N_t$ is the largest even integer such that $N_t\leq t$. 
By the Markov property, we have 
\begin{align*}
	I_{2,1} &= \E\int_0^{N_t}\E\left[\Xi_{t}^V\phi(H_t)\nabla V(H_r)\mathfrak{R}_{r, h}|\F_{r^*}\right]dr\\
	& = \E\int_0^{N_t}\Xi_{r^*}^V\nabla V(H_r)\mathfrak{R}_{r, h}\P^V_{t- r^*}\phi(\H_{r^*})dr
\end{align*}
Similar to the arguments in deriving \eqref{I11-special}, by Proposition \ref{Growth-Condition}, $V\geq 0$ and translation identity \eqref{translation-identity}, and boundedness of $\phi$, we have
\begin{align*}
	\left|\P^V_{t- r^*}\phi(\H_{r^*})\right|\leq C\|\phi\|_{\infty}\mkm_q(\Phi_{0, r^*, h})\|P_{0, t-r^*, \b_{r^*}h}^V\id_{H}\|_{R_0}
	&= C\|\phi\|_{\infty}\mkm_q(\Phi_{0, r^*, h})\|P_{r^*, t, h}^V\id_{H}\|_{R_0}\\
	&\leq C\|\phi\|_{\infty}\mkm_q(\Phi_{0, r^*, h})\|P_{0, t, h}^V\id_{H}\|_{R_0}.
\end{align*}
Hence by estimate \eqref{contraction-polynomial}, H\"older's inequality and \eqref{error-decay}, one has 
\begin{align*}
	\left|I_{2,1}\right|
	&\leq C\|\phi\|_{\infty}\|\nabla V\|_{\infty}\|P_{0, t, h}^V\id_{H}\|_{R_0}\int_{0}^{N_t}e^{\|V\|_{\infty}r^*}\E\left[\|\mathfrak{R}_{r, h}\|\mkm_q(\Phi_{0, r^*, h})\right]dr\\
	&\leq C\|\phi\|_{\infty}\|P_{0, t, h}^V\id_{H}\|_{R_0}\int_{0}^{N_t}e^{\|V\|_{\infty}r^*}\left(\E\|\mathfrak{R}_{r, h}\|^2\right)^{\frac12}\left(\E\mkm_{2q}(\Phi_{0, r^*, h})\right)^{\frac12}dr\\
	&\leq C\|\phi\|_{\infty}\|P_{0, t, h}^V\id_{H}\|_{R_0}\exp(\eta\|w\|^2)\int_{0}^{N_t}e^{\|V\|_{\infty}r^*-ar}dr\\
	&\leq C\|\phi\|_{\infty}\|P_{0, t, h}^V\id_{H}\|_{R_0}\exp(\eta\|w\|^2)\int_{0}^{N_t}e^{(2\|V\|_{\infty}-a)r}dr,\\
\end{align*}
since $r^*\leq r+1$. 

\vskip0.05in

Now we look at $I_{2, 2}$. Noting that by \eqref{error-decay} and $\|P_{0, t, h}^V\id_{H}\|_{R_0}\geq 1$,  we have 
\begin{align*}
	|I_{2, 2}| 
	&\leq C\|\phi\|_{\infty}e^{\|V\|_{\infty}t}\int_{N_t}^t\E\|\mathfrak{R}_{r, h}\| dr\\
	&\leq C\|\phi\|_{\infty}\|P_{0, t, h}^V\id_{H}\|_{R_0}\exp(\eta\|w\|^2) e^{\|V\|_{\infty}t}\int_{N_t}^{\infty}e^{-ar}dr\\
	&\leq C\|\phi\|_{\infty}\|P_{0, t, h}^V\id_{H}\|_{R_0}\exp(\eta\|w\|^2) e^{(2\|V\|_{\infty}-a)t}
\end{align*}
Combining estimates for $I_{2,1}$ and $I_{2, 2}$ we obtain 
\begin{align}\label{I2}
	 \left|I_2\right|\leq C\|\phi\|_{\infty}\|P_{0, t, h}^V\id_{H}\|_{R_0}\exp(\eta\|w\|^2)\left(\int_{0}^{N_t}e^{(2\|V\|_{\infty}-a)r}dr + e^{(2\|V\|_{\infty}-a)t}\right). 
\end{align}

{\it Step 3: Estimate on $I_3$}. The calculation for $I_3$ follows directly from \eqref{error-decay}. Indeed, 
\begin{align}\label{I3}
\begin{split}
	\left|I_{3}\right|&=\left|\E\left[\Xi_{t}^V\nabla\phi(H_t)\mathfrak{R}_{t, h}\right]\right|\leq C\|\nabla \phi\|_{\infty}e^{\|V\|_{\infty}t}\E\|\mathfrak{R}_{t, h}\|\\
	&\leq C\|\nabla \phi\|_{\infty}\|P_{0, t, h}^V\id_{H}\|_{R_0}\exp(\eta\|w\|^2) e^{(2\|V\|_{\infty}-a)t}
\end{split}
\end{align}

\vskip0.05in

Now combining \eqref{gradient-set-up} and \eqref{I1}-\eqref{I3}, for any $c>0$, we can choose $a>0$ such that $2\|V\|_{\infty}-a = -c$, then we obtain the desired gradient estimate \eqref{prop-gradient}
\begin{align*}
	\|\nabla \P^V_t\phi(w, h)\|\leq C\exp(\eta\|w\|^2)\|P_{0, t, h}^V\id\|_{R_0}\left(\|\phi\|_{\infty}+e^{-ct}\|\nabla\phi\|_{\infty}\right), 
\end{align*}
where $C=C(c, \eta, \|V\|_{\infty}, \|\nabla V\|_{\infty})>0$ is independent of $h\in\To$.  

\vskip0.05in

The proof of Proposition \ref{gradient-estimate-modulo-growth-by-potential} is complete. 
\qed

\end{document}